\documentclass[11pt]{article}
\usepackage{amsmath, amssymb, amsthm}
\usepackage{booktabs}  
\usepackage{bm}
\usepackage{subcaption}  
\usepackage{chngcntr} 
\counterwithin{figure}{section}   
\counterwithin{table}{section}    
\counterwithin{equation}{section} 
\usepackage{authblk}

\usepackage{geometry}
\usepackage[colorlinks=true, citecolor=blue, linkcolor=red, urlcolor=blue]{hyperref}
\usepackage{algorithmic}
\usepackage{enumitem}
\newtheorem{proposition}{Proposition}[section]

\newtheorem{assumption}{Assumption}[section]
\newtheorem{theorem}{Theorem}[section]
\newtheorem{corollary}{Corollary}[section]
\newtheorem{lemma}{Lemma}[section]
\newtheorem{remark}{Remark}[section]
\theoremstyle{remark}
\usepackage{graphicx}   
\newcommand{\keywords}[1]{%
  \par\noindent\textbf{Keywords:} #1%
}
\usepackage{mathrsfs} 
\begin{document}
\title{Nonlocal Hamiltonian Dynamics on Sparse L\'evy Graphs: Spectral Analysis and Multimodal Sampling}
\author[a,b,f]{Miaolei Zheng} 
\author[c,d,e]{Ting Gao\thanks{Corresponding author: gao0716@hust.edu.cn}}
\author[b,f]{Jinqiao Duan} 

\affil[a]{School of Mathematics, Sun Yat-sen University, Guangzhou 510275, China}
\affil[b]{School of Sciences, Great Bay University, Dongguan 523000, China}
\affil[c]{School of Mathematics and Statistics, Huazhong University of Science and Technology, Wuhan 430074, China}
\affil[d]{Center for Mathematical Science, Huazhong University of Science and Technology, Wuhan 430074, China}
\affil[e]{Steklov-Wuhan Institute for Mathematical Exploration, Huazhong University of Science and Technology, Wuhan 430074, China}
\affil[f]{Guangdong Provincial Key Laboratory of Mathematical and Neural Dynamical Systems, Great Bay University, Dongguan 523000, China}
\date{}
\maketitle
\begin{abstract}
We develop a sparse graph method for transporting probability mass toward multimodal target distributions through damped nonlocal Hamiltonian dynamics. The formulation combines logarithmic-mean mobility with symmetric L{\'e}vy-type interaction weights, coupling the evolving density to an edge momentum field.  A graph constructed from nearest-neighbor connections and sampled long-range edges provides direct mass exchange between spatially separated regions. Once the graph is constructed, the density evolution is deterministic, and each update costs linear in the number of nodes and the long-range sampling budget. Linearization around the target distribution yields a damped oscillator governed by a weighted graph Laplacian. Its spectrum characterizes the interaction between nonlocal connectivity and inertia, with the L\'evy exponent $\alpha$ tuning the nonlocal connectivity: the spectral gap determines the optimal asymptotic damping, while the largest eigenvalue governs the time-step stability. Experiments on synthetic multimodal distributions demonstrate improved mode balance and more stable mode coverage relative to first-order and MCMC baselines. The resulting framework provides a sparse implementation of nonlocal inertial density transport for sampling problems with low-dimensional spatial structure. 

\par

\end{abstract}

\section{Introduction}
Sampling from multimodal distributions is an important task in Bayesian inference~\cite{Kruschke2010BayesianDA}, computational physics~\cite{Ying2023MultimodalSV}, and generative modeling~\cite{Tran2025DiffusionbasedSL}. In many practical applications, the target distribution exhibits multiple isolated modes separated by low-density regions. Standard Markov chain Monte Carlo (MCMC) methods~\cite{Hubin2016ModeJM} rely on local update proposals. When modes are separated by low-density regions, the method struggles to cross the intervening valleys and lingers around a single mode for extended periods.
 Even with sufficient iterations, the chain may still miss modes with significant probability mass under finite computational budgets. 
 Mode collapse limits Bayesian inference and generative diversity in
high-stakes applications such as drug discovery and Bayesian decision
making~\cite{He2025NoTN}.

A common strategy for sampling from complex distributions is to adopt gradient-based dynamics.
Gradient-based samplers arise as discretizations of the Wasserstein
gradient flow of the KL divergence.
Although they perform well on many problems, they still face difficulties on multimodal distributions where modes are separated by low-density regions.
Langevin-type methods~\cite{dalalyan2017theoretical,roberts1996exponential} such as MALA and ULA rely on local gradient updates that depend only on the current position. As a result, escaping a low-density region requires a waiting time that grows exponentially with the barrier height.
Stein variational gradient descent (SVGD)~\cite{wang2019stein,3294996.3295071} introduces kernel-based interactions among particles. However, its bandwidth—whether set by median heuristic or cross-validation—fundamentally restricts the interaction range. When the separation between modes exceeds this scale, distant particles become effectively decoupled and the ensemble collapses to a single mode. In addition, the underdamped Langevin dynamics (ULD) sampler~\cite{cheng2018underdampedlangevinmcmcnonasymptotic} still needs continuous paths to cross low-density regions. Recently, Zhou et al.~\cite{zhou2025accelerated} introduced Nesterov-type
acceleration for sampling on discrete graphs, highlighting the benefit of
inertial transport for accelerating mass redistribution. However, the
	resulting exploration is still constrained by the connectivity of the
	underlying graph, especially when distant modes are weakly connected.

Direct communication between distant locations provides a way to break this barrier. L{\'e}vy processes and fractional operators~\cite{applebaum2009levy, duan2015Levy}
offer exactly this kind of direct communication, permitting jumps over
arbitrarily large distances.
Many sampling methods have exploited this mechanism. Ye and Zhu \cite{HMC_levy} combined L{\'e}vy jumps with Hamiltonian dynamics to improve multimodal exploration. Fractional Langevin Monte 
Carlo~\cite{Simsekli2017FractionalLM} replaces Gaussian noise with heavy-tailed L{\'e}vy 
noise, enabling large jumps that accelerate multimodal exploration. This idea has been extended to underdamped settings~\cite{HMC_levy} and combinatorial optimization~\cite{wang2025neurips-fractional}. In these settings, L{\'e}vy flights escape local optima in polynomial time, whereas Gaussian-driven dynamics require exponential time. More recently, De Marco et al.~\cite{demarco2026schrodingerbridgesjumpstime} also incorporated L{\'e}vy-type jumps into the Schrödinger bridge framework for time series generation. 
Motivated by these developments, we combine L{\'e}vy-type long-range
	interactions with accelerated graph-based transport, leading to a
	nonlocal Hamiltonian sampling framework on sparse graphs.

In this paper, we propose a density-level nonlocal Hamiltonian sampler driven by L\'evy jump kernels, formulated on the nonlocal Wasserstein metric. Our work makes three contributions. First, we develop a
second-order density transport on sparse graphs, combining local
nearest-neighbor connections with L\'evy-type long-range edges. This construction gives, once the graph has been built, a per-step cost reduction from \(O(N^2)\) to \(O(N(k_{\mathrm{local}}+m))\), where \(N\) is the number of nodes, \(k_{\mathrm{local}}\) is the local KNN parameter, and \(m\) is the long-range sampling budget per node. Second, we characterize the linearized
relaxation through the weighted graph spectrum. In particular, we derive
the optimal damping from the spectral gap and separate the
continuous-time relaxation rate from the temporal and
graph-discretization errors. Third, numerical experiments on multimodal
targets show improved mode balance relative to the tested first-order,
MCMC, MALA, and ULD baselines. The method evolves an explicit density on
a fixed spatial grid and is evaluated in low dimensions, including
two-dimensional PCA representations of MNIST and dihedral-angle
densities of alanine dipeptide.

This paper is organized as follows. In Section 2, we derive the damped nonlocal Hamiltonian system and present its graph discretization. Section 3 studies the linearized relaxation and numerical approximation. In Section 4, we compare the proposed method against baseline methods on synthetic benchmarks with known target densities and on real-world datasets. Section 5 concludes the paper with a summary of our findings and a discussion of future directions.

\section{Nonlocal Wasserstein Geometry and Hamiltonian Dynamics}\label{sec:NLW}

We first introduce the nonlocal Wasserstein geometry and then use it to formulate the damped Hamiltonian dynamics for the
density--momentum pair \((\rho,\Phi)\). 
The formulation combines the nonlocal transport geometry of \cite{erbar2014gradient,esposito2019nonlocal} with the Hamiltonian structure of optimal transport \cite{wang2022accelerated}.

\subsection{ Preliminaries}
Let \(\Omega \subset \mathbb{R}^d\) be equipped with the Lebesgue
measure \(\mathrm{m}\), and denote by \(\mathcal{P}(\Omega)\)
the space of Borel probability measures. The energy is the Kullback–Leibler divergence
\begin{equation}
E(\rho) = D_{KL}(\rho \| \pi) = \int_{\Omega} \rho(x) \log \frac{\rho(x)}{\pi(x)}\mathrm{d}x\,,
\label{eq:KL}
\end{equation}
where \(\pi \propto e^{-f}\) is the target distribution.

Let \(J\) be a symmetric jump kernel in the edge space
\(\mathrm{G}=\{(x,y)\in\mathbb{R}^d\times\mathbb{R}^d: x\neq y\}\) with respect to
the reference measure \(\mathrm{m}\), satisfying
\[J(x,\mathrm{d}y)\mathrm{m}(\mathrm{d}x)
=J(y,\mathrm{d}x)\mathrm{m}(\mathrm{d}y),
    \quad x\neq y,\]
and let \(J\mathrm{m}\) denote the associated product measure
\(J\mathrm{m}(\mathrm{d}x,\mathrm{d}y) = J(x,\mathrm{d}y)\mathrm{m}(\mathrm{d}x).
\)
Let \(\theta: \mathbb{R}_+ \times \mathbb{R}_+ \to \mathbb{R}_+\) be an interpolation function satisfying the assumptions in \cite{erbar2014gradient}, 
e.g., the logarithmic mean \(\theta(s,t) = \frac{s-t}{\log s - \log t}\). The edge density is defined as \(\hat{\rho}(x,y) = \theta(\rho(x),\rho(y))\).

For a test function \(\phi:\mathbb{R}^d\to\mathbb{R}\), the nonlocal gradient is defined on pairs \((x,y)\in\mathbb{R}^d\times\mathbb{R}^d\) as
\begin{equation}\label{eq:nonlocal gradient}
\overline{\nabla}\phi(x,y) = \phi(y) - \phi(x).
\end{equation}

The evolution of a probability density \(\rho_t\) subject to nonlocal transport is governed by the continuity equation
\begin{equation}
\partial_t \rho_t + \overline{\nabla}\cdot (\hat{\rho}_t \, \overline{\nabla}\phi_t) = 0,
\label{eq:nonlocal_ce}
\end{equation}
where \(\hat{\rho}_t(x,y) = \theta(\rho_t(x),\rho_t(y))\) and \(\phi_t\) is a potential function.  
In Erbar's general framework, the action functional is defined for pairs \((\mu,\nu)\) of measures on \(\mathbb{R}^d\) and \(\mathrm{G}\), respectively. In our setting, where \(\mu=\rho \mathrm{m}\) is absolutely continuous and the flux takes the gradient form \(\nu=\hat{\rho}\,\overline{\nabla}\phi\,J\mathrm{m}\), this action functional reduces to
\begin{equation*}\label{eq:action function}
\mathcal{A}(\rho,\phi)=\frac12\int_G |\overline{\nabla}\phi(x,y)|^2\,\hat{\rho}(x,y)\,J \mathrm{m}(\mathrm{d}x,\mathrm{d}y),
\end{equation*}
which matches the simplified form established in \cite[Lemma 2.3 and Proposition 5.14]{erbar2014gradient}.

The nonlocal Wasserstein distance \(W_{\mathrm{NL}}\) is defined as
\begin{equation}\label{eq:non-local W metric}
W_{\mathrm{NL}}^2(\rho_0,\rho_1) = \inf_{(\rho,\phi)\in \mathcal{CE}_1(\rho_0,\rho_1)} \int_0^1 \mathcal{A}(\rho_t,\phi_t) \, \mathrm{d}t,
\end{equation}
where \(\mathcal{CE}_1(\rho_0,\rho_1)\) denotes the set of pairs \((\rho,\phi)\) satisfying the continuity equation (\ref{eq:nonlocal_ce}) with endpoints \(\rho(0)=\rho_0,\rho(1)=\rho_1\), along with the integrability condition \(\int_0^1\mathcal{A}(\rho_t,\phi_t)\,\mathrm{d}t<\infty\).

\subsection{Hamiltonian formulation and damped dynamics}
Following the nonlocal Wasserstein geometry of Erbar~\cite{erbar2014gradient}, the tangent space $T_\rho\mathcal{P}$ at $\rho$ is identified with fluxes on the edge space $\mathrm{G}$, while the cotangent space $T_\rho^*\mathcal{P}$ is identified with functions $\Phi:\mathbb{R}^d\to\mathbb{R}$. Under the assumptions in Erbar’s framework, the metric tensor
 \(G_\rho^{-1}\) defines the inverse metric on the zero-mass tangent space.
\begin{equation}
G_\rho^{-1}: T_\rho^*\mathcal{P} \to T_\rho\mathcal{P},
\qquad
G_\rho^{-1}[\Phi](x) = \int_{\Omega} \bigl( \Phi(x) - \Phi(y) \bigr) \, \hat{\rho}(x,y) \, J(x,\mathrm{d}y).
\label{eq:inversemetric}
\end{equation}
The metric inner product on $T_\rho\mathcal{P}$ is then defined as
\begin{equation}
g_\rho(\sigma_1,\sigma_2) = \int_{\Omega} \sigma_1(x)\,\Phi_2(x)\,\mathrm{m}(\mathrm{d}x),
\qquad \sigma_i = G_\rho^{-1}[\Phi_i], \ i=1,2,
\label{eq:innerproduct}
\end{equation}
or equivalently,
\(g_\rho(\sigma_1,\sigma_2) = \frac12\int_G \overline{\nabla}\Phi_1\,\overline{\nabla}\Phi_2\,\hat{\rho}\,J\mathrm{m}\).

The Hamiltonian is defined as the sum of kinetic and potential energy:
\begin{equation}
\mathcal{H}(\rho,\Phi) = \frac12 \, g_\rho(\sigma,\sigma) + E(\rho), \qquad \sigma = G_\rho^{-1}[\Phi].
\label{eq:Hamilton_def}
\end{equation}
Substituting (\ref{eq:inversemetric}) into (\ref{eq:Hamilton_def}) and utilizing symmetry yields the explicit symmetric form:
\begin{equation}
\mathcal{H}(\rho,\Phi) = \frac14 \iint \bigl( \Phi(x) - \Phi(y) \bigr)^2 \, \hat{\rho}(x,y) \, J(x,\mathrm{d}y)\mathrm{m}(\mathrm{d}x) + E(\rho).
\label{eq:Hamilton_sym}
\end{equation}

The undamped Hamiltonian equations are $\partial_t \rho = \frac{\delta \mathcal{H}}{\delta \Phi}$ and $\partial_t \Phi = -\frac{\delta \mathcal{H}}{\delta \rho}$. Combined with (\ref{eq:Hamilton_sym}), this yields the following explicit system.
\begin{proposition}[Hamiltonian dynamics]
\label{prop:hamiltonian}
The nonlocal Hamiltonian system for $(\rho,\Phi)$ is described by the following equations:
\begin{align}
\partial_t \rho_t(x) &= \int \bigl( \Phi_t(x) - \Phi_t(y) \bigr) \, \hat{\rho}_t(x,y) \, J(x,\mathrm{d}y), \label{eq:state}\\
\partial_t \Phi_t(x) &= -\frac12 \int \bigl( \Phi_t(x) - \Phi_t(y) \bigr)^2 \, \partial_1 \theta(\rho_t(x),\rho_t(y)) \, J(x,\mathrm{d}y) - \frac{\delta E}{\delta \rho_t}(x), \label{eq:momentum}
\end{align}
where $\partial_1\theta$ denotes the partial derivative with respect to the first argument, and $\frac{\delta E}{\delta\rho}$ is the first variation of (\ref{eq:KL}):
\begin{equation}
\frac{\delta E}{\delta \rho}(x) = \log\frac{\rho(x)}{\pi(x)} + 1.
\label{eq:variation_E}
\end{equation}
\end{proposition}
\begin{proof}
    See Appendix~\ref{app:ham}.
\end{proof}

The undamped system conserves the Hamiltonian and therefore does not
relax to equilibrium. 
To introduce dissipation while preserving the Hamiltonian structure, we add a linear damping term \(\gamma_t \Phi_t\) to the momentum equation (\ref{eq:momentum}). This yields the following damped nonlocal Hamiltonian system.

\begin{proposition}
The damped nonlocal Hamiltonian system takes the explicit form
\begin{equation}
\begin{cases}\label{eq:NLW-AIG}
\partial_t \rho_t(x) = \displaystyle\int \bigl( \Phi_t(x) - \Phi_t(y) \bigr) \, \hat{\rho}_t(x,y) \, J(x,\mathrm{d}y), \\[12pt]
\partial_t \Phi_t(x) + \gamma_t \Phi_t(x) = -\dfrac12 \displaystyle\int \bigl( \Phi_t(x) - \Phi_t(y) \bigr)^2 \, \partial_1\theta(\rho_t(x),\rho_t(y)) \, J(x,\mathrm{d}y) - \dfrac{\delta E}{\delta \rho_t}(x),
\end{cases}
\end{equation}
where $\hat{\rho}_t(x,y)=\theta(\rho_t(x),\rho_t(y))$, $\partial_1\theta$ denotes the partial derivative with respect to the first argument, \(J(x,\mathrm{d}y)\) is the jump measure from \(x\), and $\frac{\delta E}{\delta\rho}$ is the first variation of the energy.
\end{proposition}
We refer to the damped nonlocal Hamiltonian system \eqref{eq:NLW-AIG} as the
nonlocal Wasserstein Hamiltonian flow. In the overdamped limit
\(\gamma\to\infty\), it reduces to the first-order nonlocal Wasserstein
gradient flow. 

\subsection{Graph discretization}\label{sec:graph_discretization}
Our graph construction is related to asymptotically compatible
discretizations of nonlocal operators~\cite{Tian2016AsymptoticallyCS},
where the interaction radius and the grid size are jointly refined. For a systematic treatment of nonlocal modeling and the numerical
discretization of nonlocal operators, we refer to~\cite{3360184}.
We now discretize the continuous dynamics on a graph \(\mathcal{G}=(\mathcal{V},\mathcal{E})\), where \(\mathcal{V}=\{1,\ldots,N\}\) is the node set and \(\mathcal{E}\) is the undirected edge set with symmetric weights \(\nu_{ij}=\nu_{ji}\ge 0\). For each node \(i\in\mathcal{V}\), 
let \(\rho_i(t)>0\) be the probability mass at node $i$ satisfying \(\sum_i\rho_i=1\), and
\(\rho=(\rho_1,\ldots,\rho_N)^\top\in\mathbb{R}^N\).
The target probability mass is \(\pi_i>0\) on \(\mathcal{V}\), and the free energy is
\[
E(\rho)=\sum_i\rho_i\log\frac{\rho_i}{\pi_i},\qquad 
G_i:=\frac{\delta E}{\delta\rho_i}=\log\frac{\rho_i}{\pi_i}+1. \label{eq:free_energy}
\]

To express the dynamics in terms of edge variables, we introduce \(u_{ij}:=\Phi_j-\Phi_i\) on each edge. This leads to the following graph-based formulation.

\begin{proposition}
\label{prop:velocity_dynamics}
Let \((\rho_i, \Phi_i)\) be a solution of the damped nonlocal Hamiltonian system (\ref{eq:NLW-AIG}). 
Then the coupled system for \((\rho_i, u_{ij})\) evolves as
\begin{equation}
\left\{
\begin{aligned}
\partial_t\rho_i &= -\sum_j \theta_{ij}\,u_{ij}\,\nu_{ij},\\[6pt]
\partial_t u_{ij} &= (G_i-G_j)-\gamma_t u_{ij}+\mathcal{C}_{ij}(\rho,u),
\end{aligned}
\right. \label{eq:velocity_dynamics}
\end{equation}
where the nonlocal convection term is defined as
\begin{equation}
\label{eq:convection}
\mathcal{C}_{ij}(\rho,u):=\frac12\left(\sum_k\partial_1\theta_{ik}\,u_{ik}^2\,\nu_{ik}
-\sum_k\partial_1\theta_{jk}\,u_{jk}^2\,\nu_{jk}\right).
\end{equation}
\end{proposition}

The antisymmetric edge flux is \(m_{ij}:=\theta_{ij}u_{ij}\nu_{ij}\), and \(u_{ij}=\Phi_j-\Phi_i\) is the edge potential difference. The first equation enforces mass conservation, while the second governs the edge dynamics under damping and nonlocal convection.

The velocity formulation (\ref{eq:velocity_dynamics}) can be viewed as a discretization of the continuous PDE on a Riemannian manifold \(\mathcal{M}\):
$\partial_t\rho=-\overline{\nabla}\cdot(\rho u), 
\partial_t u=-\overline{\nabla}\log\frac{\rho}{\pi}-\gamma u- (u \cdot\overline{\nabla} )u$.

\begin{remark} In the overdamped limit $\gamma\to\infty$, the velocity 
equation is dominated by the damping term $-\gamma u_{ij}$. Balancing 
with the driving force $(G_i-G_j)$ gives $u_{ij} = \gamma^{-1}(G_i-G_j) + O(\gamma^{-2})$. 
Substituting into the density equation and rescaling time by 
$\tilde t = \frac t \gamma$ yields the first-order nonlocal Wasserstein 
gradient flow:
\begin{equation}\label{eQ:first_NLW}
\partial_{\tilde t}\rho_i = -\sum_j \theta_{ij}(G_i-G_j)\nu_{ij}. 
\end{equation}
\end{remark}

The proposed damped nonlocal Hamiltonian system is the natural extension of the accelerated information gradient flow (AIG) \cite{wang2022accelerated} to the nonlocal Wasserstein geometry. A closely related acceleration mechanism also appears in the Stein accelerated flow \cite{stein2025accelerated}, where the interaction is mediated by a smooth kernel \(K\).

\section{Convergence analysis}
\label{sec:convergence}
The convergence of sampling algorithms in non-convex settings has been extensively studied in recent years. Non-asymptotic bounds for Langevin-type methods under non-convex or weakly convex conditions are now available in a series of works \cite{pang2025projected,li2022sharp, altschuler2026shifted}.
In this section, we analyze the method through continuous-time mixing, temporal discretization, and sparse-graph approximation.

\subsection{Continuous-time mixing}
\label{subsec:mixing}

We first study the long-time behavior of the continuous-time system 
(\ref{eq:velocity_dynamics}). The kinetic energy, relative entropy, and total Hamiltonian are respectively defined as
\begin{align}
\mathcal{K} = \frac14\sum_{i,j} m_{ij}u_{ij}
=\frac12\sum_{\{i,j\}\in\mathcal E}\theta_{ij}u_{ij}^2\nu_{ij}, \quad
{E}(\rho) = \sum_i \rho_i \log\frac{\rho_i}{\pi_i}, \quad
{H}= {E}(\rho) + \mathcal{K}.
\end{align}
Here $\mathcal E$ denotes the realized undirected sparse edge set.
A direct computation using (\ref{eq:velocity_dynamics}) yields the 
energy–dissipation identity
\begin{equation}
\frac{d}{dt}{H} =-2\gamma\,\mathcal{K}\le 0,
\label{eq:ham_diss}
\end{equation}
thus ${H}$ is non-increasing, where $\gamma >0 $ is a constant. Recall the measure \(J\mathrm{m}\) on \(\mathrm{G}\) defined in Section~2. For a probability density \(\rho\), we define a discrete analog of the nonlocal Fisher information
\[
\mathcal{I}(\rho\|\pi) := \frac12 \sum_{i,j} \theta_{ij}(\rho)(G_i-G_j)^2 \nu_{ij},
\]
which follows the Erbar--Maas framework.
The energy-dissipation identity (\ref{eq:ham_diss}) establishes monotonicity of \(H\) along the dynamics, but does not quantify the rate of convergence. To this end, we perform a linearized analysis of the velocity dynamics (\ref{eq:velocity_dynamics}) in a neighborhood of the equilibrium \((\pi,0)\).

\subsubsection{Linearization around equilibrium}
\label{sec:linearization}
The velocity dynamics (\ref{eq:velocity_dynamics}) admits the equilibrium \((\rho,u)=(\pi,0)\).
On the discrete graph \(\mathcal{G}=(\mathcal{V},\mathcal{E})\), we consider a neighborhood
\[
    \mathcal U_\delta
    :=
    \left\{ \rho\in\mathbb{R}^{N}: \rho_i=\pi_i(1+\varphi_i), 
\sum_{i=1}^{N}\rho_i=1, \|\varphi\|_{\ell^2(\pi)}<\delta
    \right\},
\]
where \(
\|\varphi\|_{\ell^2(\pi)}^2
=\sum_i \pi_i\varphi_i^2.
\)
For $\rho \in \mathcal U_\delta$, we introduce the relative perturbation $\varphi_i>-1$ defined by
\(
\rho_i = \pi_i(1+\varphi_i),\quad i=1,\ldots,N,
\)
with the zero-mean condition
\(
\sum_{i=1}^N \pi_i \varphi_i = 0,
\)
where $\|\varphi\|_{\ell^2(\pi)} \ll 1$ (equivalently, $\|\varphi\|_{\ell^2(\pi)} \lesssim \delta$). 

We also use the equivalent norm
\(
    \|\rho-\pi\|_{\pi^{-1}}^2
    :=
    \sum_i\frac{(\rho_i-\pi_i)^2}{\pi_i}
    =
    \|\varphi\|_{\ell^2(\pi)}^2,
\)
where the equality follows from \(\rho_i-\pi_i=\pi_i\varphi_i\).

As shown in Appendix~\ref{app:local_convergence}, the first-order expansions read
\[
G_i-G_j=\varphi_i-\varphi_j+O(\varphi^2),\quad
\theta_{ij}(\rho)=\theta_{ij}^0+O(\varphi),\quad
\mathcal C_{ij}(\rho,u)=O(\|u\|^2), 
\]
where $\theta_{ij}^0:=\theta(\pi_i,\pi_j)$ denotes the logarithmic mean evaluated at the equilibrium densities.
Retaining only first-order terms, we obtain the following linearized system
\begin{equation}\label{eq:first_order_linearization}
\left\{
\begin{aligned}
\partial_t\varphi_i &= -\frac{1}{\pi_i}\sum_j\theta_{ij}^0 u_{ij}\nu_{ij},\\
\partial_t u_{ij} &= (\varphi_i-\varphi_j)-\gamma u_{ij}.
\end{aligned}
\right.
\end{equation}
Eliminating $u$ (see Appendix~\ref{app:local_convergence}) yields the second-order equation
\(
    \partial_{tt}\varphi+\gamma\partial_t\varphi+A\varphi=0,
\)
where \(A\) has the explicit form
\begin{equation}\label{eq:pi_weighted_nonlocal_graph_A}
    (A\varphi)_i
    =
    \frac{1}{\pi_i}\sum_j\theta_{ij}^0(\varphi_i-\varphi_j)\nu_{ij}.
\end{equation}
The operator $A$ defined in (\ref{eq:pi_weighted_nonlocal_graph_A}) is self-adjoint in the weighted inner product
\(
\langle \phi, \psi \rangle_\pi := \sum_i \pi_i \phi_i \psi_i,
\)
and its spectrum consists of
$0=\lambda_0<\lambda_1\le\cdots\le\lambda_{\max}$.
The spectral gap $\lambda_1>0$ controls the local convergence rate.
The linearized operator \(A\) is related to the nonlocal Fisher information
\begin{equation}
\begin{aligned}\label{eq:fisher_expansion}
\mathcal I(\rho\|\pi)
= \frac12 \sum_{i,j} \theta_{ij}^0 (\varphi_i-\varphi_j)^2 \nu_{ij}
+ O(\|\varphi\|^3)
= \langle A\varphi, \varphi\rangle_\pi + O(\|\varphi\|_\pi^3),
\end{aligned}
\end{equation}
see Appendix~\ref{app:local_convergence} for the proof of self-adjointness and the deduction of Fisher information.

The linearized system has the same algebraic structure as the classical
damped Heavy Ball system~\cite{Chen2023AcceleratingOO}, but with the nonlocal graph Laplacian as the
stiffness operator. We present the modal analysis of this second-order system in the next section, which also yields the optimal damping and the
local convergence rates determined by the nonlocal graph spectrum.

We now introduce the spectral gap condition.

\begin{assumption}[Local spectral gap]
\label{ass:local_spectral_gap}
Let $A$ be the linearized nonlocal graph Laplacian defined in
(\ref{eq:pi_weighted_nonlocal_graph_A}). Assume that there exist constants
$\lambda_1>0$ and $\delta>0$ such that, for every perturbation
$\varphi$ satisfying
\[
\rho_i=\pi_i(1+\varphi_i)>0,\qquad
\|\varphi\|_{\ell^2(\pi)}\le \delta,
\qquad
\sum_{i\in V}\pi_i\varphi_i=0,
\]
one has
\(
\langle A\varphi,\varphi\rangle_\pi
\ge
\lambda_1\|\varphi\|_{\ell^2(\pi)}^2.
\)
\end{assumption}

\begin{remark}
Under the relative perturbation \(\rho_i=\pi_i(1+\varphi_i)\), the relative entropy satisfies
\begin{equation}\label{eq:entropy_linear}
E(\rho) = \frac12 \|\varphi\|_\pi^2 + O(\|\varphi\|_\pi^3),
\end{equation}
and the Fisher information admits the expansion~(\ref{eq:fisher_expansion}).
Thus, Assumption~\ref{ass:local_spectral_gap} provides a lower bound
for the quadratic term in the Fisher information expansion
near the equilibrium $\pi$.

\end{remark}

For the second-order dynamics, we use the phase-space Lyapunov functional
\begin{equation}
    \mathcal L(\rho,u)
    :=
    E(\rho)
    +
    K(\rho,u),
    \qquad
    K(\rho,u)
    :=
    \frac14
    \sum_{i,j}
    \theta_{ij}u_{ij}^2\nu_{ij}.
    \label{eq:phase_space_lyapunov}
\end{equation}
At the equilibrium \((\pi,0)\), we have
\(
    \mathcal L(\pi,0)=0.
\) Our energy argument follows the Lyapunov-based approach used for preconditioned accelerated gradient methods \cite{2021Preconditioned}.

For the local nonlinear analysis, define the phase-space norm
\[
\|(\varphi,u)\|_{\mathcal X}^2
:=
\|\varphi\|_{\ell^2(\pi)}^2
+
\|u\|_{\mathcal U}^2,
\quad
\|u\|_{\mathcal U}^2
:=
\frac12\sum_{i,j}\theta_{ij}^0u_{ij}^2\nu_{ij},
\]
and let
\(
\mathcal V_\delta
:=
\left\{
(\rho,u):
\rho_i=\pi_i(1+\varphi_i),\
\sum_i\pi_i\varphi_i=0,\
\|(\varphi,u)\|_{\mathcal X}<\delta
\right\}.
\)
\begin{assumption}[Local regularity]
	\label{ass:local_regular}
	On the bounded state region considered in the local analysis, there exist
	constants
	\[
	0<\rho_{\min}\le \rho_i\le \rho_{\max},\qquad
	0<\pi_{\min}\le \pi_i,
	\qquad
	\|u\|_{\infty}\le U_{\max},
	\]
	such that the logarithmic-mean mobility and its derivatives satisfy
	\[
	0<c_\theta\le \theta_{ij}(\rho)\le C_\theta,\qquad
	|\partial_1\theta_{ij}(\rho)|\le C_{\partial\theta},
	\qquad
	|\partial_{ab}^2\theta_{ij}(\rho)|\le C_{\theta,2}.
	\]
\end{assumption}

\begin{theorem}[Local exponential convergence]
	\label{thm:local_exp}
	Under Assumptions~\ref{ass:local_spectral_gap}
	and~\ref{ass:local_regular}, there exist
	\(\delta>0\) and \(\delta_0\in(0,\delta]\) such that,
	for every initial condition
	\(
	(\rho(0),u(0))\in\mathcal V_{\delta_0},
	\)
	the corresponding nonlinear solution exists globally, remains in
	\(\mathcal V_\delta\), and the phase-space Lyapunov functional satisfies
	\[
	\mathcal L(\rho(t),u(t))
	\le
	C_{\mathrm{ent}}(1+t)^2
	e^{-\lambda_c t}
	\mathcal L(\rho(0),u(0)),
	\]
	where \(\lambda_c=2r_1(\gamma)>0\).
	If \(\gamma\neq2\sqrt{\lambda_1}\), the polynomial factor can be omitted.
\end{theorem}

\begin{proof}
	See Appendix~\ref{app:local_convergence}.
\end{proof}
Related spectral-gap-based acceleration and optimal damping results
for $\chi^2$-driven graph dynamics were established by
Zhou et al.~\cite{zhou2025accelerated}.
For the present nonlinear logarithmic-mean dynamics, the linearization
around equilibrium yields the following damping and rate comparison.
\begin{corollary}[Optimal damping and comparison with first-order dynamics]
	\label{cor:acceleration}
	For the settling time
	\[
	T_{\mathrm{mix}}(\varepsilon)
	:=
	\inf\{t\ge0:R(s)\le\varepsilon \text{ for all }s\ge t\},
	\qquad
	R(t):=
	\frac{\|\rho_t-\pi\|_{\pi^{-1}}}
	{\|\rho_0-\pi\|_{\pi^{-1}}},
	\]
	the slowest-mode decay rate is maximized at
	\(
	\gamma^*=2\sqrt{\lambda_1},
	\qquad
	r_1(\gamma^*)=\sqrt{\lambda_1}.
	\)
	For a fixed graph and fixed initial data in the local regime, as
	\(\varepsilon\downarrow0\),
	\[
	T_{\mathrm{mix}}^{(2)}
	=
	O\!\left(\lambda_1^{-1/2}\log\frac1\varepsilon\right),
	\qquad
	T_{\mathrm{mix}}^{(1)}
	=
	O\!\left(\lambda_1^{-1}\log\frac1\varepsilon\right).
	\]
	Hence the second-order dynamics has the leading-order asymptotic
	advantage in the small-gap regime \(0<\lambda_1<1\).
\begin{proof}
	See Appendix~\ref{app:local_convergence}.
	\end{proof}

\end{corollary}

\begin{remark}
The corresponding first-order linearization is \(\dot{\varphi} + A \varphi = 0, \) 
whose slowest mode decays as \(
e^{-\lambda_1 t}\).
\end{remark}
A numerical illustration of this spectral picture is provided in
Section~\ref{subsubsec:mixing_time_verification}, where we compute the spectral
gap and the resulting linearized modal mixing-time predictions
for several values of \(\alpha\).

\begin{lemma}[Local equivalence]
\label{lem:wnl_l2}
Let \(A=BC\) be the linearized nonlocal graph Laplacian defined in
\eqref{eq:pi_weighted_nonlocal_graph_A}. For every $0<\eta<1$, after shrinking the neighborhood of $\pi$ if necessary,
\begin{equation}\label{ieq:W_l2_equiv}
\frac{1-\eta}{\sqrt{\lambda_{\max}}}
\|\rho-\pi\|_{\pi^{-1}}\le
W_{\mathrm{NL}}(\rho,\pi)\le \frac{1+\eta}{\sqrt{\lambda_1}} \|\rho-\pi\|_{\pi^{-1}}.\end{equation}
The spectral constants are exact for the tangent metric at equilibrium; the factors $1\pm\eta$ account for the variation of the mobility along a finite path,
where \(
0=\lambda_0<\lambda_1\le\cdots\le\lambda_{\max}<\infty\) are the eigenvalues of the operator $A$ restricted to the
zero-mass subspace.
\begin{proof}
    See Appendix~\ref{app:local_convergence}.
\end{proof}
\end{lemma}

\begin{corollary}[Wasserstein convergence rates]
\label{cor:WNL}
Under the assumptions of Theorem~\ref{thm:local_exp}, with zero initial
momentum $u(0)=0$, the following
convergence rate holds for the nonlocal Wasserstein distance
\(W_{\mathrm{NL}}\):
\[
W_{\mathrm{NL}}(\rho(t),\pi)
\le C_W (1+t)e^{-\lambda_c t/2}\,W_{\mathrm{NL}}(\rho(0),\pi),
\qquad t\ge0.
\]
where \(\lambda_c=2r_1(\gamma) > 0\) is the energy decay rate from
Theorem~\ref{thm:local_exp} and \(
    C_W
    =\frac{1+\eta}{1-\eta}
    \sqrt{\frac{\lambda_{\max}}{\lambda_1}}
    \sqrt{\frac{C_{\mathrm{ent}}c_{12}}{c_{11}}}>0.
\)
\end{corollary}
\begin{proof}
	See Appendix~\ref{app:local_convergence}.
\end{proof}

\subsection{Temporal discretization error}
\label{sec:time_discretization_error}
In this section, we analyze the temporal discretization error of the semi-implicit Euler scheme (\ref{eq:num_scheme}). Recall that the graph consists of \(N\) nodes indexed by \(i,j=1,\ldots,N\), with edge weights \(\nu_{ij}\). We write \(\rho_i^k\) and \(u_{ij}^k\) for the numerical approximations at node \(i\) and edge \((i,j)\), respectively, at time \(t_k \).
\begin{equation}
\label{eq:num_scheme}
\begin{cases}
\dfrac{u_{ij}^{k+1} - u_{ij}^k}{\Delta t}
= (G_i^k - G_j^k) + \mathcal{C}_{ij}(\rho^k,u^k) - \gamma_k u_{ij}^{k+1},
\\[1.2ex]
\dfrac{\rho_i^{k+1} - \rho_i^k}{\Delta t}
= -\displaystyle\sum_j \theta_{ij}(\rho^k) u_{ij}^{k+1} \nu_{ij},
\end{cases}
\end{equation}
where $\gamma_k = \gamma(t_k)$ is the damping parameter evaluated at
$t_k = k\Delta t$. Thus the velocity is updated first and the new
velocity is used in the density update.

\begin{lemma}[Lipschitz estimates]
\label{lem:lipschitz}
Assume the dynamics are confined to a forward-invariant region
\[
\mathcal D := \left\{(\rho,u): \rho_i \ge \rho_{\min}~\forall i,\; \|u\|_\infty \le U_{\max},\; \sum_i \rho_i = 1\right\},
\]
where $\rho_{\min}>0$ and $U_{\max}<\infty$ are fixed constants. Under the assumption~\ref{ass:local_regular}, there exist constants
 $L_\theta, L_G, L_{\mathcal C,\rho},L_{\mathcal C,u} > 0$, depending only on $\rho_{\min}$, $U_{\max}$, and the kernel bounds,
such that for all $(\rho^1,u^1),(\rho^2,u^2)\in\mathcal D$,
\begin{align}
 \left\|\sum_j \bigl(\theta_{ij}(\rho^1)u_{ij}^1 - \theta_{ij}(\rho^2)u_{ij}^2\bigr)\nu_{ij}\right\|_{\ell^2}
&\le L_\theta \bigl( \|e_\rho\|_{\ell^2} + \|e_u\|_{\ell^2} \bigr), \tag{L1}\\
 \|G(\rho^1) - G(\rho^2)\|_{\ell^2}
&\le L_G \|e_\rho\|_{\ell^2}, \tag{L2}\\
 \|\mathcal C(\rho^1,u^1) - \mathcal C(\rho^2,u^2)\|_{\ell^2}
 &\le L_{\mathcal C,\rho}\|e_\rho\|_{\ell^2}+L_{\mathcal C,u}\|e_u\|_{\ell^2}, \tag{L3}
\end{align}
where $e_\rho:=\rho^1-\rho^2$, $e_u:=u^1-u^2$.
\end{lemma}
\begin{proof}
    See Appendix~\ref{app:temporal discretization error}.
\end{proof}
\begin{lemma}[Temporal error bound]
\label{lem:temporal_error}
Assume that the dynamics satisfy the Lipschitz estimates in
Lemma~\ref{lem:lipschitz} and that \(\gamma_k=\gamma(t_k)\), where
\(\gamma:[0,T]\to(0,\infty)\) is Lipschitz continuous and satisfies
\[
0<\gamma_{\min}\le\gamma(t)\le\gamma_{\max}<\infty,
\qquad t\in[0,T].
\] Then there exists a
constant $C_{\mathrm{temp}} > 0$ (independent of $\Delta t$ and $k$) such that for all
$k$ with $t_k = k\Delta t \le T$,
\[
\|e_\rho^k\|_{\ell^2}^2 + \|e_u^k\|_{\ell^2}^2 \le C_{\mathrm{temp}} (\Delta t)^2 e^{C T},
\]
where
\(
    e_\rho^k=\rho^k-\rho(t_k),
    \quad
    e_u^k=u^k-u(t_k)
\)
are the global discretization errors at time \(t_k\).
Consequently, the $\ell^2$ error decays as $\mathcal{O}(\Delta t)$ on any
finite time interval $[0,T]$.
\end{lemma}
\begin{proof}
See Appendix~\ref{app:temporal discretization error}.
\end{proof}

For the unregularized linearized mode
$\ddot c+\gamma\dot c+\lambda c=0$ with constant $\gamma>0$, write
$v=\dot c$. The update order in \eqref{eq:num_scheme} gives
\[
\begin{pmatrix}c^{k+1}\\v^{k+1}\end{pmatrix}
=\frac{1}{1+\gamma\Delta t}
\begin{pmatrix}
1+\gamma\Delta t-\lambda\Delta t^2 & \Delta t\\
-\lambda\Delta t & 1
\end{pmatrix}
\begin{pmatrix}c^k\\v^k\end{pmatrix}.
\]
The Jury conditions show that a nonzero mode is stable when
\begin{equation}\label{ieq:CFL_condition}
\lambda\Delta t^2<4+2\gamma\Delta t.
\end{equation}
Consequently, it is sufficient to impose this condition with the actual
$\lambda_{\max}$ of the weighted sparse graph, equivalently
\[
\Delta t<\frac{\gamma+\sqrt{\gamma^2+4\lambda_{\max}}}{\lambda_{\max}}.
\]
An estimate such as $\lambda_{\max}=O(h^{-\alpha})$, when established
under the chosen kernel normalization and uniform weight bounds, describes
mesh dependence but does not determine the prefactor for a particular
graph. The condition above applies to the unregularized linear update.

\subsection{Sparse graph approximation error}
\label{sec:sparse_error}

We now construct a sparse approximation of the discrete fractional
operator $\mathcal L_N^{\mathrm{full}}$ and quantify the resulting error. In this section, we absorb the volume element into the edge weight for
notational simplicity and write
\[ \nu_{ij} := C_{d,\alpha}h^d |x_i-x_j|^{-d-\alpha}.\]
This section incorporates the volume element directly into the graph
weight. Consequently,
\(\nu_{ij}\le C_\nu h^{-\alpha}\), $C_{\nu}>0$.

Let $\mathcal L_N^{\mathrm{full}}$ denote the discrete fractional operator obtained
by applying a standard quadrature rule to the integral fractional
Laplacian on a uniform grid:
\begin{equation}
(\mathcal L_N^{\mathrm{full}} f)_i := \sum_{j\neq i}(f_j-f_i)\nu_{ij},
\qquad \nu_{ij}\propto\|x_i-x_j\|^{-d-\alpha}.
\label{eq:discrete_fractional}
\end{equation}
 
This operator is the discrete counterpart of the negative of the usual positive fractional Laplacian, $-(-\Delta)^{\alpha/2}$.
The local radius $r_{cut}$ is the largest realized distance in the symmetric KNN edge set.
For \(0<\alpha<2\) and any cutoff radius \(r_{\mathrm{cut}}>0\),
the fractional Laplacian splits into local and long-range parts:
\[
\begin{aligned}
	(-\Delta)^{\alpha/2}u(x)
	= C_{d,\alpha}\operatorname{p.v.}
	\int_{\|y-x\|\le r_{\mathrm{cut}}}
	\frac{u(x)-u(y)}{\|x-y\|^{d+\alpha}}\,dy 
	+ C_{d,\alpha}
	\int_{\|y-x\|>r_{\mathrm{cut}}}
	\frac{u(x)-u(y)}{\|x-y\|^{d+\alpha}}\,dy,
\end{aligned}
\]
where \(\operatorname{p.v.}\) denotes the Cauchy principal value.
On the discrete grid, the local graph is defined by the symmetric
$k_{\mathrm{local}}$-nearest-neighbor relation. 
The long-range part
is estimated by a one-time Monte Carlo sampling of the L{\'e}vy kernel
during graph construction. In addition, the sampled edges and weights remain fixed
throughout the subsequent evolution, so the overall dynamics is deterministic.
\begin{remark}
    We use $C_{d,\alpha}=1$ as a normalization of the graph
    operator; the reported spectra, damping parameters, and settling times
    are expressed in this normalized unit system.
\end{remark}
\paragraph{Sampling scheme and graph symmetrization.}
The sparse graph is constructed in three steps. First, we use the symmetric local edge set
\[
\mathcal E_{\rm loc}=\{\{i,j\}:j\in\mathrm{KNN}(i)\ \text{or}\ i\in\mathrm{KNN}(j)\}.
\]
Second, for each node \(i\), let
\(
\mathcal S_i:=\{j:\ j\ne i\}.
\)
We perform \(m\) independent draws with replacement from \(\mathcal S_i\),
according to the probability distribution
\[
p_{ij}
=
\frac{|x_i-x_j|^{-(d+\alpha)}}
{\sum_{\ell\in\mathcal S_i}|x_i-x_\ell|^{-(d+\alpha)}},
\qquad j\in\mathcal S_i,
\]
so that \(\sum_{j\in\mathcal S_i}p_{ij}=1\). Equivalently,
\(
J_{i,1},\dots,J_{i,m}
\)
are i.i.d. draws from \(p_{i\cdot}\).

Third, the sampled directed edges are stored as undirected edges. For each
ordered pair \((i,j)\), define the directed inclusion indicator
\(
\eta_{ij}
:=
\mathbf 1
\left\{
\exists r\in\{1,\dots,m\}:\ J_{i,r}=j
\right\}.
\)
Repeated selections of the same directed edge are collapsed before
\(\eta_{ij}\) is defined. Equivalently, \(\eta_{ij}=1\) if direction
\(i\to j\) appears at least once among the \(m\) draws. Hence
\(
\mathbb E[\eta_{ij}]
=
1-(1-p_{ij})^m
=:
q_{ij}.
\)

Since the graph is stored as an undirected graph, define the symmetric
edge inclusion indicator
\(
\zeta_{ij}
:=
1-(1-\eta_{ij})(1-\eta_{ji}),
\)
so that \(\zeta_{ij}=\zeta_{ji}\). Its expectation is
\(
\mathbb E[\zeta_{ij}]
=
1-(1-p_{ij})^m(1-p_{ji})^m
=:
Q_{ij}.
\)
The symmetrized kernel is then defined by
\begin{equation}\label{eq:symmetric_kernel}
\tilde\nu_{ij}^{\mathrm{sym}}
:=
\zeta_{ij}\frac{\nu_{ij}}{Q_{ij}}.
\end{equation}
Consequently,
\(
\tilde\nu_{ij}^{\mathrm{sym}}
=\tilde\nu_{ji}^{\mathrm{sym}},
\mathbb E[\tilde\nu_{ij}^{\mathrm{sym}}]
=\nu_{ij}.
\)

The graph is constructed once at initialization and is fixed throughout the
simulation, so the overall dynamics is deterministic after initial sampling.

For the long-range part, let
\(
\mathcal C_i:=\{j\ne i:\{i,j\}\notin\mathcal E_{\rm loc}\}
\)
denote the set of non-KNN candidates of node \(i\).
The sparse estimator of \(\mathcal L_N^{\mathrm{full}}\) is then defined as
\begin{equation}
(L_N^{\mathrm{sparse}} f)_i
:=
\sum_{j:\{i,j\}\in\mathcal E_{\rm loc}}(f_j-f_i)\nu_{ij}
+
\sum_{j \in \mathcal C_i}
(f_j-f_i)\tilde\nu_{ij}^{\mathrm{sym}},
\label{eq:sparse_operator}
\end{equation}
where the second sum is over the long-range candidate set. Only sampled
long-range edges have nonzero weight, while local KNN edges are
deterministic. Therefore
\(
\mathbb E[(L_N^{\mathrm{sparse}} f)_i]=(\mathcal L_N^{\mathrm{full}} f)_i.
\)

Several existing results provide convergence rates for the 
discretization of the integral fractional Laplacian.
In one dimension, 
Huang and Oberman \cite{HuangOberman2014} proved that for sufficiently 
smooth functions, the 1D fractional Laplacian discretization error is 
$O(h^{2-\alpha})$ (linear interpolation) or $O(h^{3-\alpha})$ (quadratic 
interpolation) using semi-exact quadrature. For arbitrary dimension 
$d\ge 1$, Duo and Zhang \cite{DUO2019639} established that the finite 
difference scheme achieves $O(h^2)$ under 
$u \in C^{2+\lfloor\alpha\rfloor, \alpha-\lfloor\alpha\rfloor+\epsilon}(\mathbb{R}^d)$ 
with a kernel-splitting weighted trapezoidal rule ($\gamma=2$), and 
reduces to lower-order rates under minimal regularity. 

\begin{assumption}[Graph structure]
	\label{ass:graph}
	The graph \(\mathcal{G}=(\mathcal{V},\mathcal{E})\) is finite, connected and the symmetric weights satisfy $\nu_{ij} = \nu_{ji} \ge 0$,
	with $\nu_{ij} > 0$ whenever an edge exists.
	The sparse graph is sampled once at initialization. 
\end{assumption}

\begin{assumption}[Small-sampling regime]
	\label{ass:small_regime}
	The sampling budget satisfies
	\(
	m(p_{ij}+p_{ji})\le 1, 
	\forall (i,j)\in\mathcal C_i.
	\)
\end{assumption}

The following lemma states the error estimate for our graph discretization.

\begin{proposition}[Sparse vs.\ full graph approximation]
\label{prop:sparse_full}
Under Assumptions~\ref{ass:graph} and \ref{ass:small_regime}, assume
\(\|f\|_\infty\le C_f\), \(\|f\|_{\mathrm{Lip}}\le L_f\), and
\(0<\nu_{ij}\le\nu_{\max}<\infty\). Then for every \(\delta\in(0,1)\),
with probability at least \(1-\delta\),
\begin{equation}
\label{ineq:sparse_full}
\|\mathcal L_N^{\mathrm{sparse}}f-\mathcal L_N^{\mathrm{full}}f\|_\infty
\le
C_{\mathrm{graph}}
\sqrt{\frac{\log(4N/\delta)}{m}},
\end{equation}
provided the Bernstein admissibility condition \eqref{eq:uniform_sparse_condition} holds,
where \(C_{\mathrm{graph}}>0\) depends on \(d,\alpha\), the fixed grid,
\(\nu_{\max}\), \(C_f\), \(L_f\), and
\(p_{\min}:= \min_{i,\;j\in\mathcal C_i}p_{ij}>0\).
\end{proposition}

\begin{proof}
   See Appendix~\ref{app:graph error}.
\end{proof}

Figure~\ref{fig:convergence_sparse} gives a numerical illustration
consistent with the predicted $m^{-1/2}$ sparse-estimator scaling. For panels (b)--(c) of Figure~\ref{fig:convergence_sparse}, we use
\(f(x)=e^{-\|x\|^2}\) and evaluate the discrete operator at the grid
point closest to the origin. The full-graph approximation is computed
on the truncated domain \([-L,L]^2\) and compared with the full-space
reference
\(
(-\Delta)^{\alpha/2}f(0)
=
2^\alpha\Gamma(1+\alpha/2),
\)
using the standard normalization \(C_{2,\alpha}\).
Hence the reported errors contain both spatial discretization and
domain-truncation effects.

\begin{remark}
For a fixed discrete graph, Proposition~\ref{prop:sparse_full} gives a
uniform sparse-graph error of order
\(\mathcal{O}(\sqrt{\log N/m})\).
\end{remark}

\begin{figure}[htbp]
    \centering
    \begin{subfigure}[t]{0.30\textwidth}
        \centering
        \includegraphics[width=\textwidth]{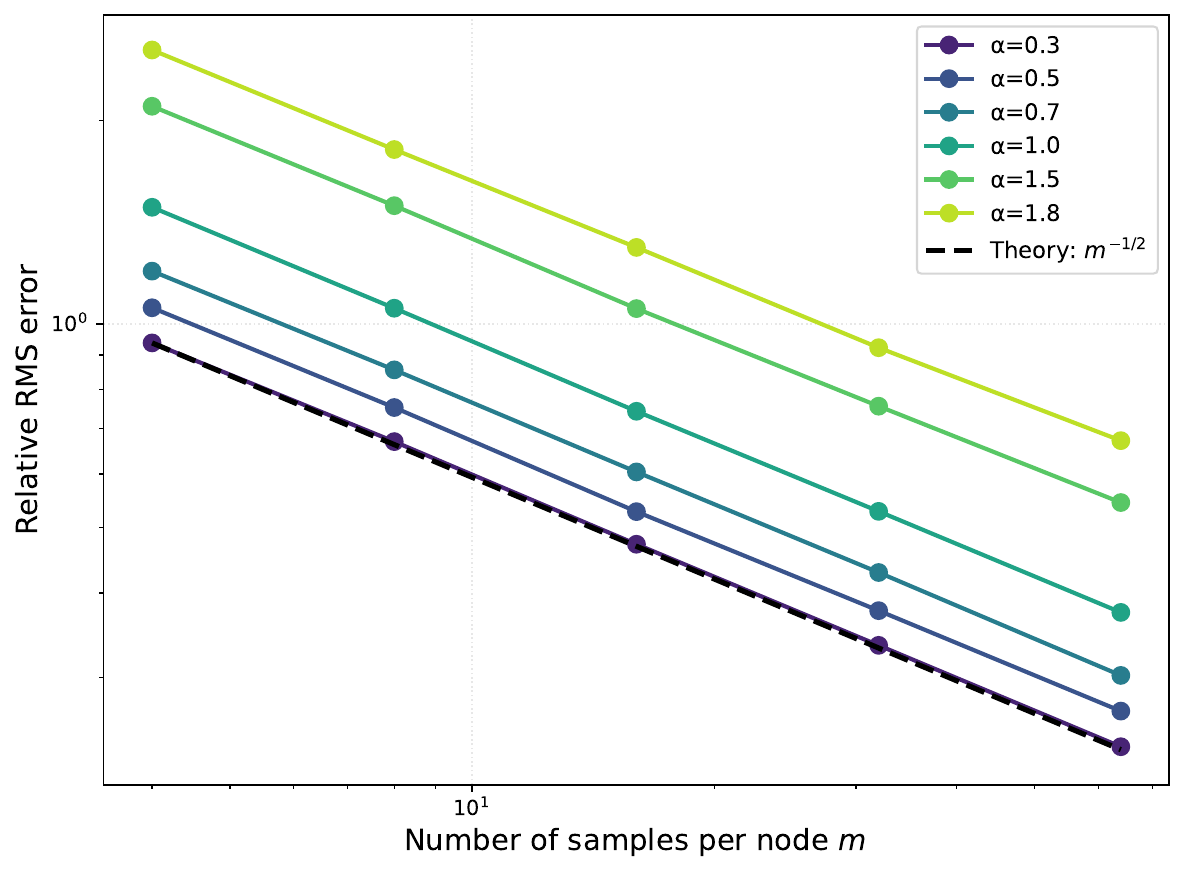}
        \caption{MC convergence of sparse graph}
        \label{fig:mc_convergence_m}
    \end{subfigure}
    \hfill
    \begin{subfigure}[t]{0.30\textwidth}
        \centering
        \includegraphics[width=\textwidth]{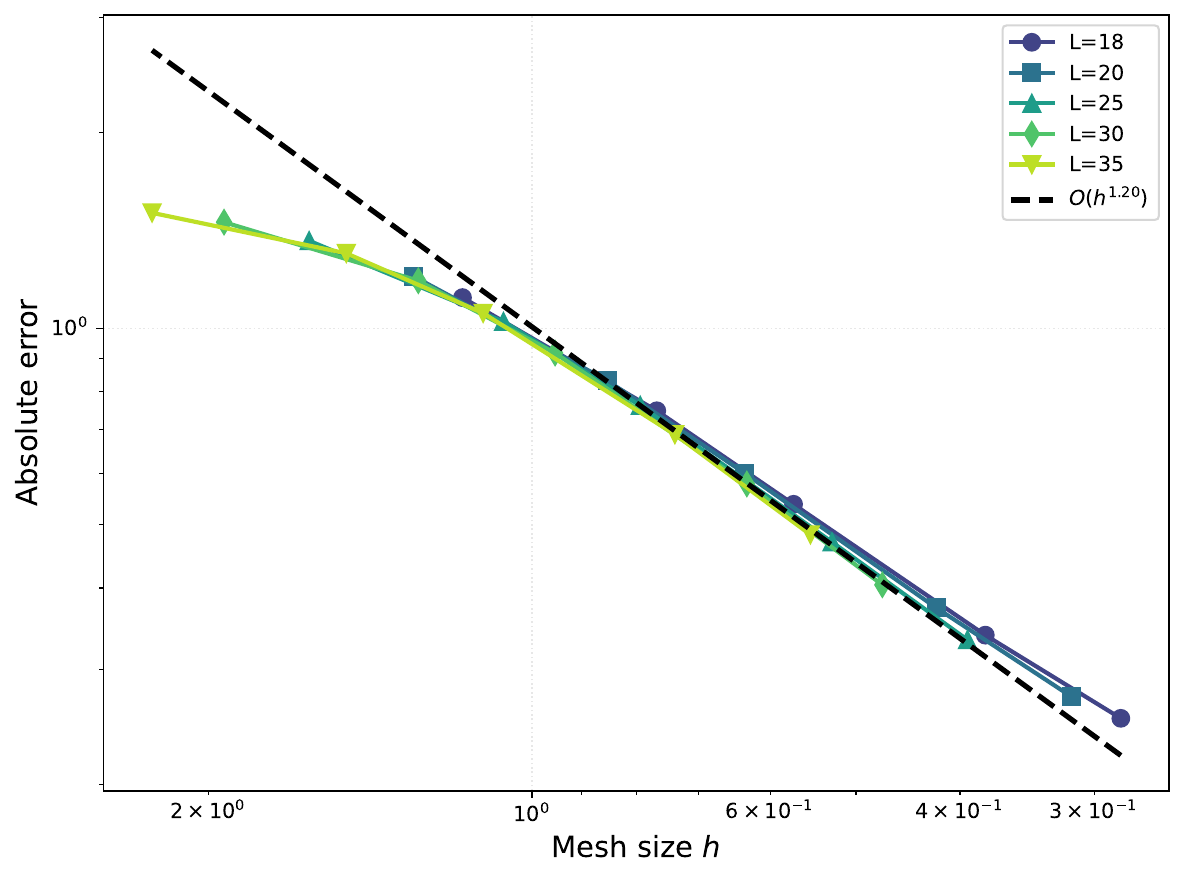}
        \caption{$\alpha=0.8$}
        \label{fig:error_alpha0.8}
    \end{subfigure}
    \hfill
    \begin{subfigure}[t]{0.30\textwidth}
        \centering
        \includegraphics[width=\textwidth]{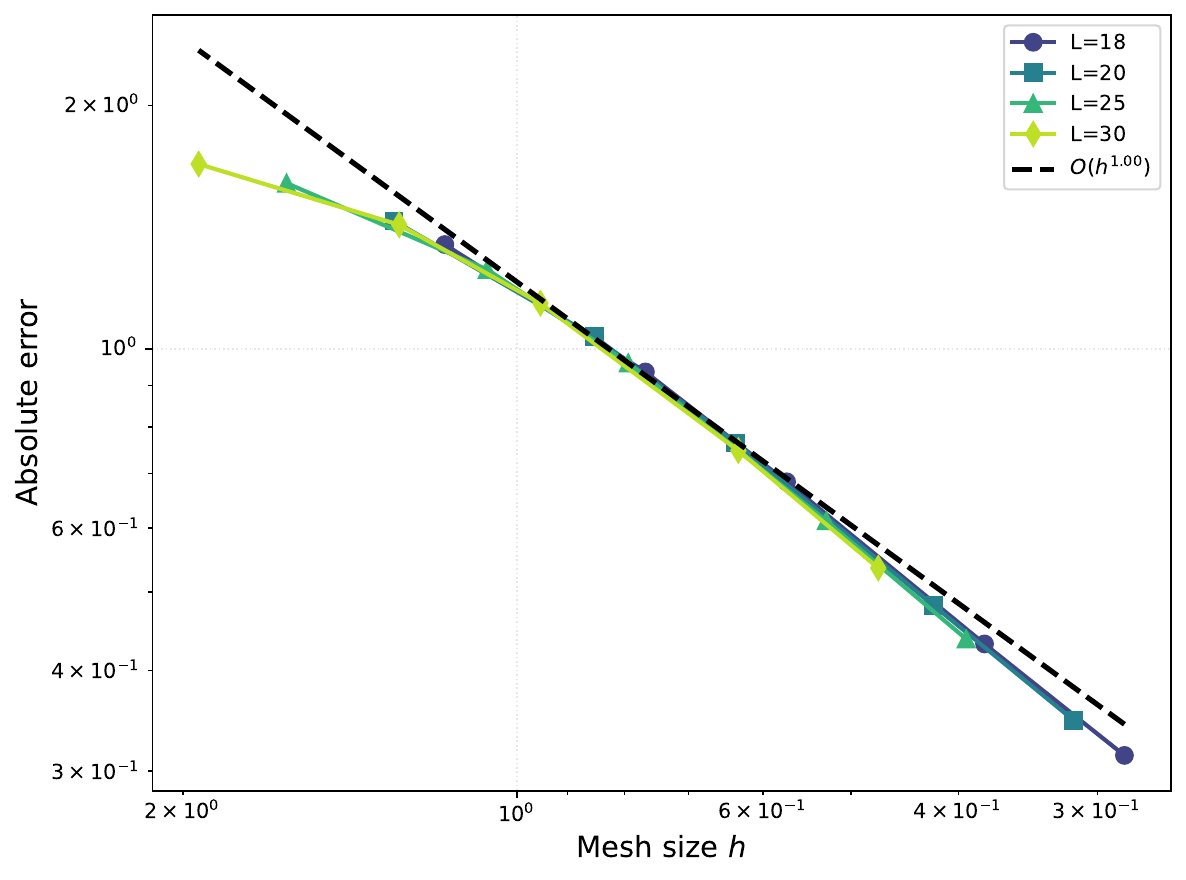}
        \caption{$\alpha=1.0$}
        \label{fig:error_alpha1.0}
    \end{subfigure}
    \caption{ Convergence analysis of the discrete fractional operator.
            (a) Numerical illustration of the \(m^{-1/2}\) scaling for the
            fixed-function sparse estimator.
            (b)--(c) Pointwise absolute errors of the truncated full-graph
            approximation for \(\alpha=0.8\) and \(\alpha=1.0\).
            The dashed lines indicate the reference slopes \(2-\alpha\).
        }
    \label{fig:convergence_sparse}
\end{figure}

The operator estimate above controls a fixed test function. To propagate this approximation through the nonlinear dynamics, we next apply the same sparse-graph concentration argument uniformly along the full-graph trajectory.

\begin{lemma}[Sparse graph forcing error along the full-graph trajectory]
\label{lem:graph_forcing}
Let $z_F^k=(\rho_F^k,u_F^k)$, $0\le k\le K$, denote the
full-graph solution of \eqref{eq:num_scheme}, with deterministic initial
data shared by the sparse scheme and $K=\lfloor T/\Delta t\rfloor$. Assume that the hypotheses of Proposition~\ref{prop:sparse_full}
are satisfied and that the density and velocity bounds stated in
Lemma~\ref{lem:lipschitz} hold uniformly along this trajectory.
For every $\delta_{\rm p}\in(0,1)$, with probability at least
$1-\delta_{\rm p}$,
\[
\max_{0\le k\le K}
\|F^{\rm sparse}(z_F^k)-F^{\rm full}(z_F^k)\|_{\mathcal X}
\le C_{\rm force}
\sqrt{\frac{\log(C_0N(K+1)/\delta_{\rm p})}{m}},
\]
provided the threshold in the Bernstein estimate lies in its
admissible range. On the same event, the density component of the forcing error satisfies the same bound at $(\rho_F^k,u_F^{k+1})$ for all $0\le k<K$.
Here $C_0$ is an absolute constant, and $C_{\rm force}$ is independent
of $m$ and $K$ for the stated uniform state bounds. 
\begin{proof}
	See Appendix~\ref{app:graph error}.
\end{proof}

\end{lemma}
Combining the sparse forcing error with the temporal discretization bound and the continuous-time relaxation estimate gives the total error decomposition below.
\begin{theorem}[Finite-horizon total error bound]
\label{thm:total_error}
Let $\rho_h^k=\rho_S^k$ be the sparse-graph solution of
\eqref{eq:num_scheme}, and let $(\rho^F(t),u^F(t))$ be the continuous
full-graph solution with the same initial data.
Assume the hypotheses of Theorem~\ref{thm:local_exp},
Lemma~\ref{lem:temporal_error}, and Lemma~\ref{lem:graph_forcing}
hold on $[0,T]$, with both discrete trajectories in the bounded
region of Lemma~\ref{lem:lipschitz}.
Then, with probability at least $1-\delta_{\rm p}$, simultaneously
for $t_k=k\Delta t\le T$,
\begin{equation}\label{eq:total_error}
\begin{aligned}
\|\rho_h^k-\pi\|_{\ell^2(\pi)}
\le{}&C_{\rm temp}e^{CT}\Delta t
+C_{\rm graph,T}
\sqrt{\frac{\log(C_0N(K+1)/\delta_{\rm p})}{m}}\\
&+C_{\rm exp}(1+t_k)e^{-r_1(\gamma)t_k},
\qquad K=\lfloor T/\Delta t\rfloor.
\end{aligned}
\end{equation}
The constants are independent of $m$, $\Delta t$ and $k$ for fixed $T$.
\begin{proof}
	See Appendix~\ref{app:graph error}.
\end{proof}

\end{theorem}
Theorem~\ref{thm:total_error} measures the distance between the numerical
solution $\rho_h^k$ and the target distribution $\pi$.
The following corollary translates the total error bound from
the \(\ell^2(\pi)\)-norm to the nonlocal Wasserstein metric using
Lemma~\ref{lem:wnl_l2}.

\begin{corollary}[Finite-horizon Wasserstein error bound]
\label{cor:wnl_error}
Under the assumptions of Theorem~\ref{thm:total_error},
assume in addition
that \(\rho_h^k\) remains in the neighborhood
of Lemma~\ref{lem:wnl_l2} for all \(0\le t_k\le T\).
Then, with probability at least \(1-\delta_{\rm p}\),
the following estimate holds for all \(0\le t_k\le T\):
\begin{equation}
\begin{aligned}
W_{\mathrm{NL}}(\rho_h^k,\pi)
\le{}
\frac{1+\eta}{\pi_{\min}\sqrt{\lambda_1}}
\Bigg[
&
C_{\mathrm{temp}}e^{CT}\Delta t+
C_{\rm graph,T}
\sqrt{\frac{\log(C_0N(K+1)/\delta_{\rm p})}{m}}
\\
&+
C_{\mathrm{exp}}(1+t_k)e^{-r_1(\gamma)t_k}
\Bigg].
\end{aligned}
\label{eq:wnl_error}
\end{equation}
If \(\gamma=\gamma^*=2\sqrt{\lambda_1}\), then \(r_1(\gamma)=\sqrt{\lambda_1}\). Here $W_{\mathrm{NL}}$ and the spectral constant $\lambda_1$
refer to the same weighted graph.
\end{corollary}
\begin{proof}
	See Appendix~\ref{app:graph error}.
\end{proof}

\section{Numerical Experiments}\label{sec: numerical results}

In this section, we assess the practical performance of our method on a range of 
synthetic and real-world sampling tasks, comparing against standard baselines—MCMC, 
MALA (Metropolis-adjusted Langevin algorithm) and ULD (underdamped Langevin dynamics). The test problems comprise three synthetic distributions (banana, single ring, and four-Gaussian-mixture), the MNIST image dataset, and the alanine dipeptide molecular conformation dataset. 
The experiments assess the effects of nonlocality and damping, multimodal exploration, and computational performance against baseline samplers.
\subsection{Experimental Setup}

All experiments share a common discretization and graph construction framework. The spatial domain \([-5,5]^2\) is discretized using a \(50 \times 50\) uniform grid, yielding mesh size \(h = 10/49 \approx 0.2\). The numerical graph combines a symmetric KNN local graph with
long-range edges sampled according to the L{\'e}vy distance kernel. Local connectivity is established through \(k_{\text{local}} = 10\) nearest neighbors, ensuring \(k_{\text{local}} > \log N\). At \(\alpha=1.0\), the long-range sampling budget is \(m=3\) per node. The time step is \(\Delta t = 5 \times 10^{-4}\). For the fixed-step experiments below, we use \(5\times10^4\) time steps.
The wall-clock benchmark uses a separate fixed-time protocol described later. For all benchmark comparisons, the proposed method and the particle-based samplers are initialized from the same Gaussian distribution.
In the implementation, a density floor of $10^{-12}$ followed by
renormalization, bounded velocity and convection forcing, and
machine-precision regularization are used as numerical safeguards.

We investigate the effect of the kernel decay exponent \(\alpha\) on sampling performance by varying it from \(0.1\) to \(2.0\). We first report results on the banana-shaped target to illustrate the behavior of the method. We also monitor the transport length and the long-range flux ratio over time as diagnostics of nonlocal transport.

\subsection{Effect of Nonlocality}

We first investigate the influence of the fractional order $\alpha$ on the convergence behavior of the proposed method. The banana distribution is a non-convex target with two modes separated by an energy barrier \cite{stein2025accelerated}, making it a representative testbed to evaluate performance in curved geometries.

\begin{figure}[htbp]
    \centering
    
    \begin{subfigure}[b]{0.22\textwidth}
        \centering
        \includegraphics[width=\textwidth, height=5.5cm, keepaspectratio]{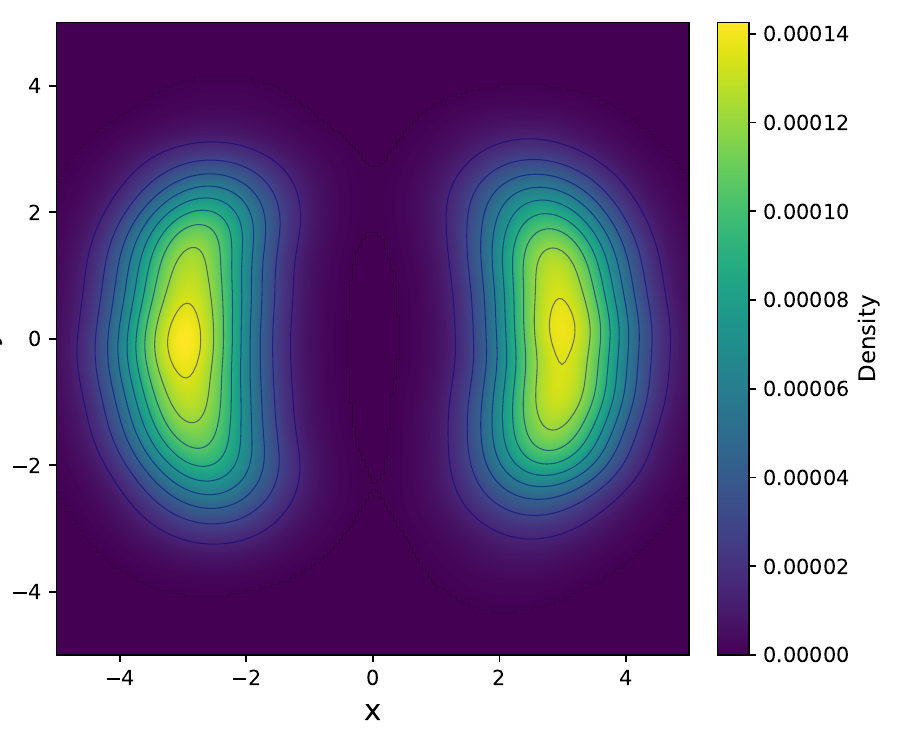}
      
        \label{fig:sampling_alpha0.3}
    \end{subfigure}
    \hfill
    \begin{subfigure}[b]{0.22\textwidth}
        \centering
        \includegraphics[width=\textwidth, height=5.5cm, keepaspectratio]{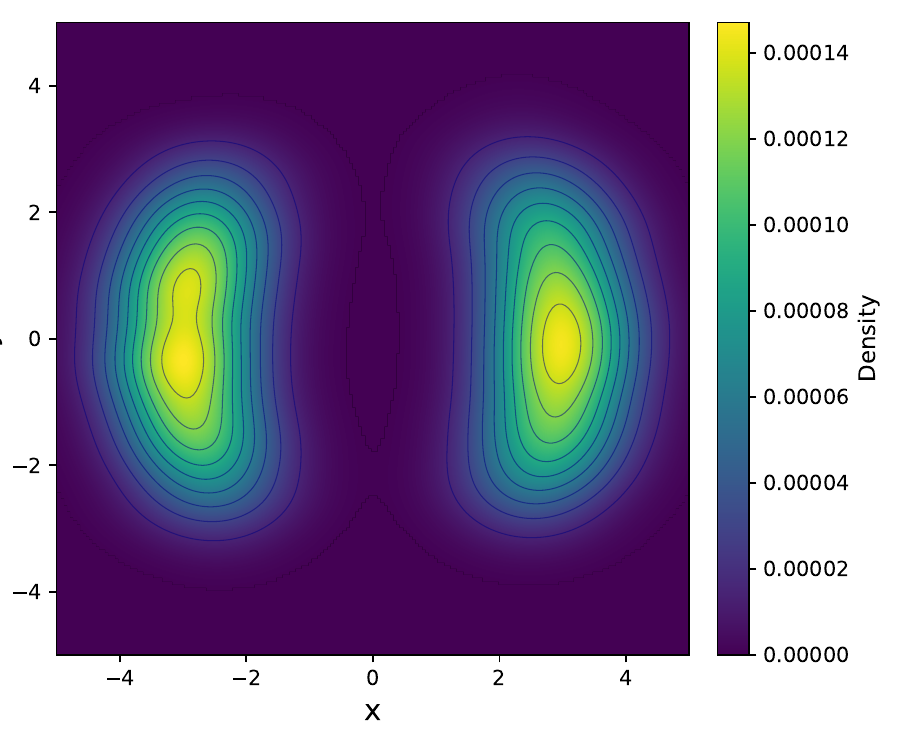}
      
        \label{fig:sampling_alpha0.7}
    \end{subfigure}
    \hfill
    \begin{subfigure}[b]{0.22\textwidth}
        \centering
        \includegraphics[width=\textwidth, height=5.5cm, keepaspectratio]{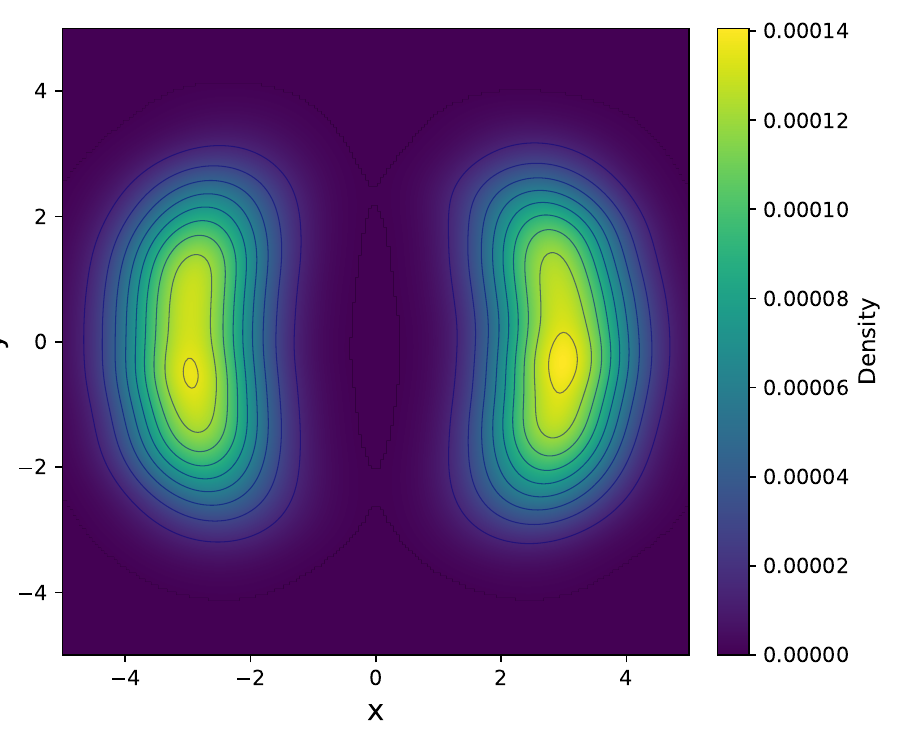}
       
        \label{fig:sampling_alpha1.5}
    \end{subfigure}
    \hfill
    \begin{subfigure}[b]{0.22\textwidth}
        \centering
        \includegraphics[width=\textwidth, height=5.5cm, keepaspectratio]{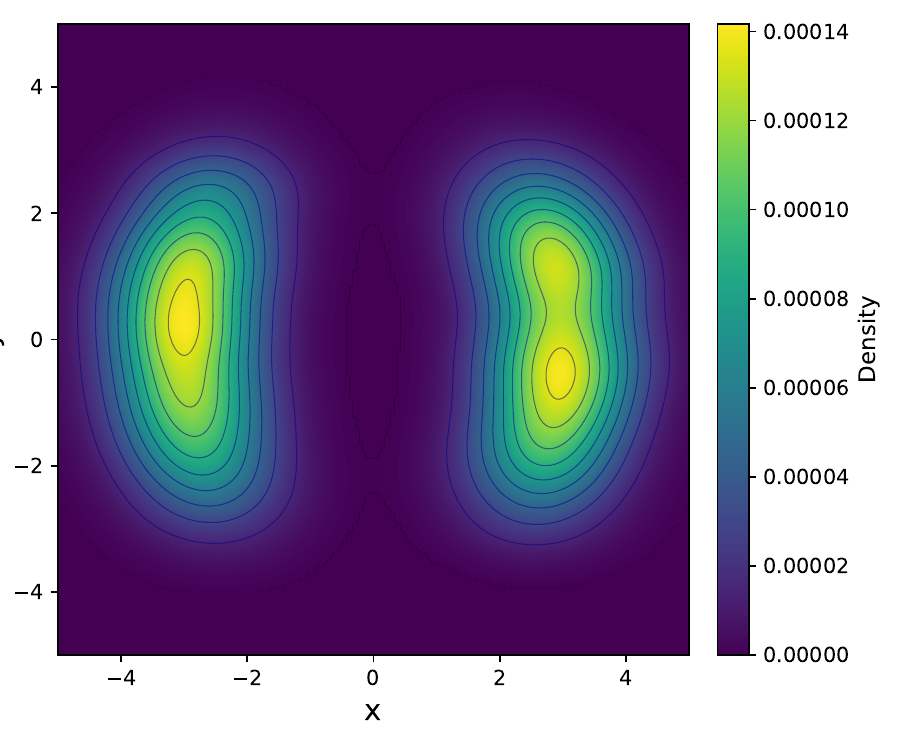}
      
        \label{fig:sampling_alpha1.8}
    \end{subfigure}
    
    \vspace{0.3cm}
    
    \begin{subfigure}[b]{0.22\textwidth}
        \centering
        \includegraphics[width=\textwidth, height=5.5cm, keepaspectratio]{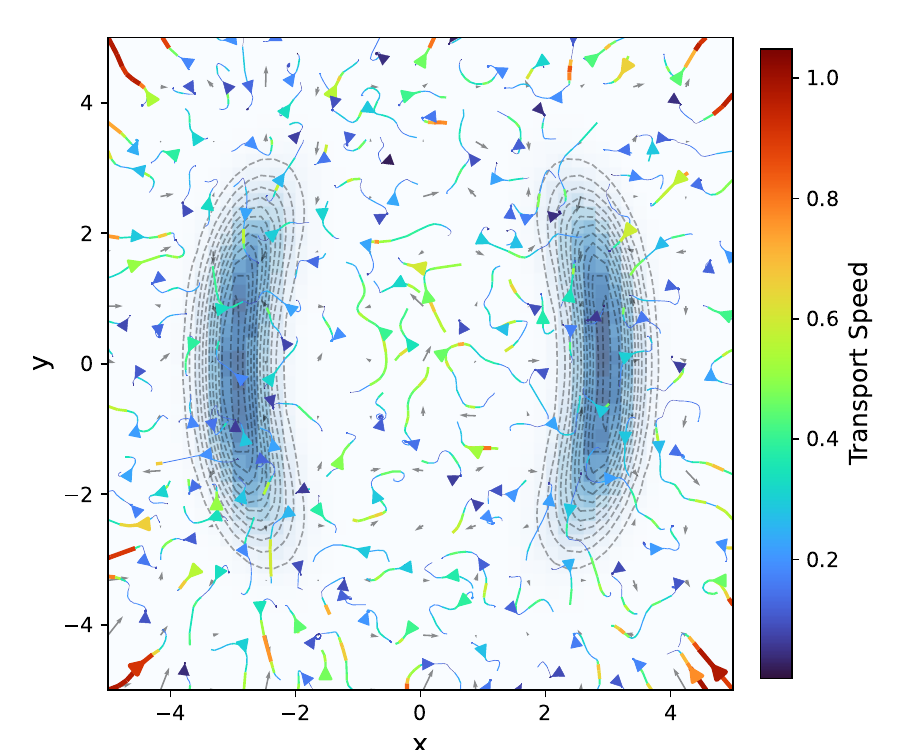}
     
        \label{fig:transport_alpha0.3}
    \end{subfigure}
    \hfill
    \begin{subfigure}[b]{0.22\textwidth}
        \centering
        \includegraphics[width=\textwidth, height=5.5cm, keepaspectratio]{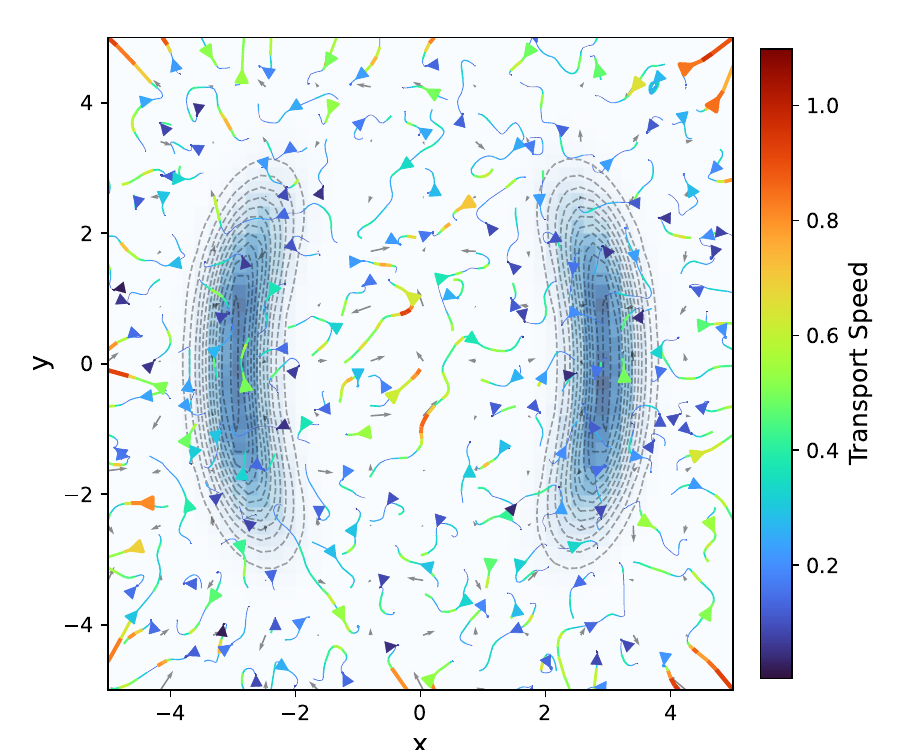}
      
        \label{fig:transport_alpha0.7}
    \end{subfigure}
    \hfill
    \begin{subfigure}[b]{0.22\textwidth}
        \centering
        \includegraphics[width=\textwidth, height=5.5cm, keepaspectratio]{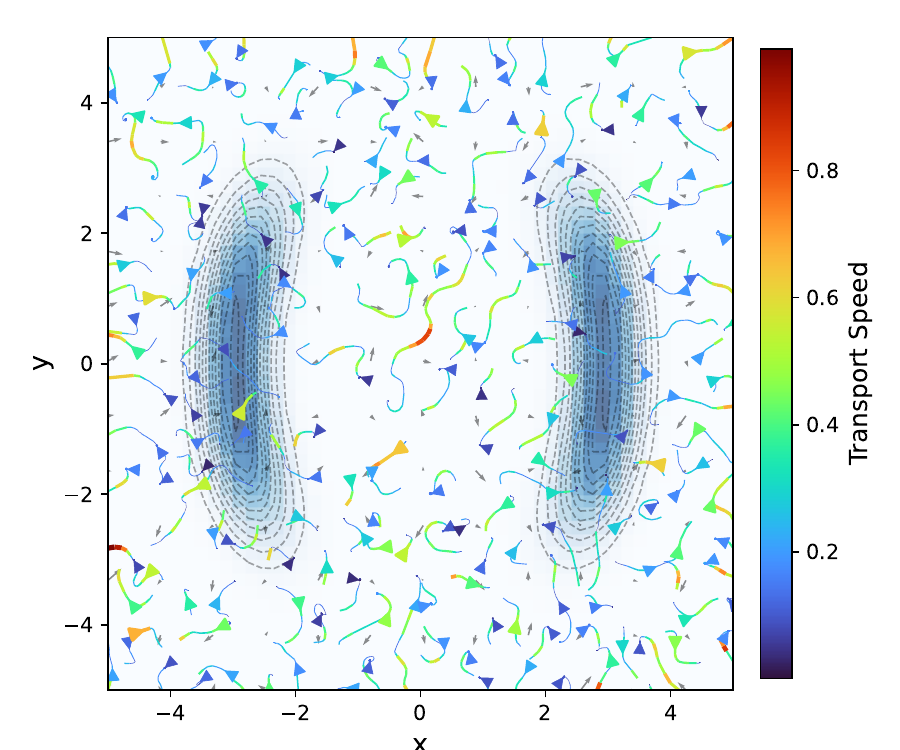}
 
        \label{fig:transport_alpha1.5}
    \end{subfigure}
    \hfill
    \begin{subfigure}[b]{0.22\textwidth}
        \centering
        \includegraphics[width=\textwidth, height=5.5cm, keepaspectratio]{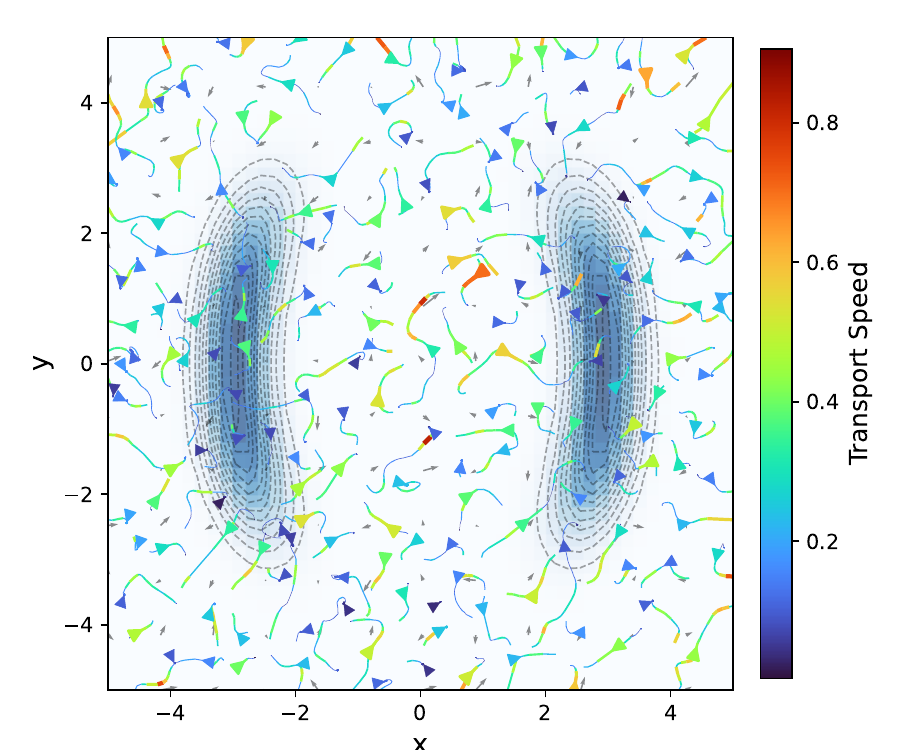}
   
        \label{fig:transport_alpha1.8}
    \end{subfigure}
    
    \caption{Sampling results (kernel density estimation (KDE), top) and transport fields (bottom) for different fractional orders $\alpha = 0.3, 0.7, 1.5, 1.8$ (from left to right).}
    \label{fig:alpha_comparison}
\end{figure}

\begin{figure}[htbp]
	\centering
	\begin{subfigure}[b]{0.48\textwidth}
		\centering
		\includegraphics[width=\textwidth, keepaspectratio=false]{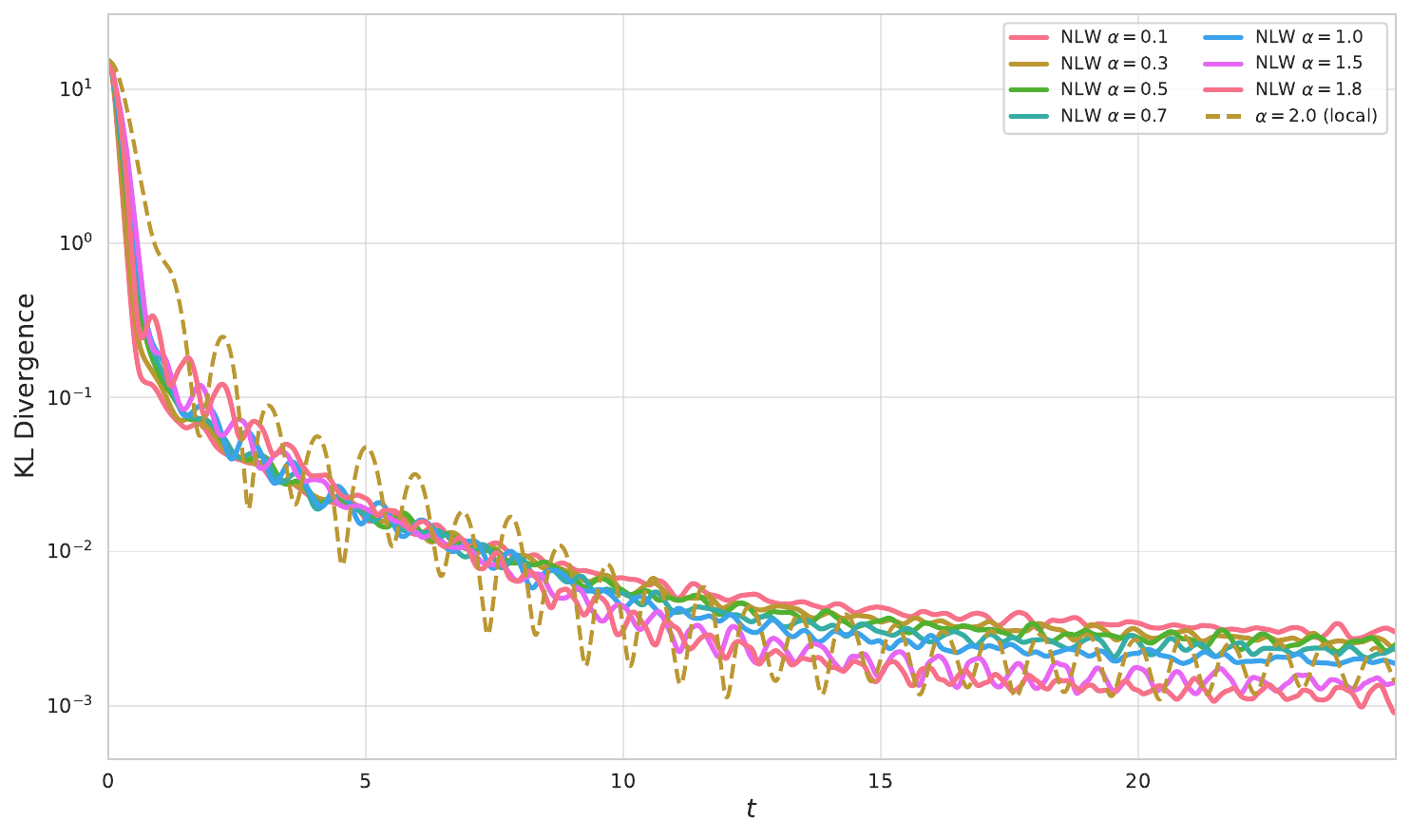}
		\caption{KL divergence}
		\label{fig:kl_curves}
	\end{subfigure}
	\hfill
	\begin{subfigure}[b]{0.48\textwidth}
		\centering
		\includegraphics[width=\textwidth, keepaspectratio=false]{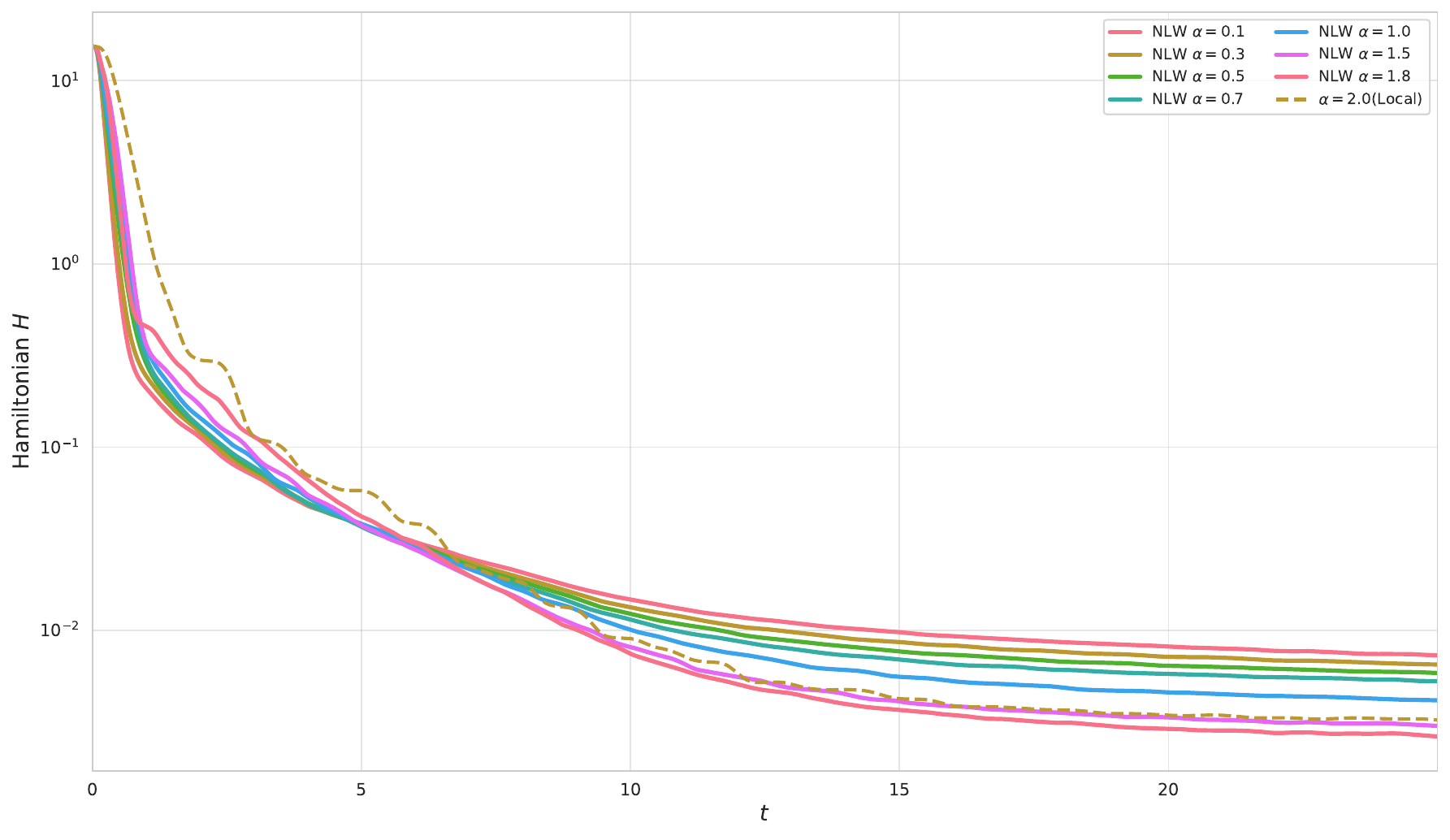}
		\caption{Hamiltonian}
		\label{fig:Hamiltonian_curves}
	\end{subfigure}
	\caption{Evolution of KL divergence and Hamiltonian over time for the banana distribution with different fractional orders $\alpha$.}
	\label{fig:banana_metrics_evolution}
\end{figure}

From Figure~\ref{fig:alpha_comparison}, we observe that the proposed method 
achieves complete mode coverage for all tested $\alpha$ values, demonstrating 
its robustness across different levels of nonlocality. Smaller \(\alpha\) yields faster, longer-range transport (indicated by redder streamlines and stronger nonlocal flux), whereas larger \(\alpha\) leads to more localized, diffusive transport while still preserving full mode coverage.

Figure~\ref{fig:banana_metrics_evolution} compares the convergence behavior for different \(\alpha\) on the Banana distribution, showing both the KL divergence and the Hamiltonian evolution over time.
First, smaller \(\alpha\) produces faster initial decay, which is consistent with enhanced long-range exploration under stronger nonlocal interactions. Second, we observe oscillatory behavior in the KL divergence, which is typical 
for second-order gradient methods with momentum \cite{su2016differential, attouch2018fast}. The Hamiltonian decreases monotonically, consistent with the dissipation structure of the damped dynamics. Third, larger $\alpha$ ($\alpha \geq 1.5$) produces more localized interactions and slower convergence in this experiment. These results confirm that the nonlocal mechanism is essential for effective sampling on non-convex landscapes.

\begin{table}[htbp]
    \centering
    \caption{Nonlocal transport characteristics on the Banana distribution.} 
    \begin{tabular}{cccc}
        \hline
        $\alpha$ & Mode & Long Jump Ratio & Transport Length \\
        \hline
        0.10 & Nonlocal & 0.145 & 1.118 \\
        0.30 & Nonlocal & 0.129 & 0.980 \\
        0.50 & Nonlocal & 0.093 & 0.828 \\
        0.70 & Nonlocal & 0.069 & 0.729 \\
        1.00 & Nonlocal & 0.031 & 0.571 \\
        1.50 & Nonlocal & 0.016 & 0.451 \\
        \hline
        2.00 & Local KNN baseline & 0.000 & 0.234 \\
        \hline
    \end{tabular}
    \label{tab:banana_transport}
\end{table}
We quantify the strength of nonlocal transport through two diagnostic statistics: the long-range flux ratio
\(R_{\text{long}} = \frac{\sum_{|x_i-x_j| > r_{\text{cut}}} |m_{ij}|}{\sum_{ij} |m_{ij}|}\),
and the transport length
\(L(t) = \frac{\sum_{ij} |m_{ij}| \, |x_i-x_j|}{\sum_{ij} |m_{ij}|}\),
where \(m_{ij} = \theta_{ij} u_{ij} \nu_{ij}\) denotes the flux along edge \((i,j)\), \(r_{\rm cut}\) is the largest distance among the retained
symmetric KNN edges. Table~\ref{tab:banana_transport} summarizes these statistics across different \(\alpha\) values. The long jump ratio decreases from \(0.145\) at \(\alpha=0.1\) to \(0.016\) at \(\alpha=1.5\), confirming that smaller \(\alpha\) induces stronger long-range connectivity. The transport length follows the same trend, decreasing from \(1.118\) to \(0.451\) as \(\alpha\) increases. These results confirm that \(\alpha\) systematically controls the strength of nonlocal interactions. The $\alpha=2.0$ entry is obtained with the separate local KNN solver, with long-range edges disabled. This agrees with the nonlocal-to-local transition
studied in~\cite{Du2022NonlocalDM}.

\subsection{First- and Second-Order Dynamics}

We compare the first-order counterpart (Eq.~(\ref{eQ:first_NLW})) against the proposed second-order flow. Three representative targets are considered: single ring, banana, and four-Gaussian-mixture, all with \(\alpha = 1.0\).

\begin{figure}[!ht]
    \centering
    \begin{subfigure}[b]{0.32\textwidth}
        \centering
        \includegraphics[width=\textwidth]{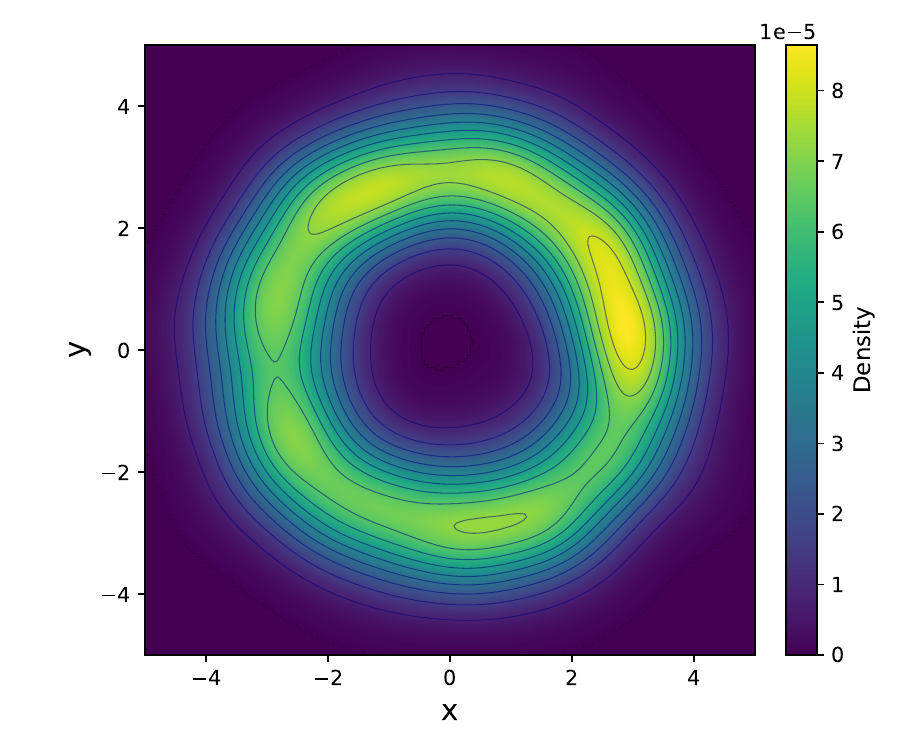}
        \caption{}
        \label{fig:first_order_single_ring}
    \end{subfigure}
    \hfill
    \begin{subfigure}[b]{0.32\textwidth}
        \centering
        \includegraphics[width=\textwidth]{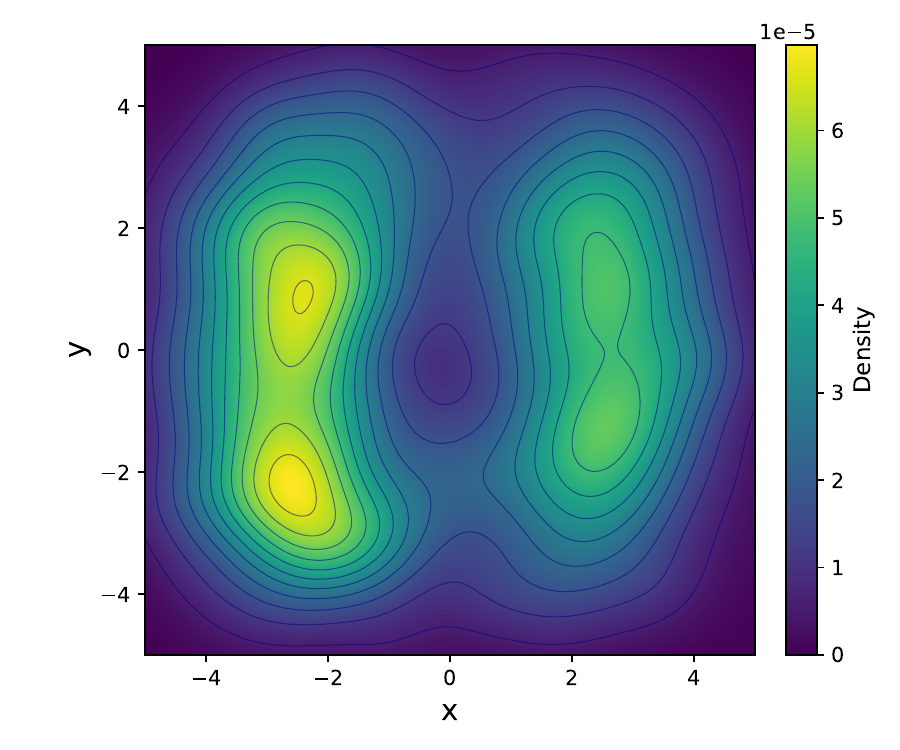}
        \caption{}
        \label{fig:first_order_banana}
    \end{subfigure}
    \hfill
    \begin{subfigure}[b]{0.32\textwidth}
        \centering
        \includegraphics[width=\textwidth]{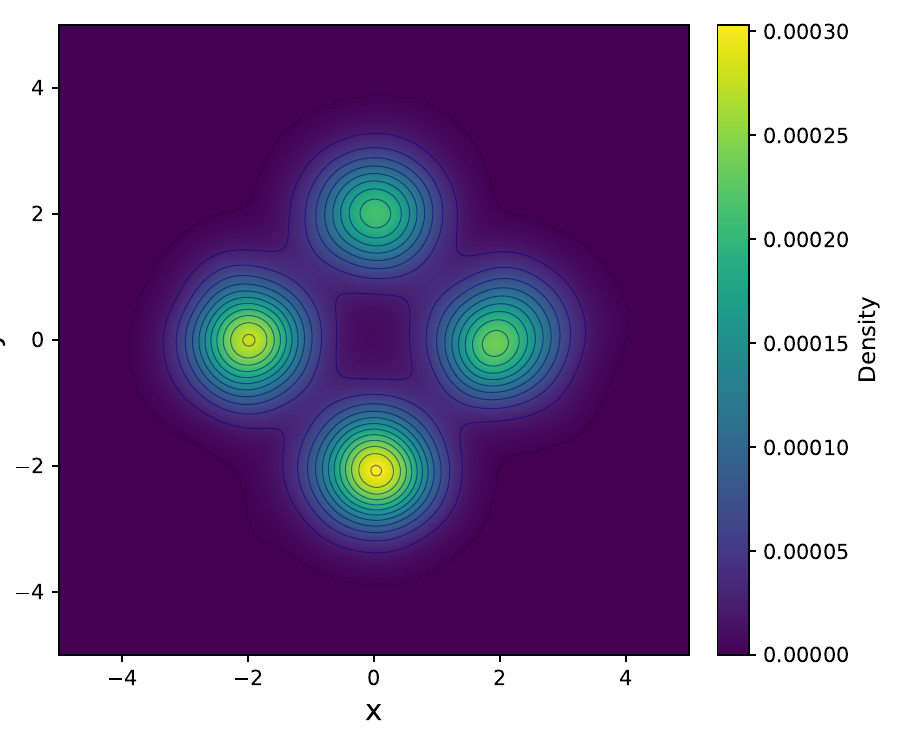}
        \caption{}
        \label{fig:first_order_four_peaks}
    \end{subfigure}
    \caption{
        First-order nonlocal gradient flow sampling results (KDE) on three target distributions 
        with $\alpha=1.0$. (a) Single ring, (b) banana, (c) four-Gaussian-mixture. 
    }
    \label{fig:first_order_results}
\end{figure}
Under the reported finite computational budget, Figure~\ref{fig:first_order_results} and Table~\ref{tab:mode_occupancy} provide a qualitative comparison with the first-order nonlocal flow. The difference motivates examining the role of damping in the second-order dynamics. A natural question then arises: what causes this performance degradation? A key insight is that the second-order dynamics reduces to the first-order flow in the overdamped limit \(\gamma \to \infty\), where the inertial term becomes negligible.

To examine this connection, we fix \(\alpha = 1.0\) on the banana distribution and sweep the damping coefficient \(\gamma\in\{0.1,0.5,1.0,2.0,5.0,10.0,100.0\}\).
Figure~\ref{fig:gamma_scan_KL} illustrates the damping dependence of the second-order nonlocal flow. Consistent with the linearized rate analysis, the fastest decay occurs near \(\gamma^*=2\sqrt{\lambda_1}\), while excessively large damping leads to a strongly overdamped, first-order-like regime with  \(r(\gamma)\sim\lambda_1/\gamma\).

The comparison above was performed for a fixed value of \(\alpha\).
However, \(\alpha\) controls the tail of the L{\'e}vy kernel and hence the
long-range connectivity of the graph. Therefore, the relative advantage
of the second-order flow may change with \(\alpha\). To isolate this
effect, we next compute the linearized spectral gap and compare the
first- and second-order settling times.

\subsubsection{Spectral and Damping Analysis}
\label{subsubsec:mixing_time_verification}

To illustrate the spectral picture from Section~\ref{sec:convergence}, we
compute the spectral gap of the linearized graph operator and compare the
resulting linearized modal mixing-time predictions for the
first- and second-order dynamics.

\textbf{Setup.} We use the same graphs as in the preceding experiments,
with the banana distribution as the target.
For each \(\alpha\in\{0.1,0.5,1.0,1.9\}\), we form the weighted graph Laplacian \(L_\alpha\) with weights \(w_{ij}^{(\alpha)}=\theta(\pi_i,\pi_j)\nu_{ij}^{(\alpha)}\). Since the theoretical linearized operator is \(A_\alpha=D^{-1}L_\alpha\), self-adjoint in the \(D\)-weighted inner product, we compute its spectrum through the symmetric similar matrix
\(
B_\alpha=D^{-1/2}L_\alpha D^{-1/2}.
\)
The spectral gap \(\lambda_1\) is the smallest positive eigenvalue of \(B_\alpha\).
Equivalently,
\[
\lambda_1(\alpha)=
\min_{\substack{
		v\in\mathbb R^N\setminus\{0\}\\
		\sum_{i\in\mathcal V}\pi_i v_i=0}}
\frac{
	\displaystyle
	\sum_{\{i,j\}\in\mathcal E}
	\theta(\pi_i,\pi_j)\nu_{ij}^{(\alpha)}(v_i-v_j)^2
}{
	\displaystyle
	\sum_{i\in\mathcal V}\pi_i v_i^2
}.
\]
The curves below are computed from this sparse-graph spectrum using
the linearized slowest mode with normalized initial displacement and zero initial
momentum:
\[
R_1(t)=e^{-\lambda_1t},\qquad
R_2(t)=(1+\sqrt{\lambda_1}t)e^{-\sqrt{\lambda_1}t},
\]
where the second expression corresponds to critical damping
$\gamma=2\sqrt{\lambda_1}$.

For the mixing-time comparison, we use the normalized
error $R(t)=
\frac{\|\rho_t-\pi\|_{\pi^{-1}}}
{\|\rho_0-\pi\|_{\pi^{-1}}}$
and use the settling-time definition in
Corollary~\ref{cor:acceleration}, with $\varepsilon=10^{-3}$.
For data observed only up to $T_{\max}$, its finite-horizon counterpart is
\[
T_{\mathrm{mix}}^{(T_{\max})}(\varepsilon)
=\inf\{t\in[0,T_{\max}]:R(s)\le\varepsilon
\ \text{for every }s\in[t,T_{\max}]\},
\]
with $\inf\varnothing=+\infty$. This definition describes settling
within the observation interval. 

\begin{figure}[t]
	\centering
	
	\begin{subfigure}[b]{0.31\textwidth}
		\centering
		\includegraphics[height=3.25cm,keepaspectratio]
{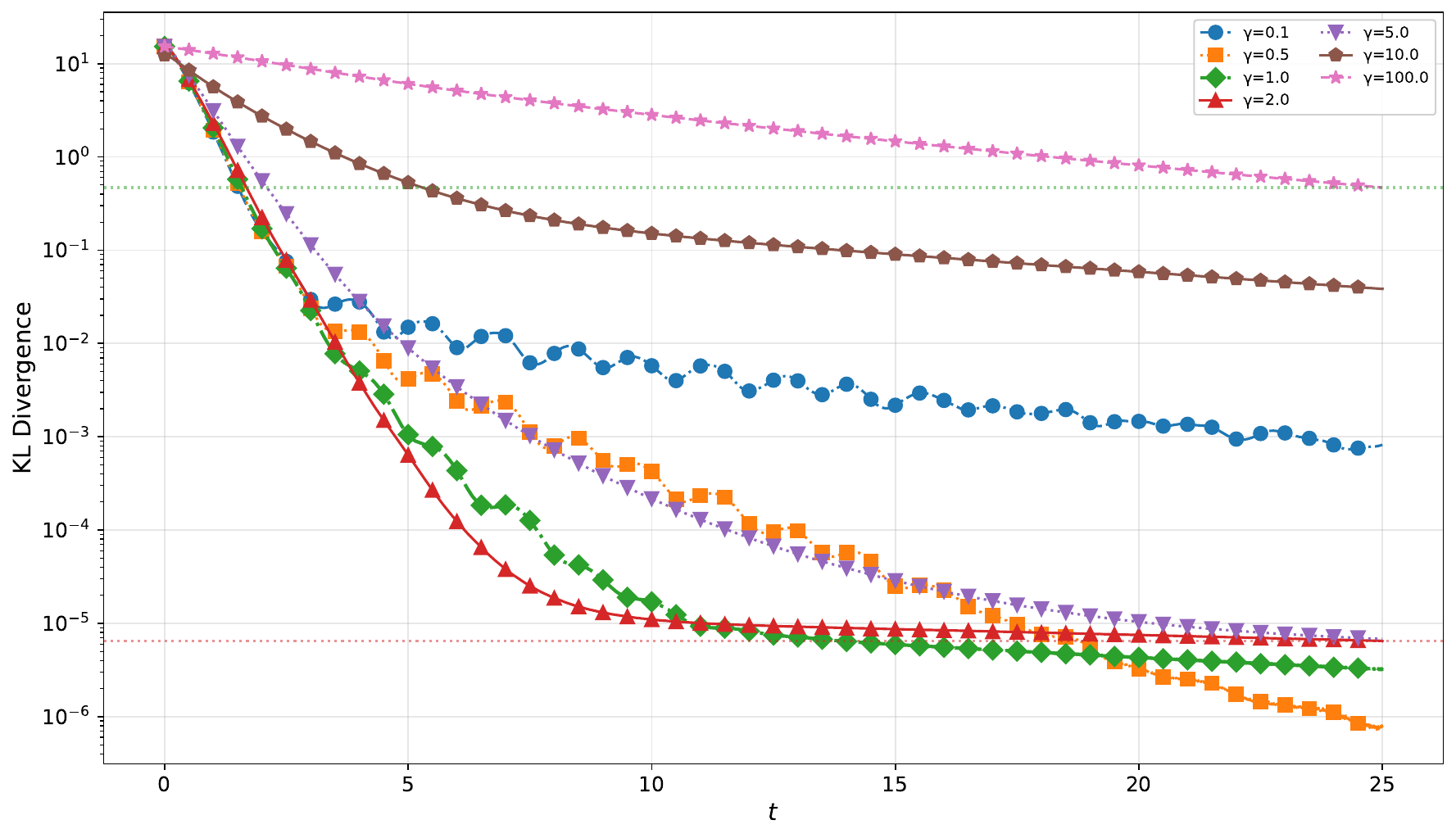}
		\caption{Damping dependence at \(\alpha=1.0\).}
		\label{fig:gamma_scan_KL}
	\end{subfigure}
	\hfill
	\begin{subfigure}[b]{0.31\textwidth}
		\centering
		\includegraphics[height=3.9cm,keepaspectratio]
		{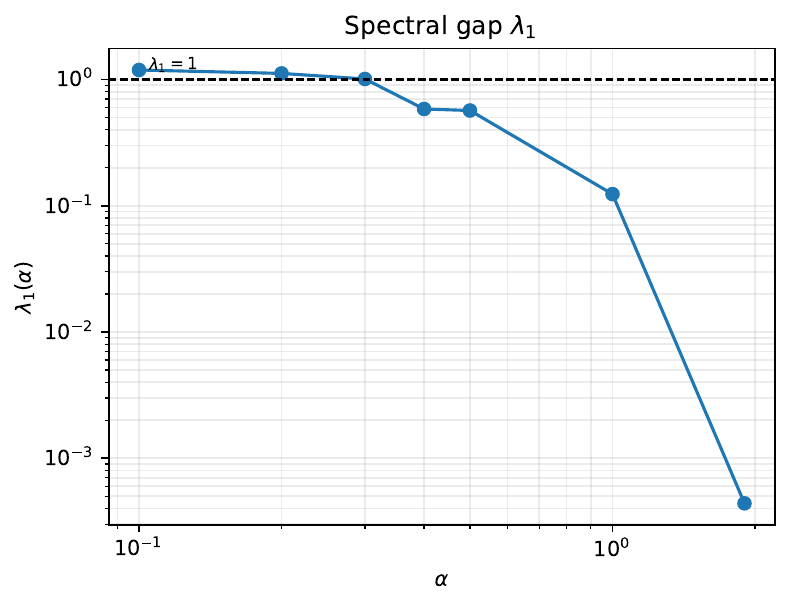}
		\caption{Spectral gap.}
		\label{fig:lambda1_alpha}
	\end{subfigure}
	\hfill
	\begin{subfigure}[b]{0.31\textwidth}
		\centering
		\includegraphics[height=3.9cm,keepaspectratio]
		{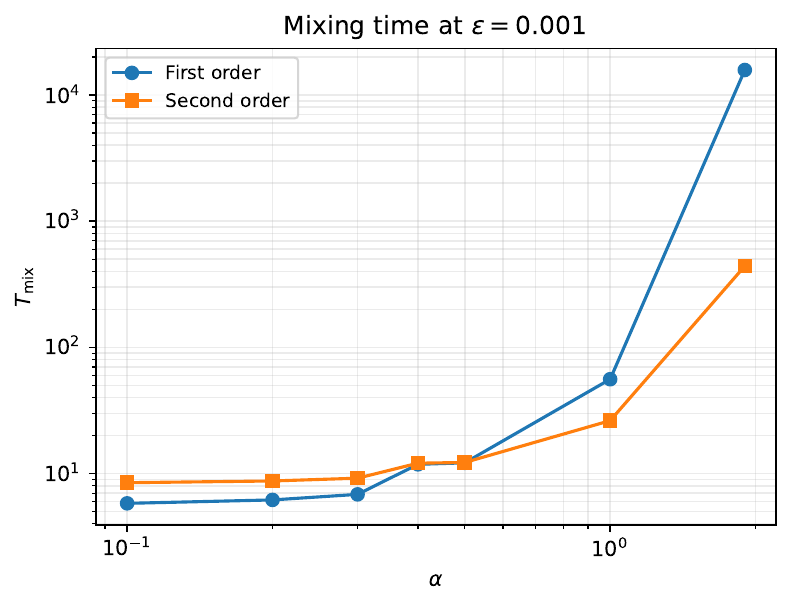}
		\caption{Predicted mixing time.}
		\label{fig:mixing_alpha_1e-3}
	\end{subfigure}
	
	\caption{
			Spectral and damping analysis of the nonlocal flow.
			(a) KL divergence of the second-order dynamics for different
			damping coefficients at \(\alpha=1.0\).
			(b) Spectral gap \(\lambda_1\) computed from the realized sparse graphs.
			(c) Corresponding linearized modal mixing-time predictions for the
			first- and second-order flows with \(\varepsilon=10^{-3}\).
	}
	\label{fig:spectral_damping_summary}
\end{figure}
\begin{figure}[t]
    \centering
    \begin{subfigure}[b]{0.24\linewidth}
        \centering
        \includegraphics[width=\linewidth]{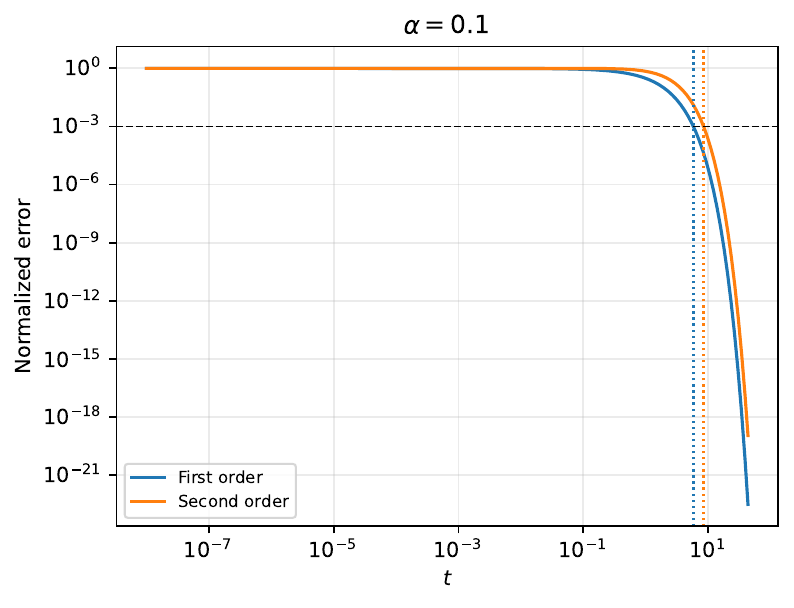}
        \caption{\(\alpha=0.1\)}
        \label{fig:alpha_relaxation_0.1}
    \end{subfigure}
    \hfill
    \begin{subfigure}[b]{0.24\linewidth}
        \centering
        \includegraphics[width=\linewidth]{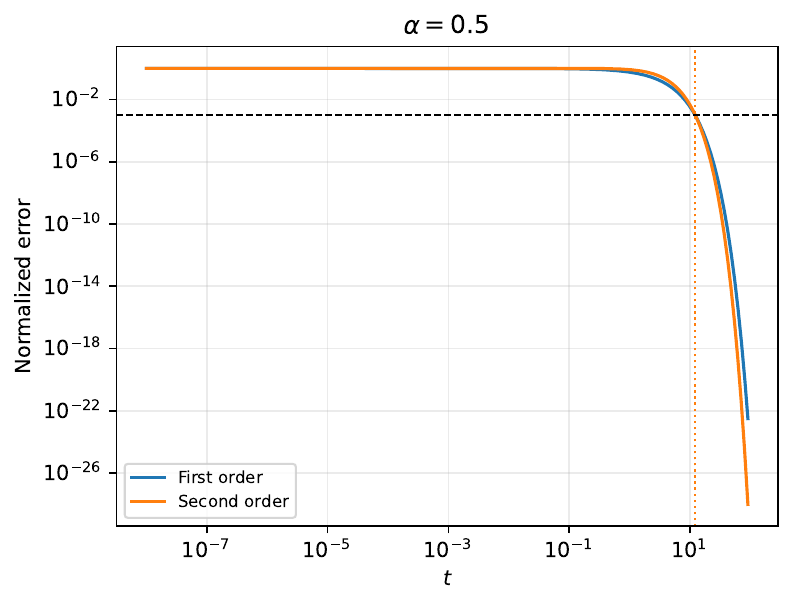}
        \caption{\(\alpha=0.5\)}
        \label{fig:alpha_relaxation_0.5}
    \end{subfigure}
    \hfill
    \begin{subfigure}[b]{0.24\linewidth}
        \centering
        \includegraphics[width=\linewidth]{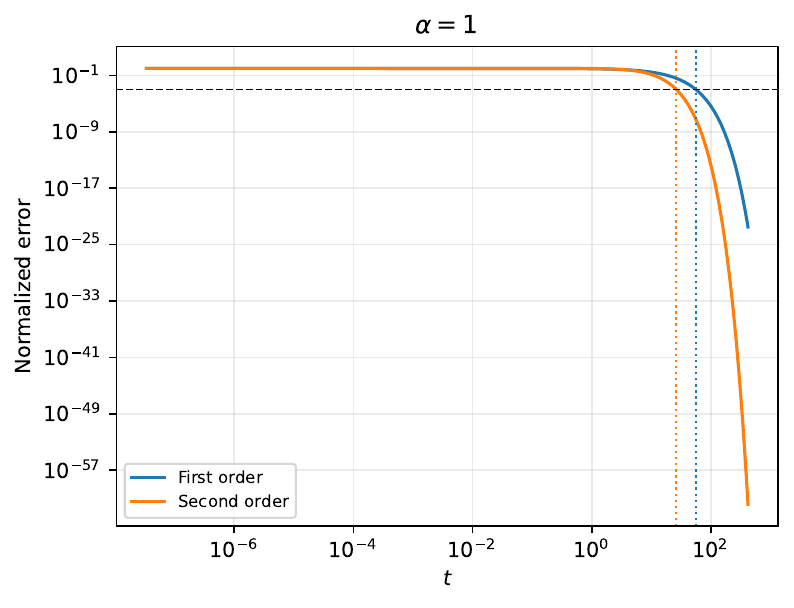}
        \caption{\(\alpha=1.0\)}
        \label{fig:alpha_relaxation_1.0}
    \end{subfigure}
    \hfill
    \begin{subfigure}[b]{0.24\linewidth}
        \centering
        \includegraphics[width=\linewidth]{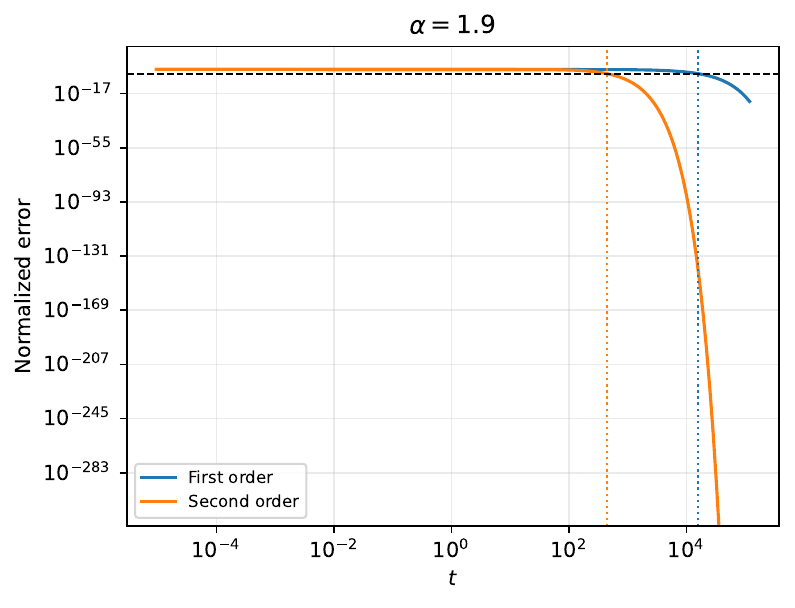}
        \caption{\(\alpha=1.9\)}
        \label{fig:alpha_relaxation_1.9}
    \end{subfigure}
    \caption{Linearized modal relaxation predictions obtained from
    the actual sparse-graph spectral gaps for
    \(\alpha=0.1,0.5,1.0,1.9\). The tolerance is
    \(\varepsilon=10^{-3}\), and vertical dotted lines denote the predicted
  settling times.}
    \label{fig:alpha_relaxation}
\end{figure}

Figure~\ref{fig:lambda1_alpha} shows that, for the present
discretization, \(\lambda_1\) increases as \(\alpha\) decreases,
consistent with the stronger long-range connectivity induced by smaller
\(\alpha\). Panel (c) shows the mixing-time trends predicted by the
linearized modal formulas, with the relative
asymptotic rates governed by the ratio \(1/\sqrt{\lambda_1}\).

Figure~\ref{fig:alpha_relaxation} shows the corresponding normalized
linearized modal relaxation predictions for four values of \(\alpha\). For smaller $\alpha$, the present discretization yields
$\lambda_1>1$ under the chosen time normalization, and the
first-order curve reaches the prescribed tolerance earlier. For larger $\alpha$, the spectral gap falls below unity, and
the second-order curve reaches the tolerance earlier,
consistent with the asymptotic rate comparison.

In the present discretization, stronger nonlocality is associated with a
larger spectral gap \(\lambda_1\), which accelerates the first-order
relaxation but reduces the relative asymptotic advantage of inertia.

(i) The L{\'e}vy exponent \(\alpha\) controls the long-range connectivity
of the graph and is reflected in the linearized spectral gap
\(\lambda_1\). Numerically, smaller \(\alpha\) is associated with a larger
\(\lambda_1\), consistent with the stronger long-range connectivity
induced by the heavier-tailed interaction kernel.

(ii) The linearized first-order and critically damped second-order
dynamics have asymptotic amplitude decay rates $\lambda_1$ and
$\sqrt{\lambda_1}$, respectively. Thus, under the present time
normalization, the second-order dynamics are asymptotically faster
when $\lambda_1<1$, with the leading-order mixing-time ratio
\(
\frac{T^{(1)}_{\mathrm{mix}}}{T^{(2)}_{\mathrm{mix}}}\sim\lambda_1^{-1/2}.
\)

(iii) In the present experiments, stronger nonlocality is associated
with a larger spectral gap and hence a smaller relative asymptotic
advantage of the second-order dynamics.
These observations are specific to the present discretization and
normalization.

\subsection{Comparison with Other Samplers}

In this section, we compare the proposed method with baseline samplers, including MCMC, MALA, and ULD, on three benchmark distributions. These targets allow us to evaluate the method's ability to capture manifold structures, traverse curved geometries, and explore isolated modes.

\begin{figure}[htbp]
	\centering
	\begin{subfigure}[b]{0.18\textwidth}
		\centering
		\includegraphics[width=\textwidth]{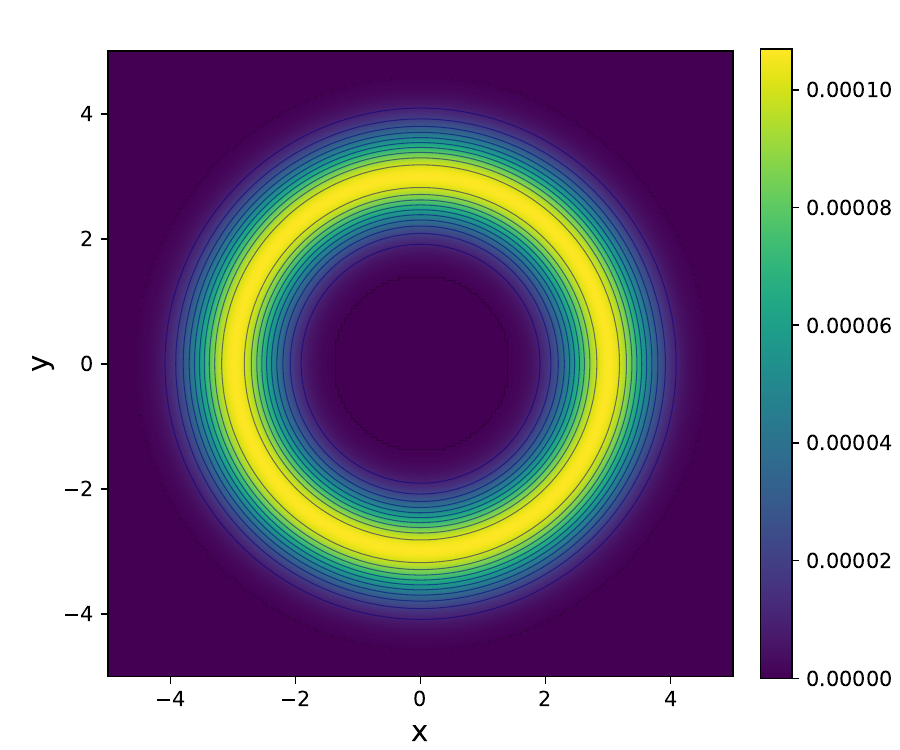}
		\caption{Target}
	\end{subfigure}
	\hfill
	\begin{subfigure}[b]{0.18\textwidth}
		\centering
		\includegraphics[width=\textwidth]{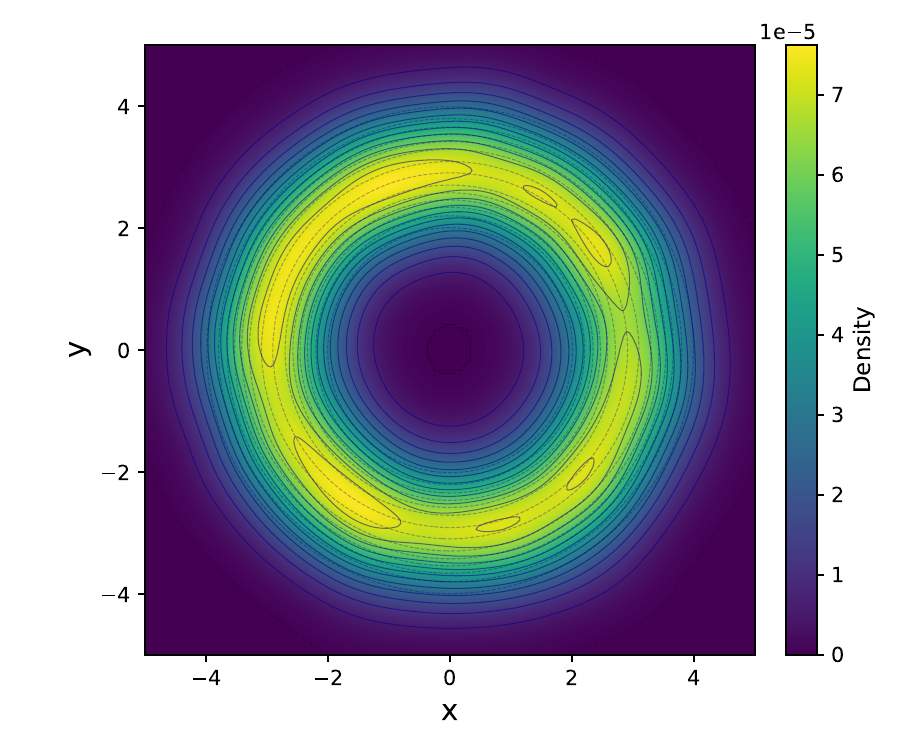}
		\caption{Our method.}
	\end{subfigure}
	\hfill
	\begin{subfigure}[b]{0.18\textwidth}
		\centering
		\includegraphics[width=\textwidth]{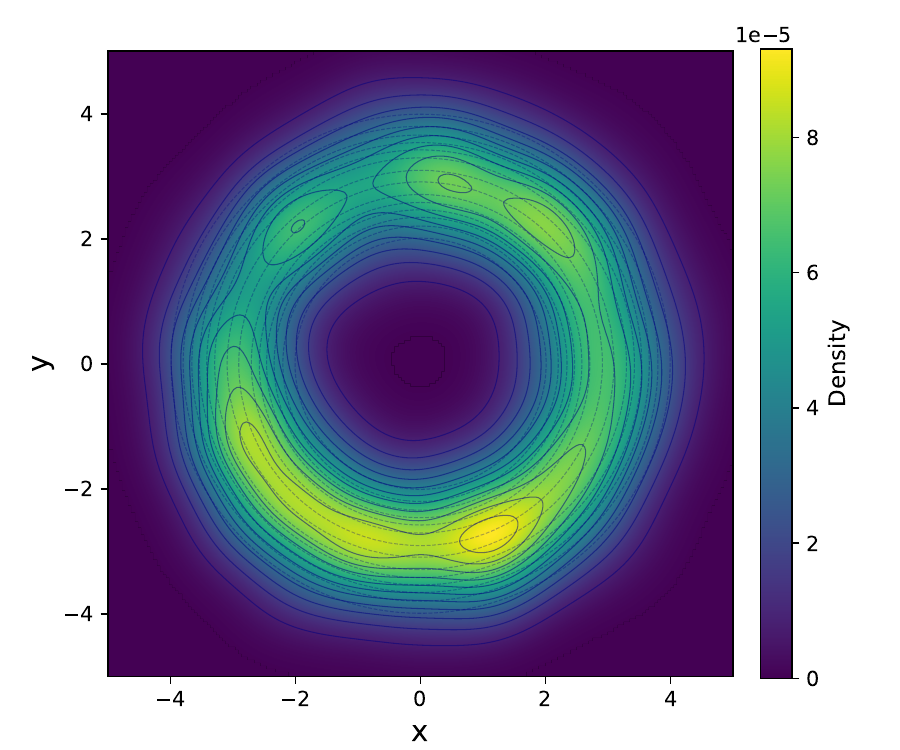}
		\caption{MALA}
	\end{subfigure}
	\hfill
	\begin{subfigure}[b]{0.18\textwidth}
		\centering
		\includegraphics[width=\textwidth]{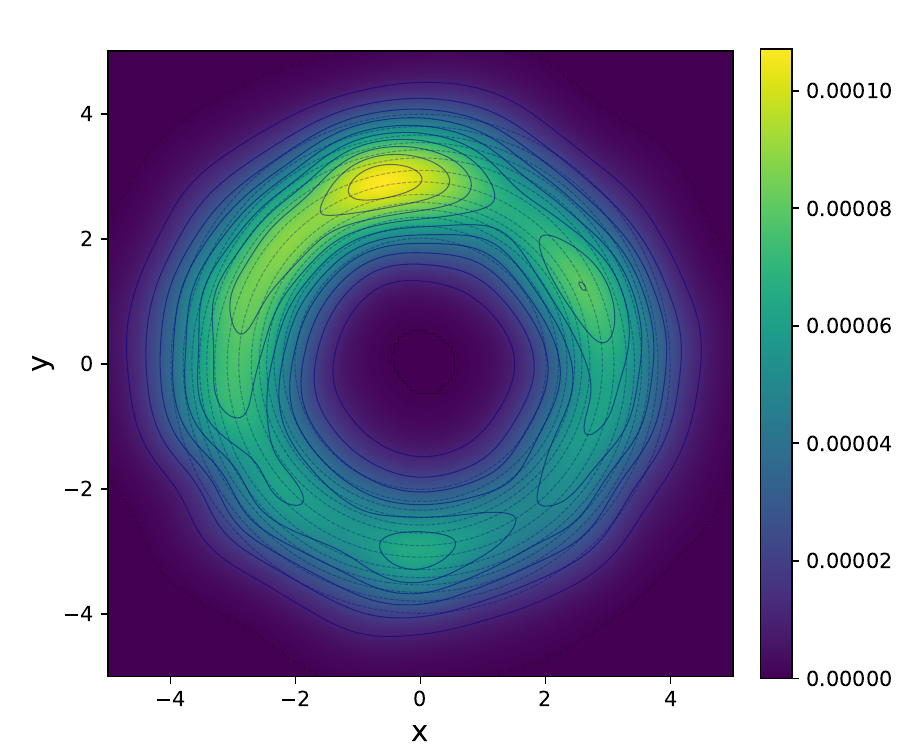}
		\caption{MCMC}
	\end{subfigure}
	\hfill
	\begin{subfigure}[b]{0.18\textwidth}
		\centering
		\includegraphics[width=\textwidth]{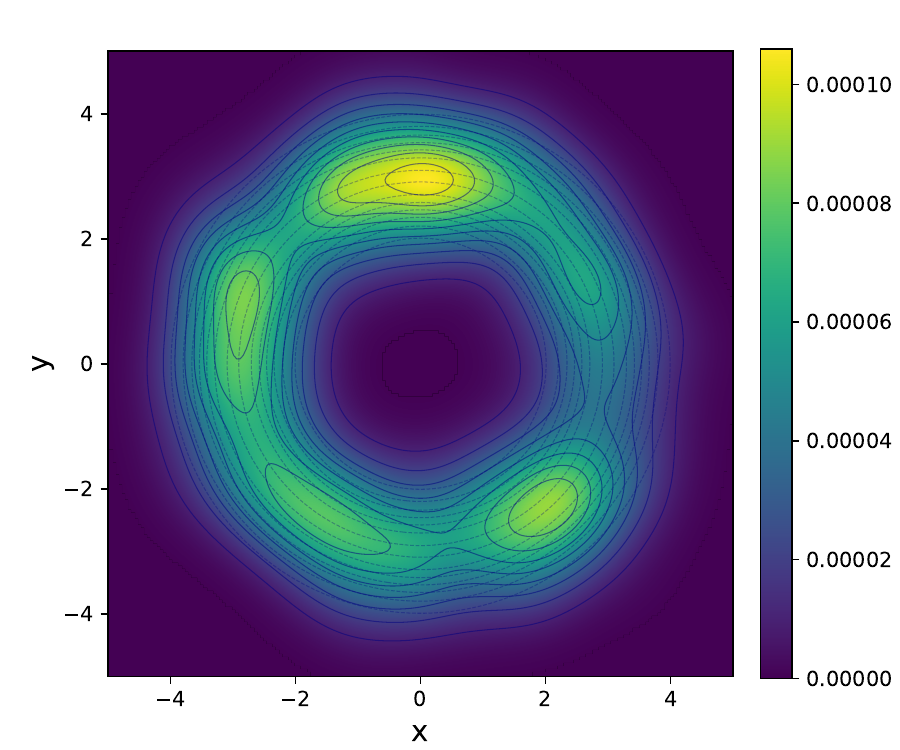}
		\caption{ULD}
	\end{subfigure}
	\caption{Sampling results (KDE) on the single ring distribution $\pi \propto e^{-2(r-3)^2}$. The proposed method ($\alpha=1.0$) provides broader coverage than the tested baselines in this experiment.}
	\label{fig:single_ring_comparison}
\end{figure}

\begin{figure}[htbp]
	\centering
	\begin{subfigure}[b]{0.18\textwidth}
		\centering
		\includegraphics[width=\textwidth]{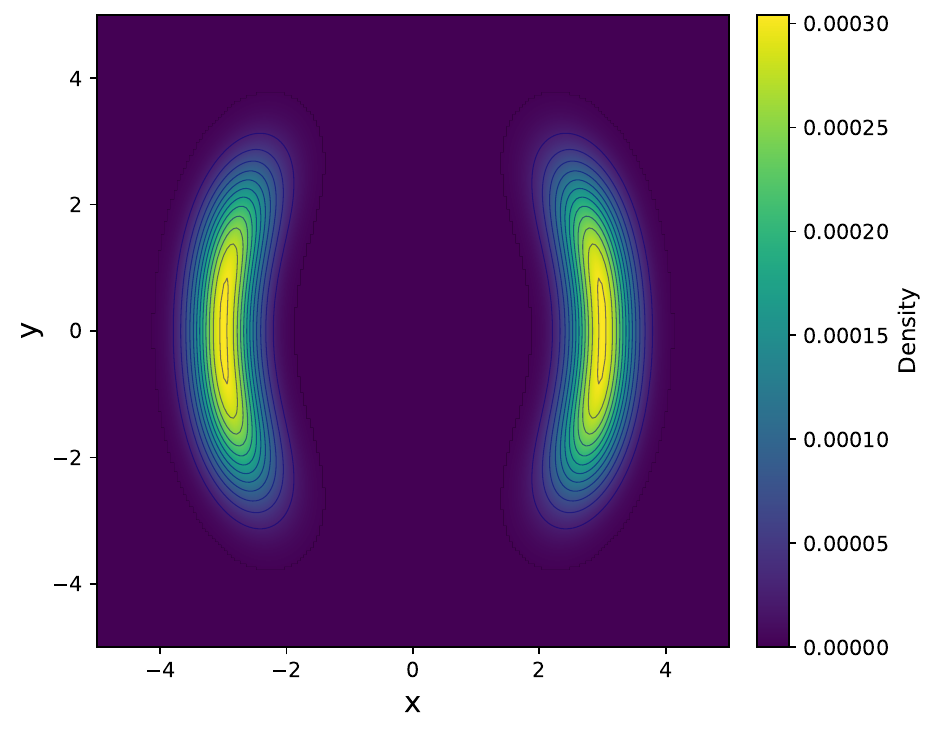}
		\caption{Target}
	\end{subfigure}
	\hfill
	\begin{subfigure}[b]{0.18\textwidth}
		\centering
		\includegraphics[width=\textwidth]{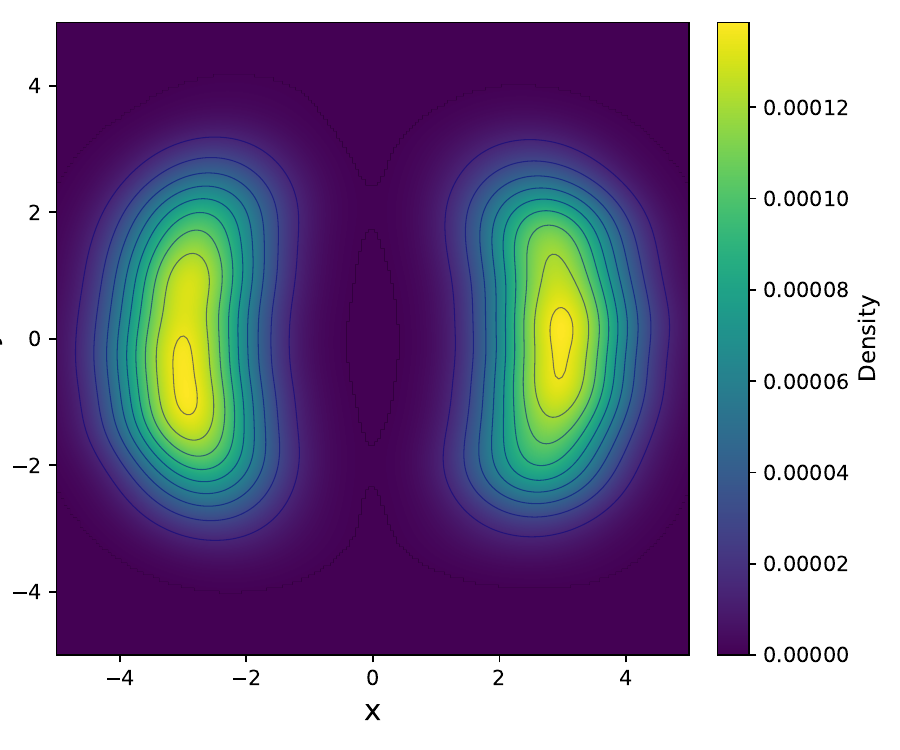}
		\caption{Our method.}
	\end{subfigure}
	\hfill
	\begin{subfigure}[b]{0.18\textwidth}
		\centering
		\includegraphics[width=\textwidth]{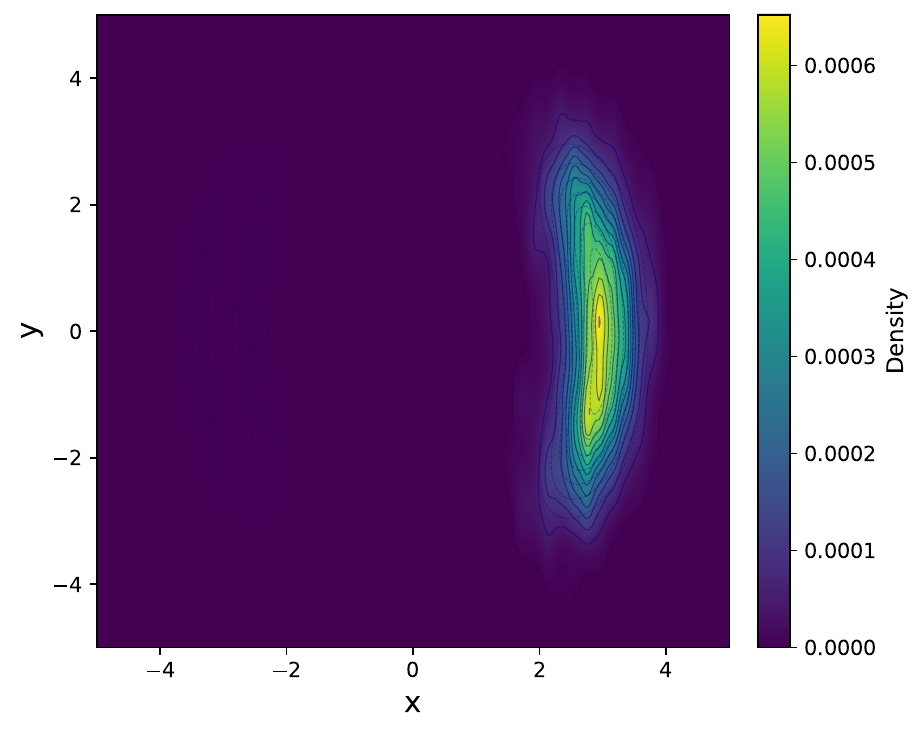}
		\caption{MALA}
	\end{subfigure}
	\hfill
	\begin{subfigure}[b]{0.18\textwidth}
		\centering
		\includegraphics[width=\textwidth]{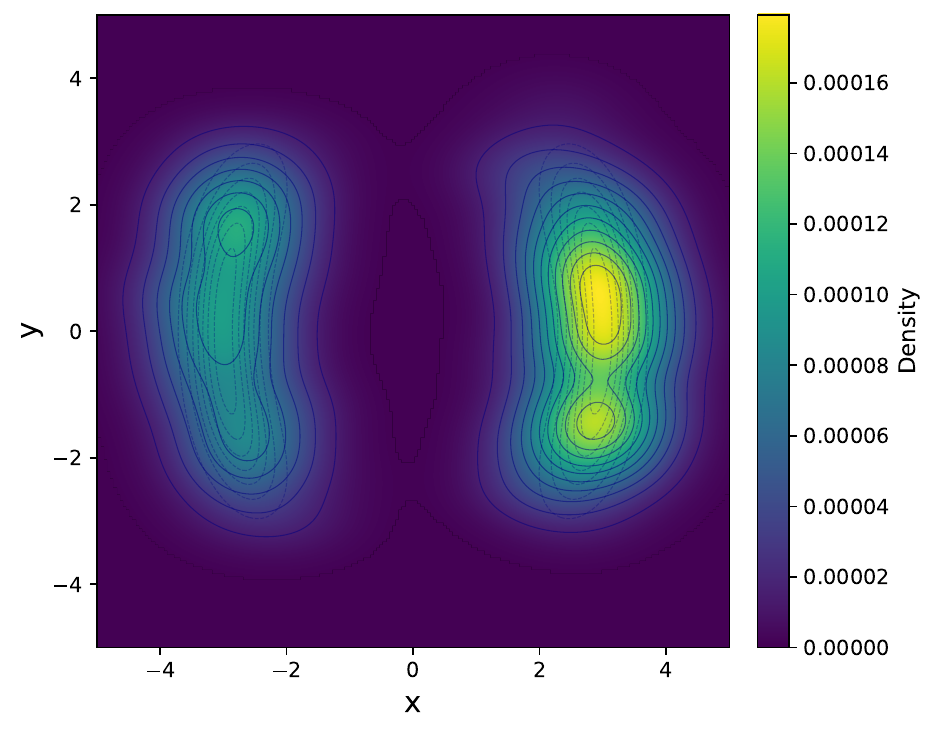}
		\caption{MCMC}
	\end{subfigure}
	\hfill
	\begin{subfigure}[b]{0.18\textwidth}
		\centering
		\includegraphics[width=\textwidth]{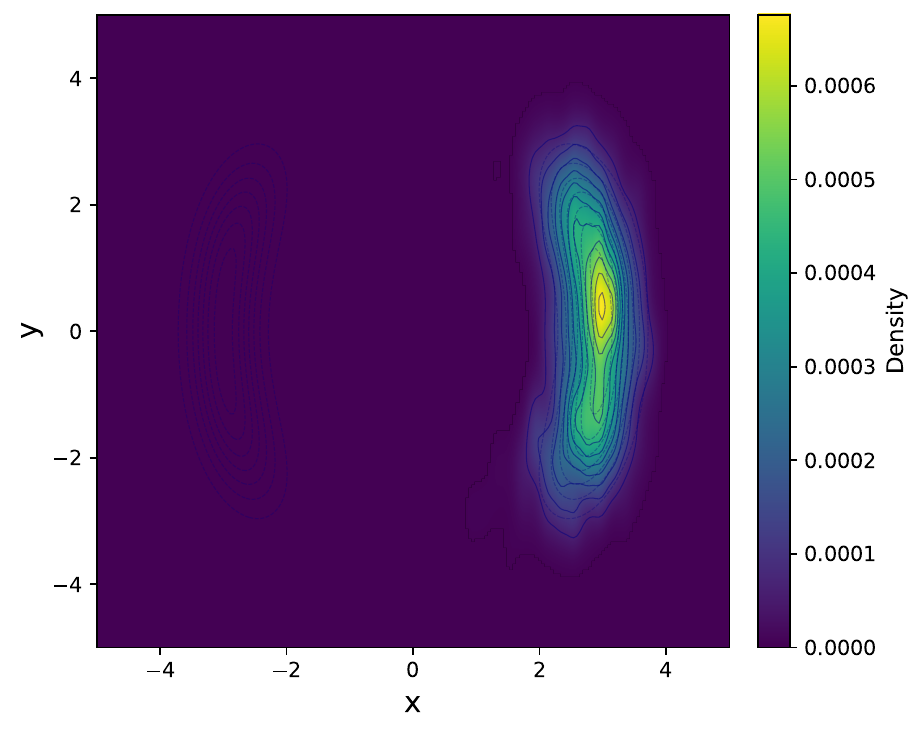}
		\caption{ULD}
	\end{subfigure}
	
	\caption{ Sampling results (KDE) on the Banana distribution. The proposed method $(\alpha=1.0)$ provides broader coverage than the tested baselines in this experiment.}
	\label{fig:banana_comparison}
\end{figure}

\begin{figure}[htbp]
	\centering
	\begin{subfigure}[b]{0.18\textwidth}
		\centering
		\includegraphics[width=\textwidth]{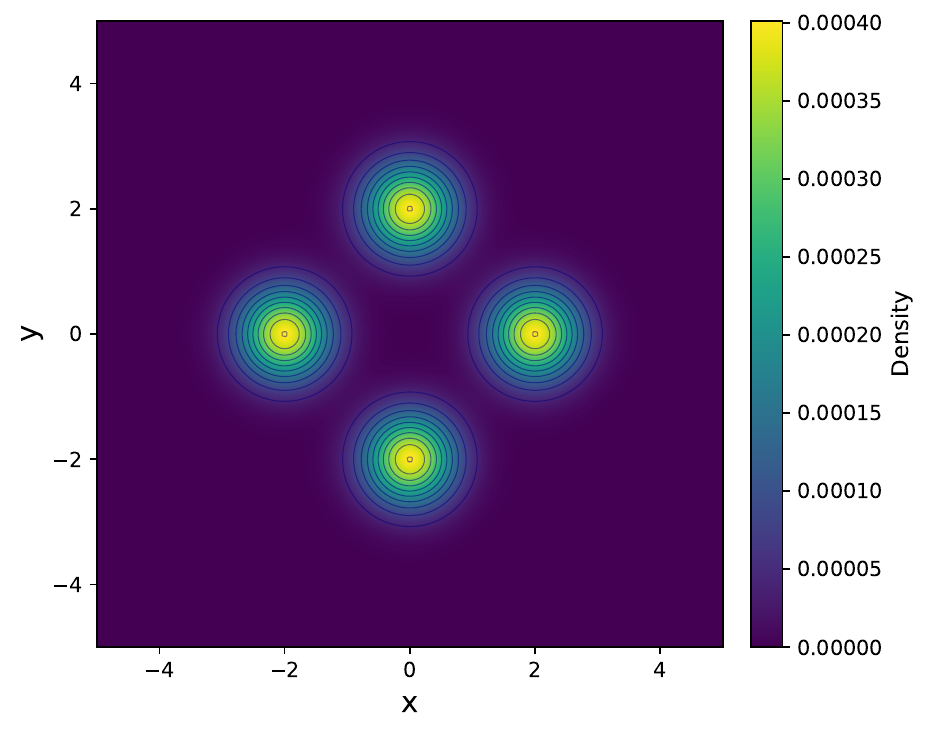}
		\caption{Target}
	\end{subfigure}
	\hfill
	\begin{subfigure}[b]{0.18\textwidth}
		\centering
		\includegraphics[width=\textwidth]{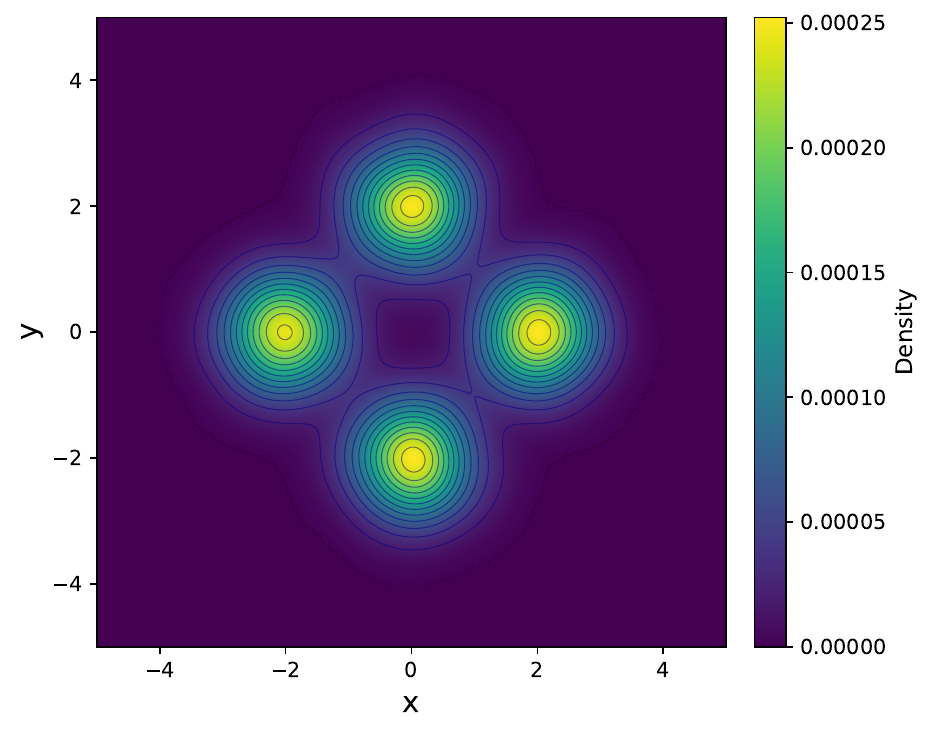}
		\caption{Our method.}
	\end{subfigure}
	\hfill
	\begin{subfigure}[b]{0.18\textwidth}
		\centering
		\includegraphics[width=\textwidth]{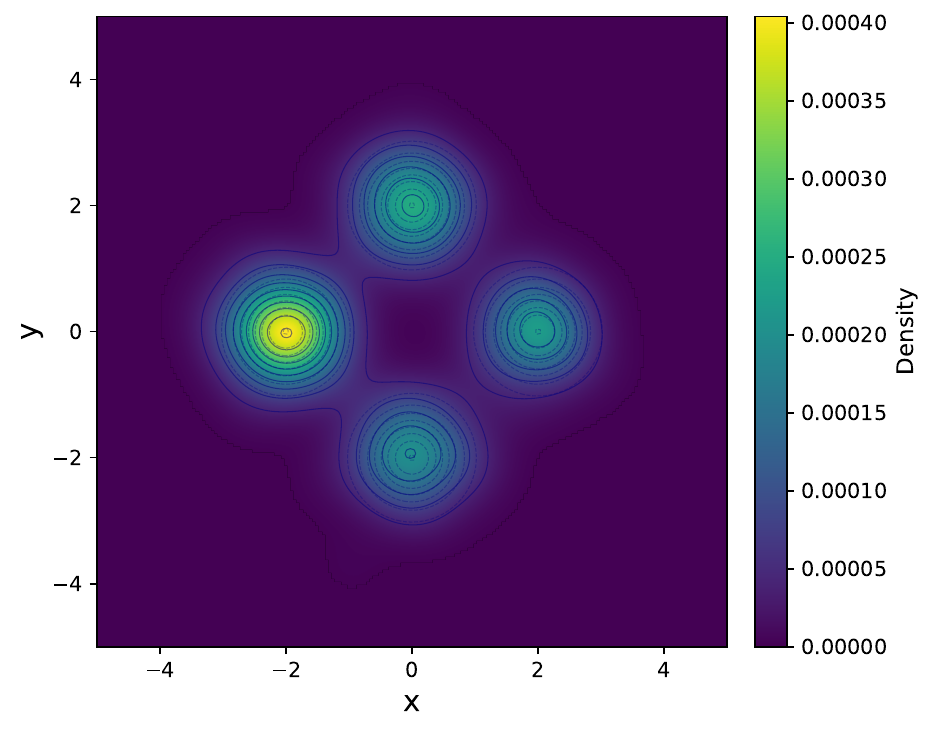}
		\caption{MALA}
	\end{subfigure}
	\hfill
	\begin{subfigure}[b]{0.18\textwidth}
		\centering
		\includegraphics[width=\textwidth]{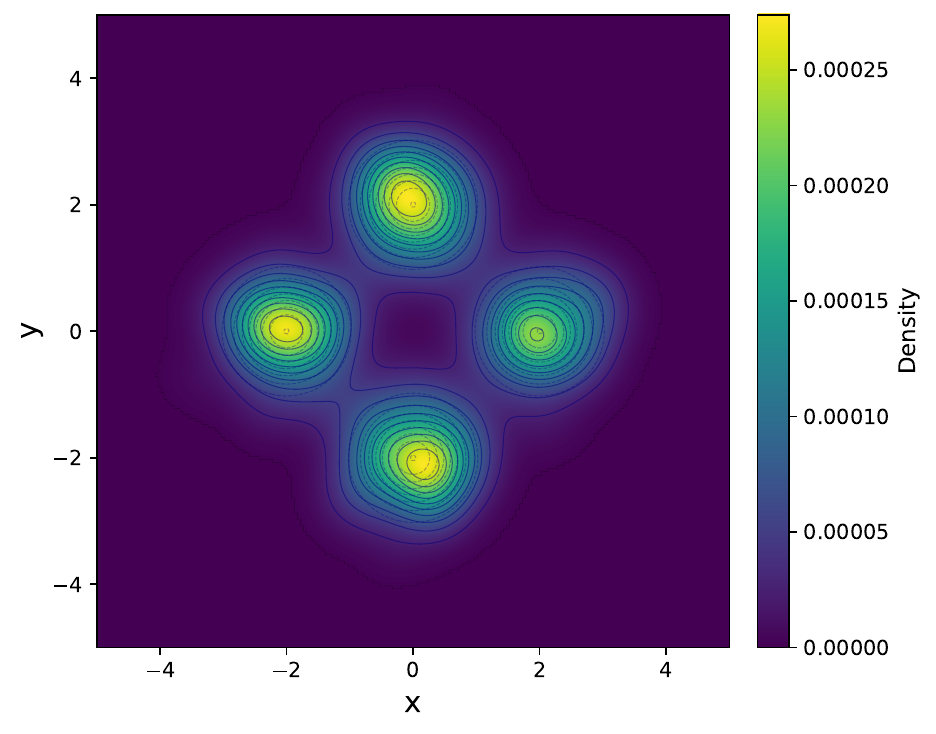}
		\caption{MCMC}
	\end{subfigure}
	\hfill
	\begin{subfigure}[b]{0.18\textwidth}
		\centering
		\includegraphics[width=\textwidth]{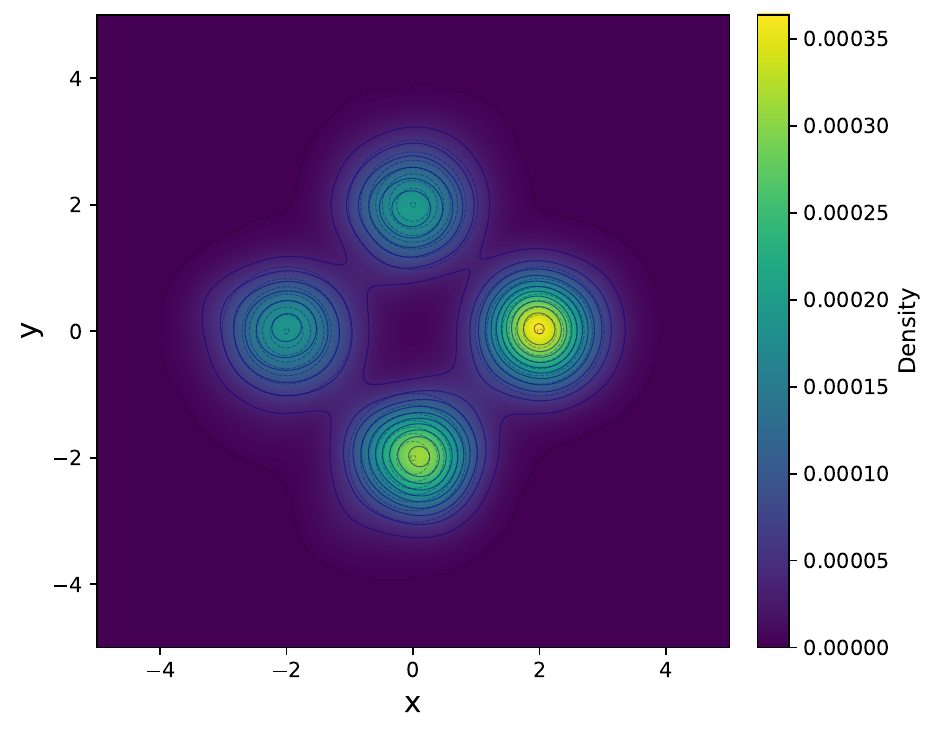}
		\caption{ULD}
	\end{subfigure}
	\caption{Sampling results (KDE) on the four-Gaussian-mixture distribution $\pi = \frac{1}{4}\sum \mathcal{N}(\mu_i, 0.25I_2)$. The proposed method $(\alpha=1.0)$ provides broader coverage than the tested baselines in this experiment.}
	\label{fig:four_peaks_comparison}
\end{figure}

Under the reported computational budget, our method achieves better mode coverage than the tested baselines across the three test distributions. On the single ring target (Figure~\ref{fig:single_ring_comparison}), 
MALA and ULD remain trapped in local regions, while our method uniformly covers the entire manifold. The Banana distribution (Figure~\ref{fig:banana_comparison}) reveals a similar pattern: the baselines fail to capture the curved tails, whereas our method recovers the full shape. The most significant improvement appears on the four-Gaussian-mixture (Figure~\ref{fig:four_peaks_comparison}), where our method explores all four modes with balanced weights, in contrast to the mode collapse or imbalanced coverage observed in the baselines.

\paragraph{Equal-wall-clock comparison.}
To complement the per-step complexity analysis, we compare the methods
	under a common wall-clock budget on the four-Gaussian-mixture target.
	This experiment measures end-to-end computational efficiency, including
	the one-time sparse-graph construction cost of the proposed method.	
\begin{figure}[t]
	\centering
	\begin{subfigure}[b]{0.48\textwidth}
		\centering
		\includegraphics[width=\textwidth]{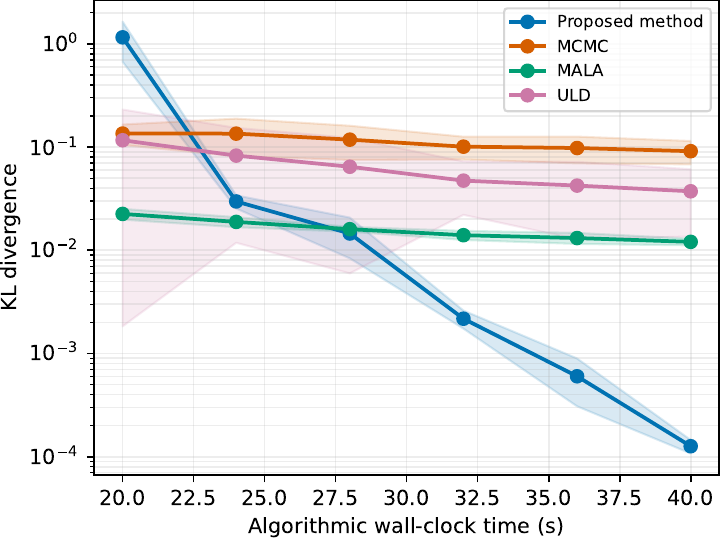}
		\caption{KL divergence.}
		\label{fig:wallclock_kl}
	\end{subfigure}
	\hfill
	\begin{subfigure}[b]{0.48\textwidth}
		\centering
		\includegraphics[width=\textwidth]{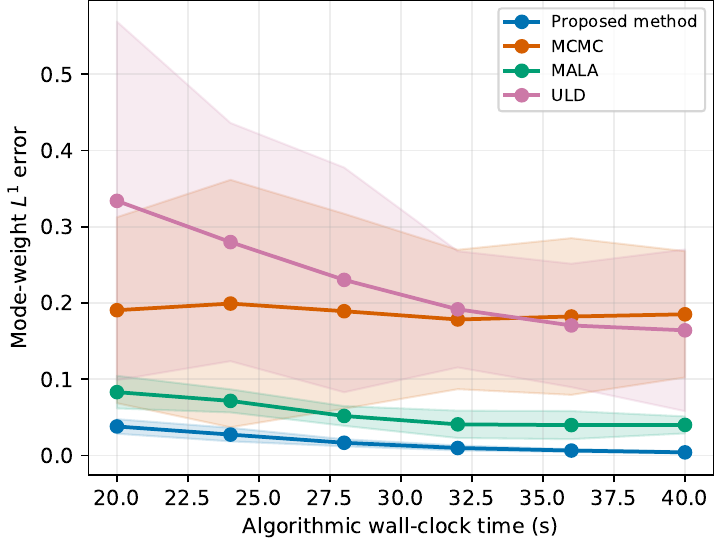}
		\caption{Mode-weight error.}
		\label{fig:wallclock_mode}
	\end{subfigure}
	\caption{
			Equal-wall-clock comparison on the four-Gaussian-mixture target.
			All methods use a \(40\)-s computational budget; curves show the mean
			over five independent seeds and shaded regions denote one standard
			deviation. The proposed-method timing includes its one-time graph setup,
			while the particle-baseline timing includes warm-up.
			(a) KL divergence. (b) Mode-weight \(L^1\) error.
		}
	\label{fig:wallclock_comparison}
\end{figure}

For the wall-clock comparison, all methods are run under the same
	\(40\)-s computational budget, with results averaged over five
	independent seeds. For the proposed method, the reported time includes
	sparse-graph construction and subsequent density evolution, whereas
	for the particle baselines it includes both warm-up and production.
	Detailed timing and implementation settings are provided in
	Appendix~\ref{app:numerical_details}.

Figure~\ref{fig:wallclock_comparison} shows the end-to-end wall-clock
comparison on the four-Gaussian-mixture target. The proposed method
incurs a one-time sparse-graph construction cost, which leads to weaker
initial wall-clock performance. As this setup cost is amortized, its
performance improves relative to the tested particle samplers within
the common computational budget. By the end of the \(40\)-s run, the
proposed method exhibits favorable KL and mode-weight-error trends
among the compared methods.

\subsubsection{Mode Exploration and Collapse Prevention}

In this subsection, we quantify the mode coverage of the proposed method on two multimodal distributions: four-Gaussian-mixture (four modes) and banana (two wells). The target occupancy ratios are $25\%$ per mode and $50\%$ per well, respectively.
We compute the occupancy ratios by averaging over the last 5,000 iterations of each run.

\begin{figure}[htbp]
\centering
\begin{subfigure}[b]{0.48\textwidth}
\centering
\includegraphics[width=\textwidth]{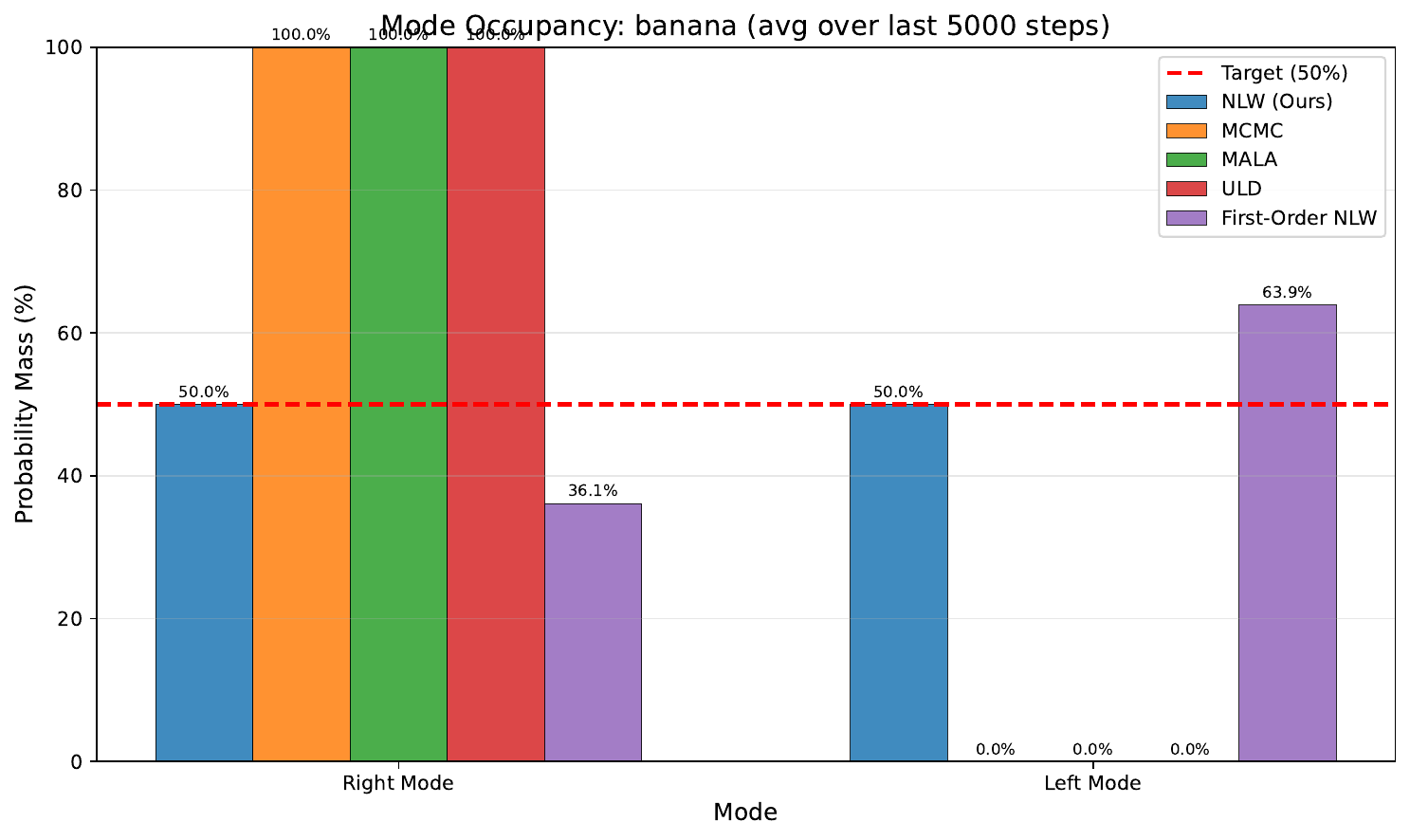}
\caption{Banana}
\label{fig:ring_double_well_occupancy}
\end{subfigure}
\hfill
\begin{subfigure}[b]{0.48\textwidth}
\centering
\includegraphics[width=\textwidth]{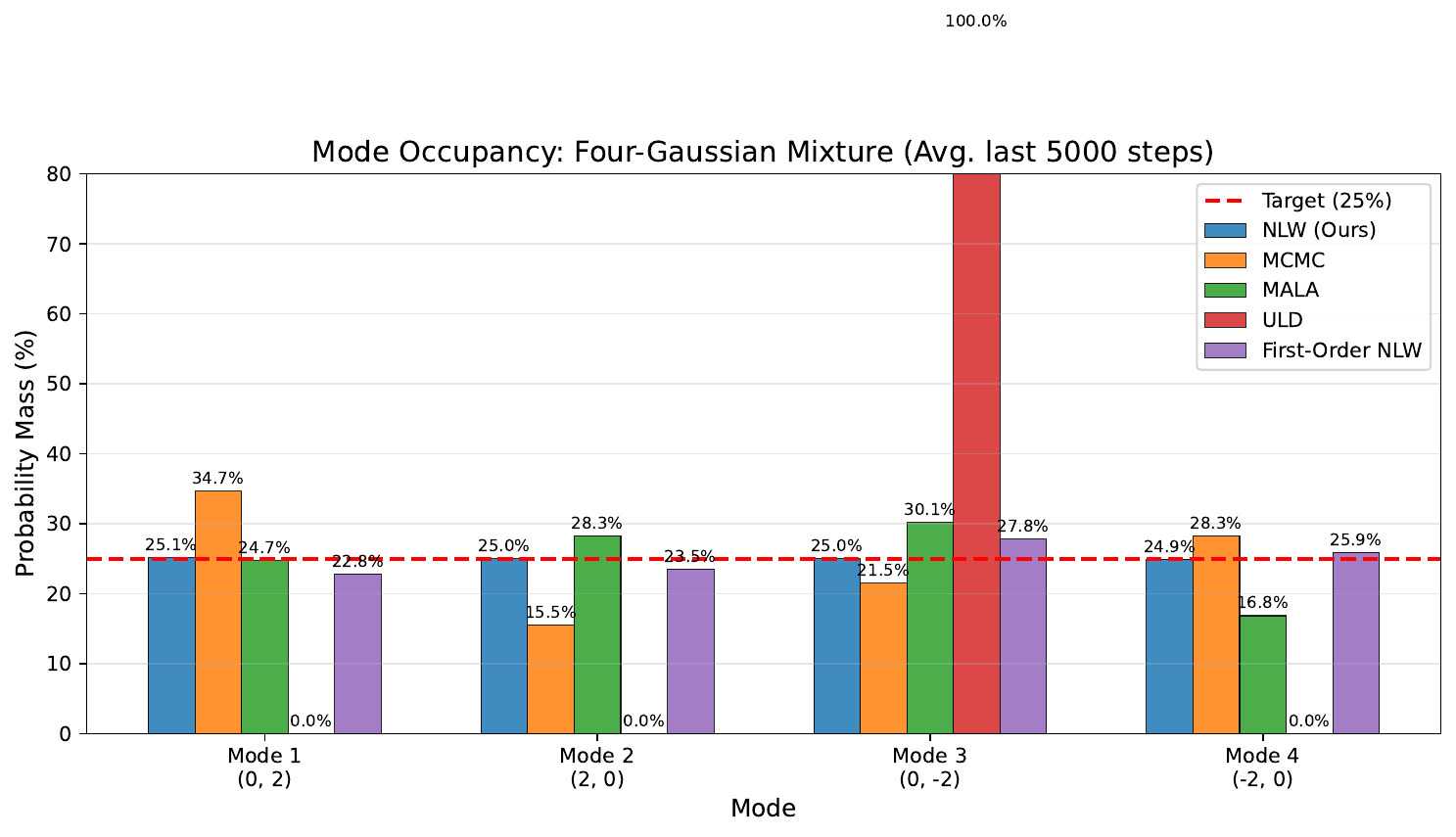}
\caption{Four-Gaussian-mixture}
\label{fig:four_peaks_occupancy}
\end{subfigure}
\caption{Mode occupancy comparison. Our method ($\alpha=1.0$) achieves balanced sampling in both cases.}
\label{fig:mode_occupancy_comparison}
\end{figure}

\begin{table}[htbp]
\centering
\caption{Mode occupancy ratios (\%) on Four-Gaussian-Mixture and banana.}
\label{tab:mode_occupancy}
\begin{tabular}{lcccccc}
\toprule
 & \multicolumn{4}{c}{Four-Gaussian-mixture}
 & \multicolumn{2}{c}{Banana} \\
\cmidrule{2-5} \cmidrule{6-7}
Method & M1 & M2 & M3 & M4 & Right & Left \\
\midrule
Target
& 25.0 & 25.0 & 25.0 & 25.0 & 50.0 & 50.0 \\

Proposed method ($\alpha=1.0$)
& 25.1 & 25.0 & 25.0 & 24.9 & 50.0 & 50.0 \\

MCMC
& 34.68 & 15.54 & 21.50 & 28.28 & 100.0 & 0.0 \\

MALA
& 24.74 & 28.28 & 30.14 & 16.84 & 100.0 & 0.0 \\

ULD
& 0.00 & 0.00 & 100.00 & 0.00 & 100.0 & 0.0 \\

First-order flow
& 22.8 & 23.5 & 27.8 & 25.9 & 36.1 & 63.9 \\
\bottomrule
\end{tabular}
\end{table}
As shown in Table~\ref{tab:mode_occupancy} and
Figure~\ref{fig:mode_occupancy_comparison}, the proposed method
maintains nearly balanced mode occupancy on both targets.
For the four-Gaussian-mixture, MCMC and MALA visit all four modes but
exhibit noticeable imbalance, whereas ULD remains concentrated in a
single mode over the reported occupancy window.
For the banana target, all three particle baselines remain in the same
well during the final \(5{,}000\) production samples, while the
first-order nonlocal flow visits both wells but with an imbalanced
occupancy. In contrast, the proposed second-order nonlocal flow achieves
approximately equal occupancy of the two wells. This aligns with the known difficulty of local gradient dynamics in crossing low-density barriers~\cite{arjovsky2017towards} and the tendency of likelihood-seeking objectives to concentrate mass in high-density regions~\cite{wang2024taming}. 

\subsection{Experiments on Real-World Data}

In addition to synthetic distributions, we further validate our method on two real-world problems. The first is 
a two-dimensional PCA representation of MNIST with multimodal cluster structure. The second is the alanine dipeptide conformational 
space, a benchmark molecular system with a rugged free-energy landscape 
and multiple barriers.
\begin{figure}[htbp]
    \centering
    
    \begin{subfigure}[b]{0.15\textwidth}
        \centering
        \includegraphics[width=\textwidth]{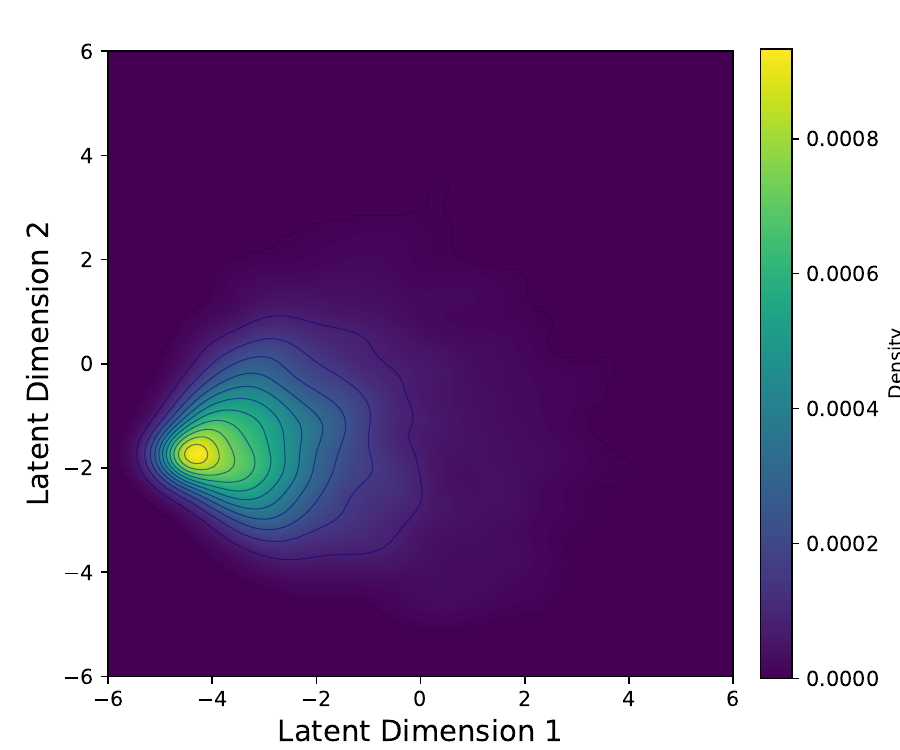}
        \caption{Target}
        \label{fig:mnist_target_kde}
    \end{subfigure}
    \hfill
    \begin{subfigure}[b]{0.15\textwidth}
        \centering
        \includegraphics[width=\textwidth]{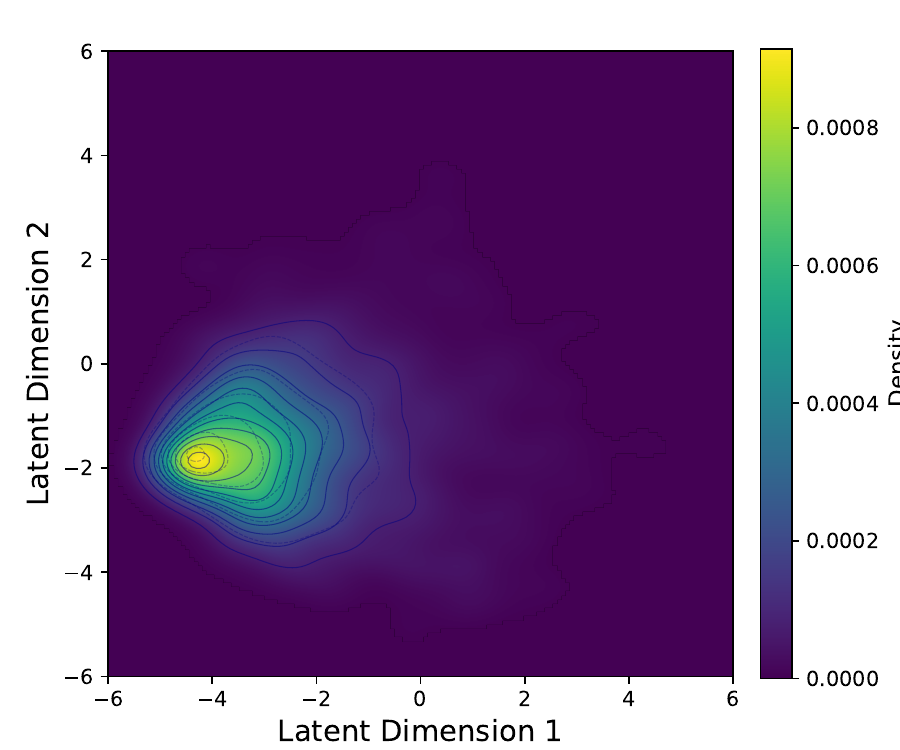}
        \caption{$\alpha=0.3$}
        \label{fig:mnist_nlw_alpha0.3_kde}
    \end{subfigure}
    \hfill
    \begin{subfigure}[b]{0.15\textwidth}
        \centering
        \includegraphics[width=\textwidth]{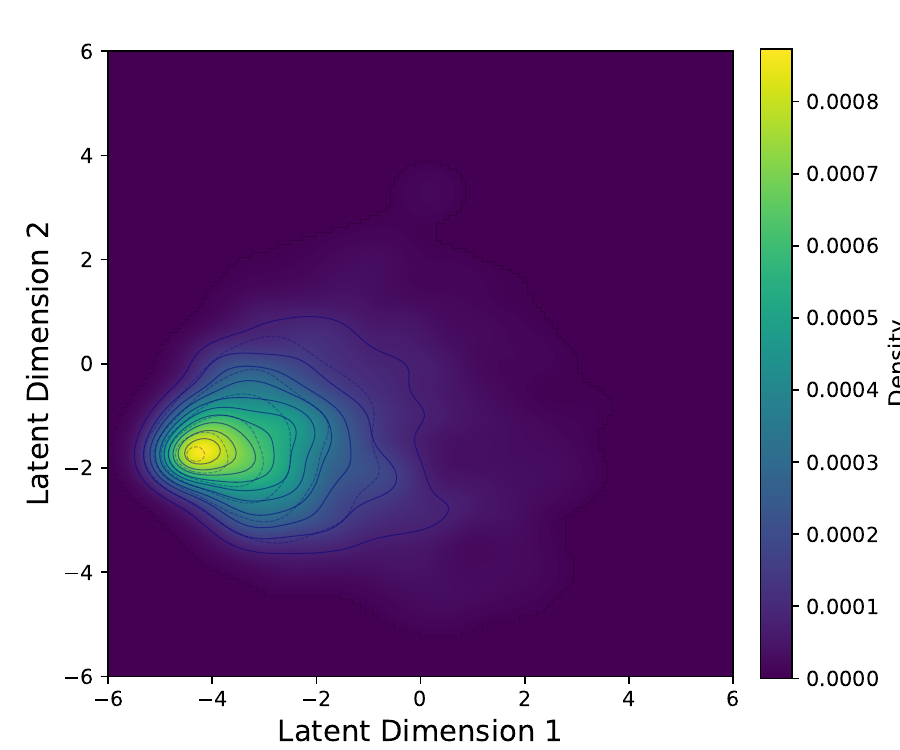}
        \caption{$\alpha=0.5$}
        \label{fig:mnist_nlw_alpha0.5_kde}
    \end{subfigure}
    \hfill
    \begin{subfigure}[b]{0.15\textwidth}
        \centering
        \includegraphics[width=\textwidth]{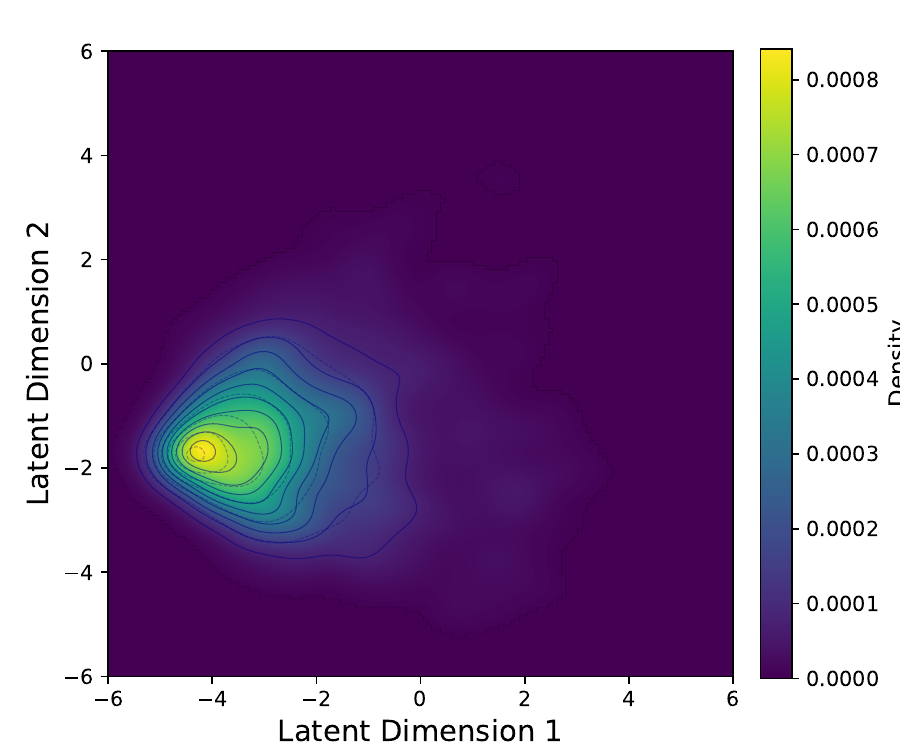}
        \caption{$\alpha=1.0$}
        \label{fig:mnist_nlw_alpha1.0_kde}
    \end{subfigure}
    \hfill
    \begin{subfigure}[b]{0.15\textwidth}
        \centering
        \includegraphics[width=\textwidth]{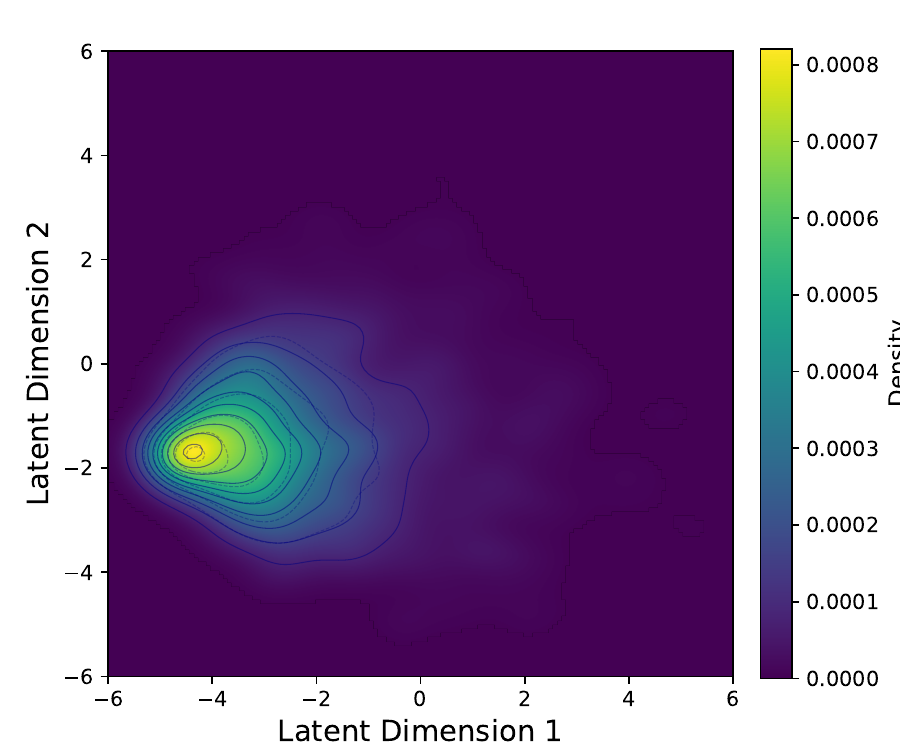}
        \caption{$\alpha=1.5$}
        \label{fig:mnist_nlw_alpha1.5_kde}
    \end{subfigure}
    
    \par\vspace{0.3cm}
    
    \begin{subfigure}[b]{0.15\textwidth}
        \centering
        \includegraphics[width=\textwidth]{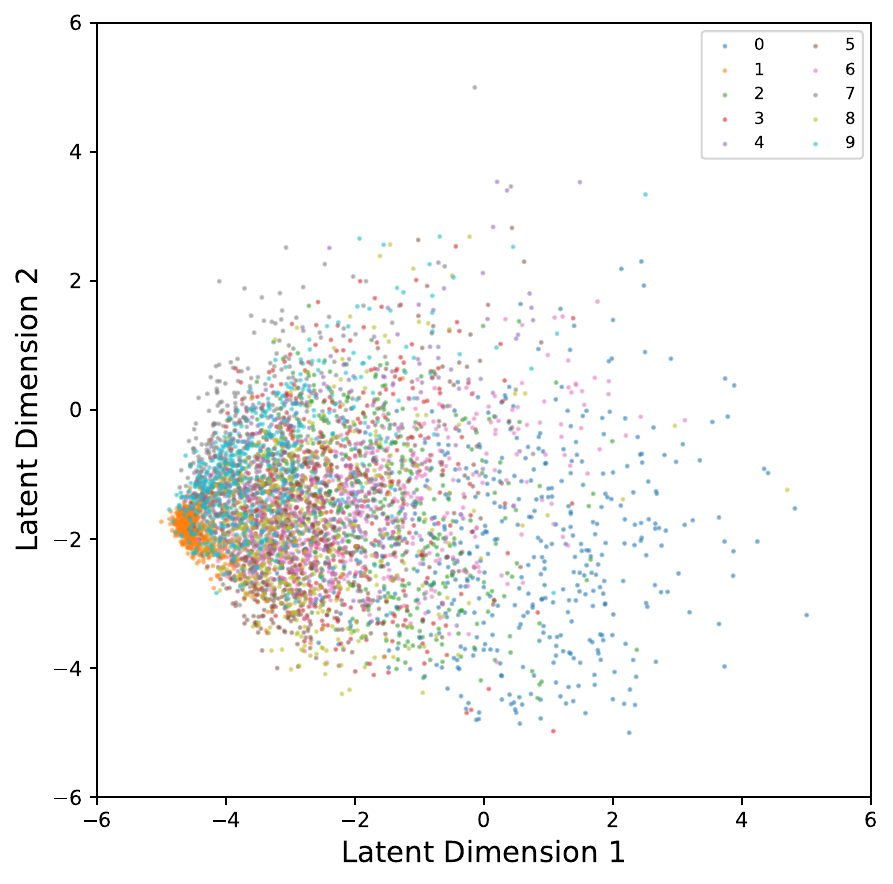}
        \caption{Target}
        \label{fig:mnist_target_samples}
    \end{subfigure}
    \hfill
    \begin{subfigure}[b]{0.15\textwidth}
        \centering
        \includegraphics[width=\textwidth]{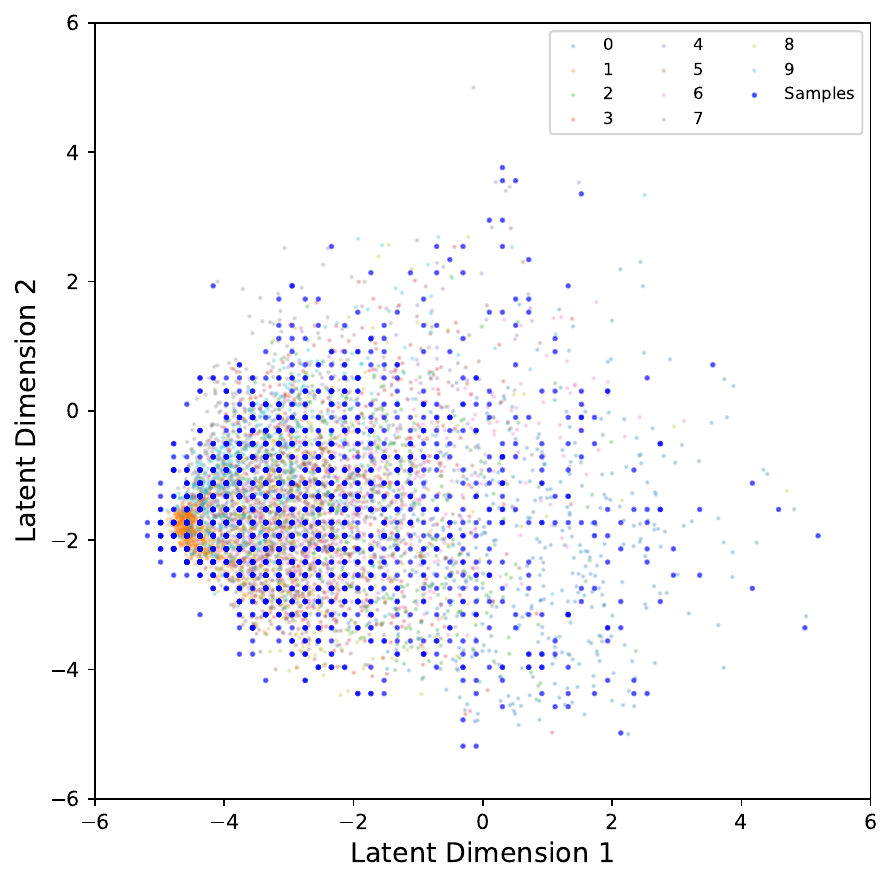}
        \caption{$\alpha=0.3$}
        \label{fig:mnist_nlw_alpha0.3_samples}
    \end{subfigure}
    \hfill
    \begin{subfigure}[b]{0.15\textwidth}
        \centering
        \includegraphics[width=\textwidth]{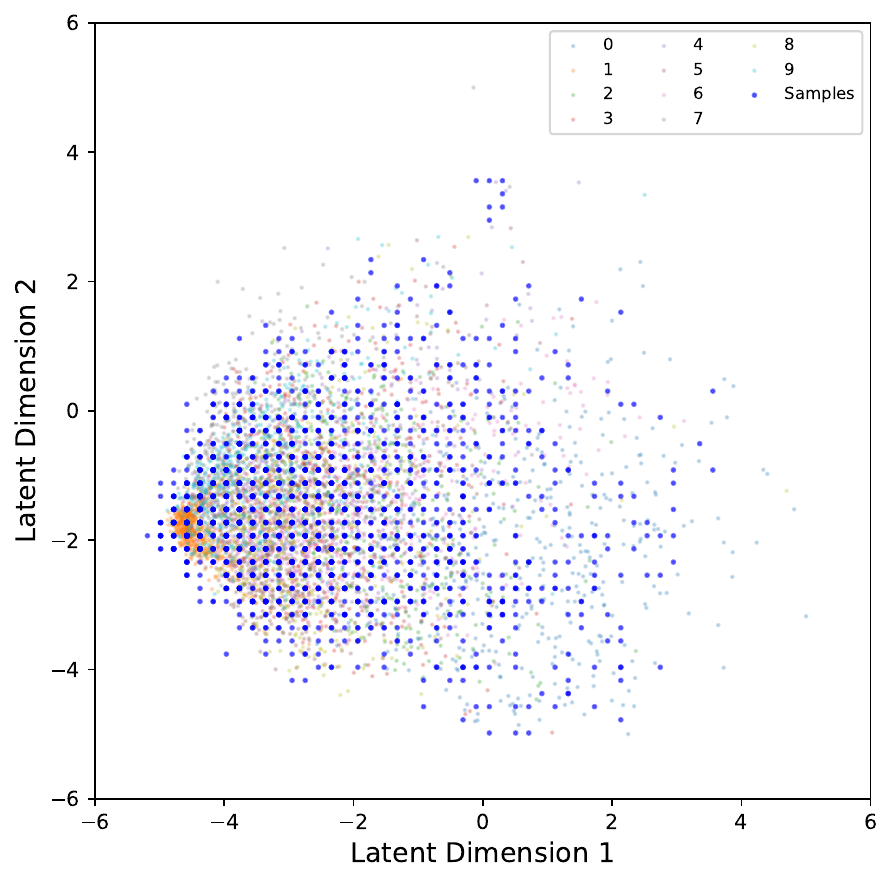}
        \caption{$\alpha=0.5$}
        \label{fig:mnist_nlw_alpha0.5_samples}
    \end{subfigure}
    \hfill
    \begin{subfigure}[b]{0.15\textwidth}
        \centering
        \includegraphics[width=\textwidth]{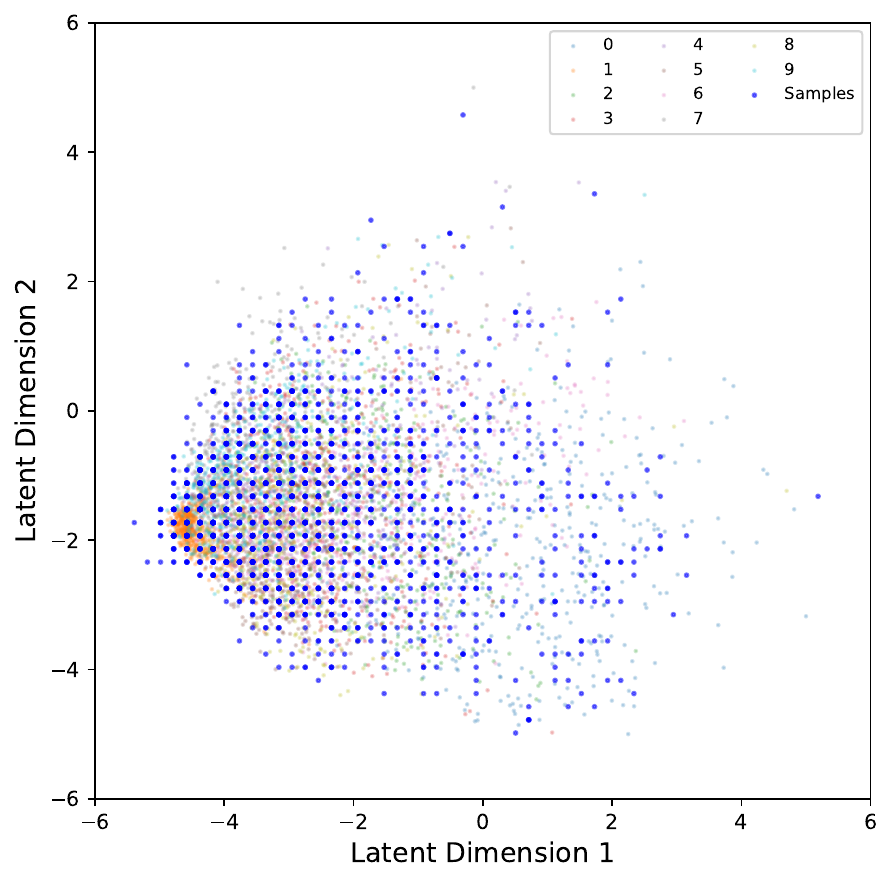}
        \caption{$\alpha=1.0$}
        \label{fig:mnist_nlw_alpha1.0_samples}
    \end{subfigure}
    \hfill
    \begin{subfigure}[b]{0.15\textwidth}
        \centering
        \includegraphics[width=\textwidth]{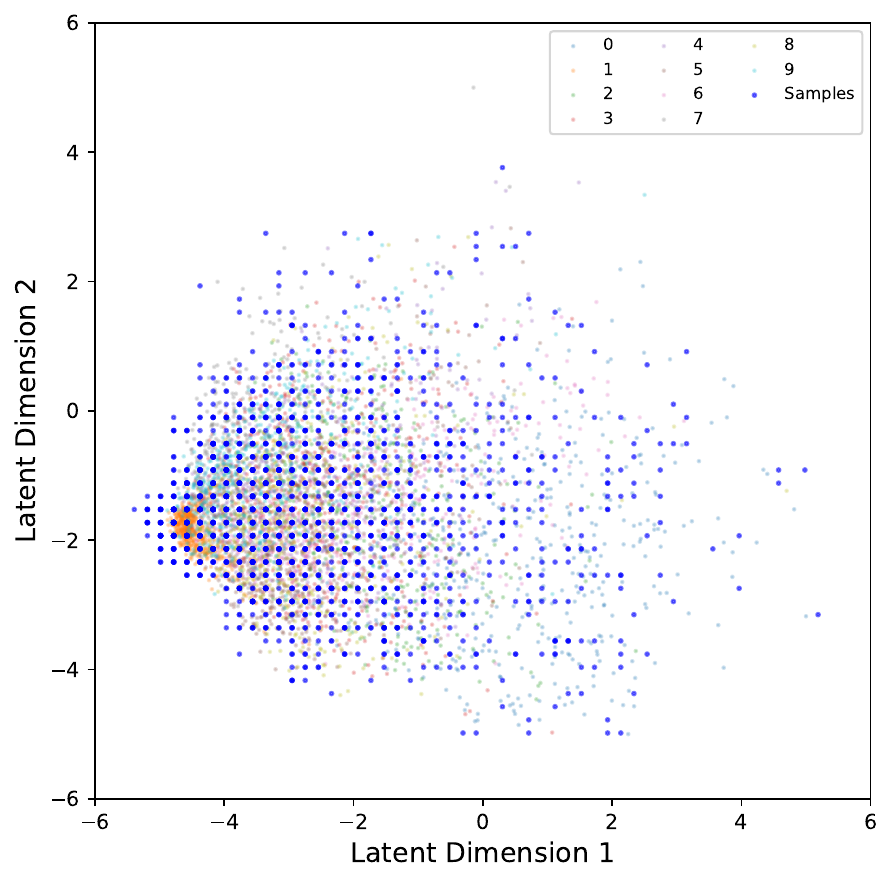}
        \caption{$\alpha=1.5$}
        \label{fig:mnist_nlw_alpha1.5_samples}
    \end{subfigure}
    \caption{MNIST latent space sampling results for different fractional orders $\alpha$. Top row: KDE-estimated densities. Bottom row: scatter plots of sampled points overlaid with true data (gray points). }
    \label{fig:mnist_results}
\end{figure}

First, we consider the MNIST dataset embedded in a two-dimensional latent space. Specifically, we apply PCA to the 784-dimensional pixel space of MNIST and retain the first two principal components, yielding a 2D representation of the data. The target distribution is then estimated via kernel density estimation (KDE) on this latent space. We compare the sampling results of the proposed method with different $\alpha$ values by visualizing both the KDE-estimated densities and the raw samples.

Figure~\ref{fig:mnist_results} shows the KDE densities and scatter plots of our method on the MNIST latent space for different $\alpha$. The method consistently captures the multimodal structure of the target distribution across all $\alpha$ values, though with varying degrees of localization. Smaller $\alpha$ yields more dispersed samples with broader mode coverage, while larger $\alpha$ produces more concentrated distributions around high-density regions.

We next consider alanine dipeptide, a benchmark molecular system widely used in computational chemistry. Following the standard setup, the molecular conformation is described by the backbone dihedral angles $\phi$ and $\psi$, which define the Ramachandran plot. For the angular coordinates, all graph distances, edge displacements, and transport diagnostics are computed with periodic wrapping.
The target distribution is estimated via KDE from a converged long-time molecular dynamics trajectory. We apply the proposed method to this dataset and compare the results for different fractional orders $\alpha$.

\begin{figure}[htbp]
    \centering
    \begin{subfigure}[b]{0.18\textwidth}
        \centering
        \includegraphics[width=\textwidth]{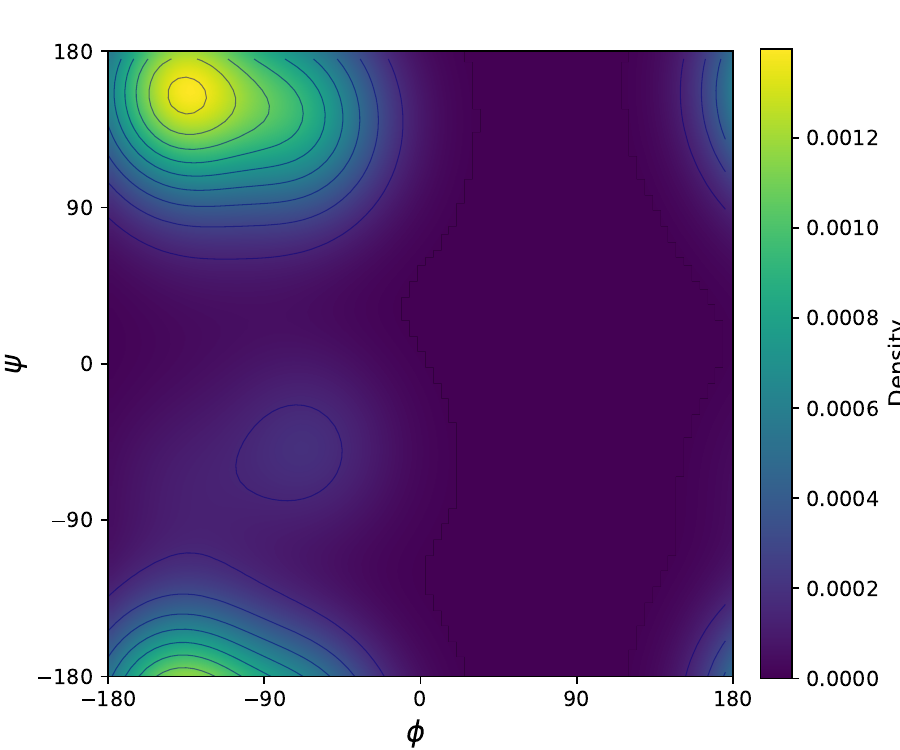}
        \caption{Target}
        \label{fig:ala_target}
    \end{subfigure}
    \hfill
    \begin{subfigure}[b]{0.18\textwidth}
        \centering
        \includegraphics[width=\textwidth]{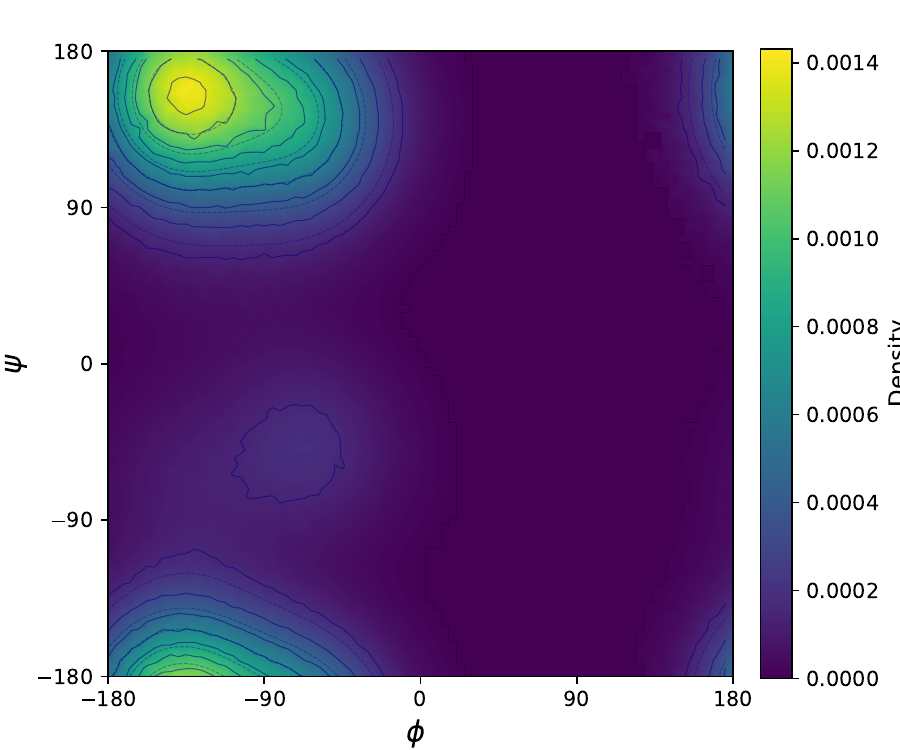}
        \caption{$\alpha=0.3$}
        \label{fig:ala_alpha0.3}
    \end{subfigure}
    \hfill
    \begin{subfigure}[b]{0.18\textwidth}
        \centering
        \includegraphics[width=\textwidth]{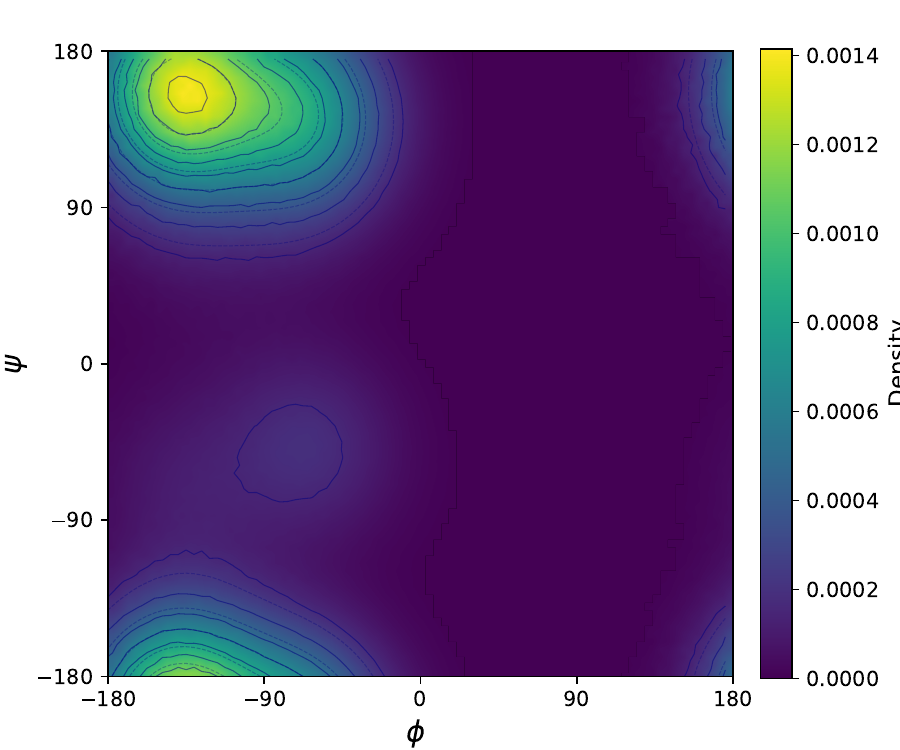}
        \caption{$\alpha=0.5$}
        \label{fig:ala_alpha0.5}
    \end{subfigure}
    \hfill
    \begin{subfigure}[b]{0.18\textwidth}
        \centering
        \includegraphics[width=\textwidth]{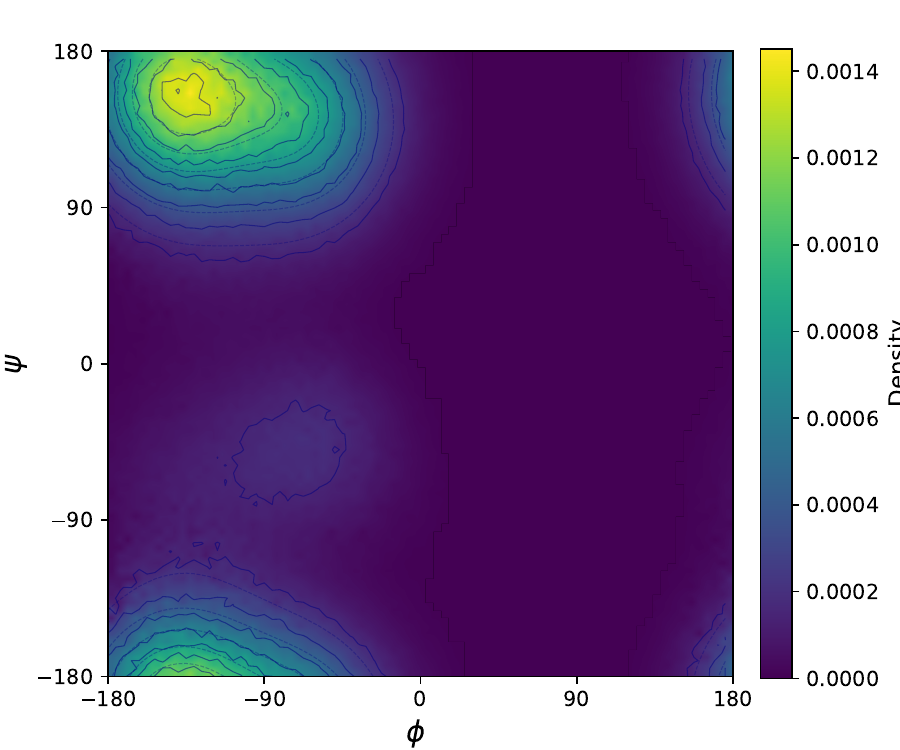}
        \caption{$\alpha=1.0$}
        \label{fig:ala_alpha1.0}
    \end{subfigure}
    \hfill
    \begin{subfigure}[b]{0.18\textwidth}
        \centering
        \includegraphics[width=\textwidth]{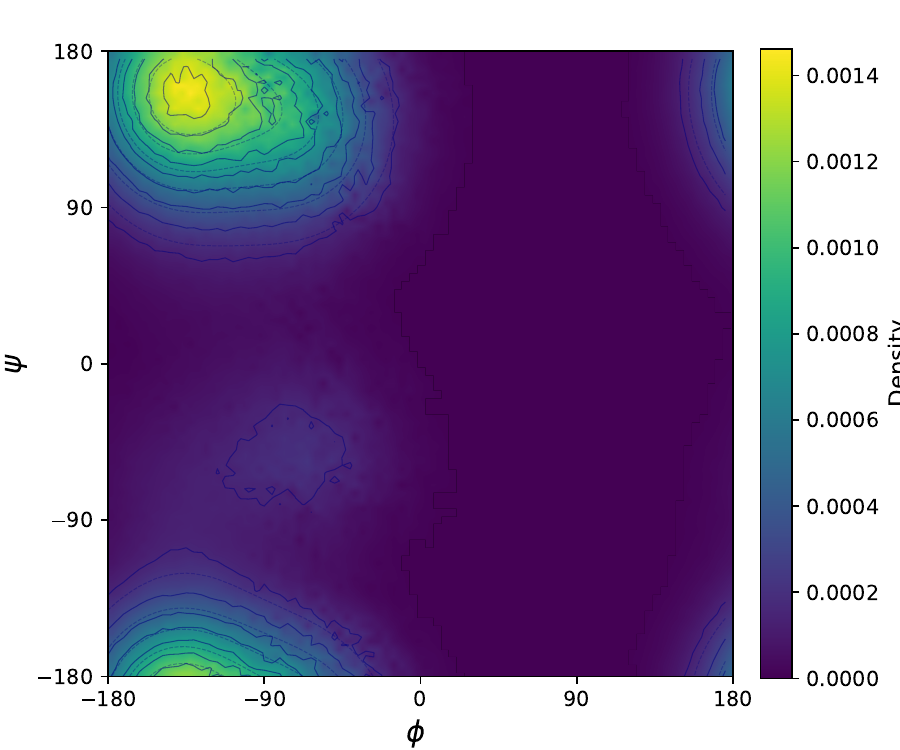}
        \caption{$\alpha=1.5$}
        \label{fig:ala_alpha1.5}
    \end{subfigure}
    
    \caption{Ramachandran plots of alanine dipeptide sampling for different fractional orders $\alpha$. 
             The target distribution exhibits two major metastable states: C7eq ($\phi \approx -80^\circ, \psi \approx 80^\circ$) and C7ax ($\phi \approx 80^\circ, \psi \approx -80^\circ$).}
    \label{fig:alanine_alpha_comparison}
\end{figure}
From Figure~\ref{fig:alanine_alpha_comparison}, we observe that the sampling quality depends significantly on $\alpha$. Smaller $\alpha$ ($0.5$) faithfully reproduces the target distribution, capturing both major metastable states (C7eq and C7ax). Larger $\alpha$ ($1.5$), however, yields localized sampling that misses one of the conformational basins. This suggests that stronger long-range transport can improve basin exploration on this rugged free-energy landscape.
These results confirm the flexibility and effectiveness of the nonlocal transport mechanism on real-world multimodal data.

\begin{figure}[htbp]
    \centering
    \begin{subfigure}[b]{0.48\textwidth}
        \centering
        \includegraphics[width=\textwidth]{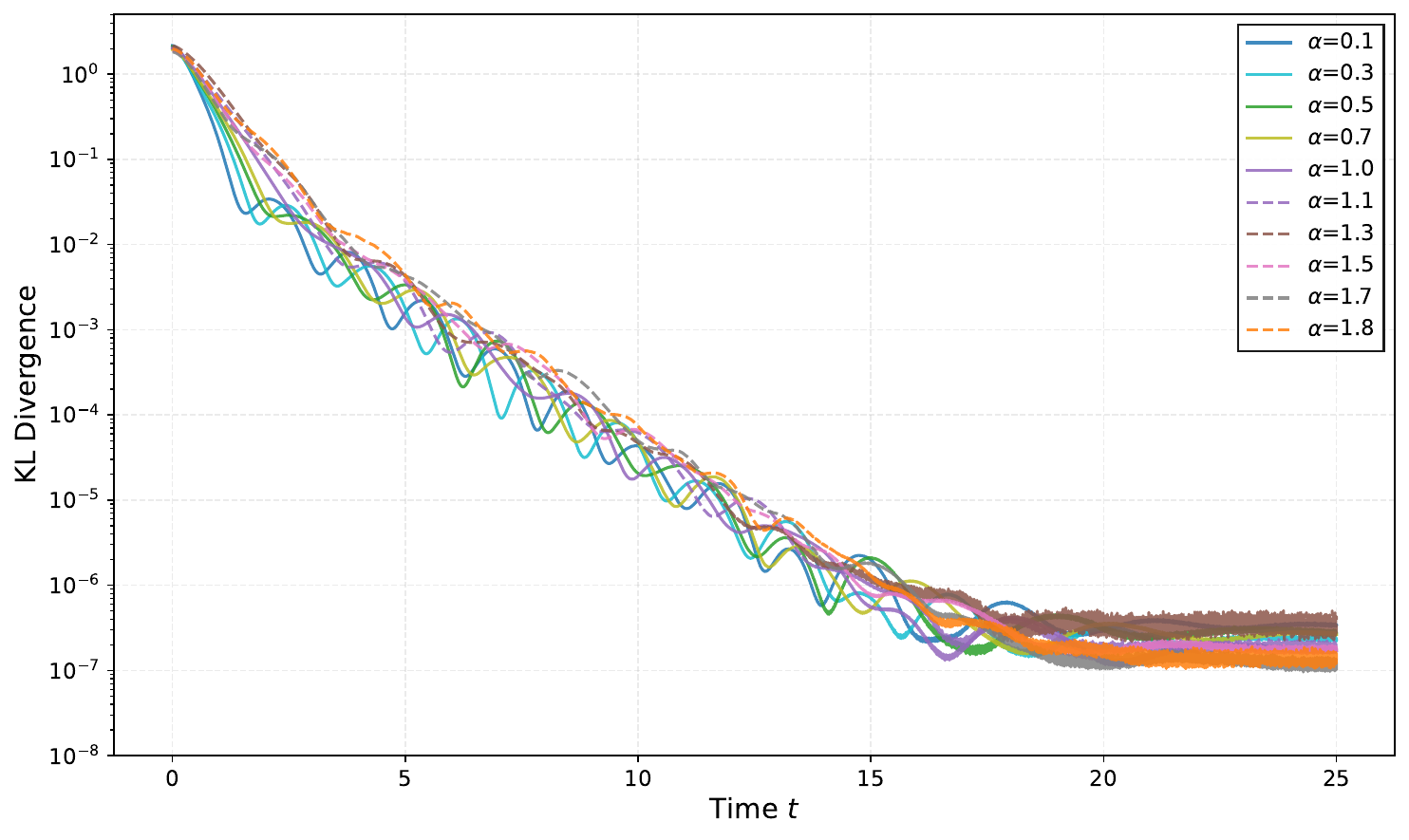}
        \caption{KL evolution for MNIST}
        \label{fig:mnist_kl_evolution}
    \end{subfigure}
    \hfill
    \begin{subfigure}[b]{0.48\textwidth}
        \centering
        \includegraphics[width=\textwidth,height=4.8cm]{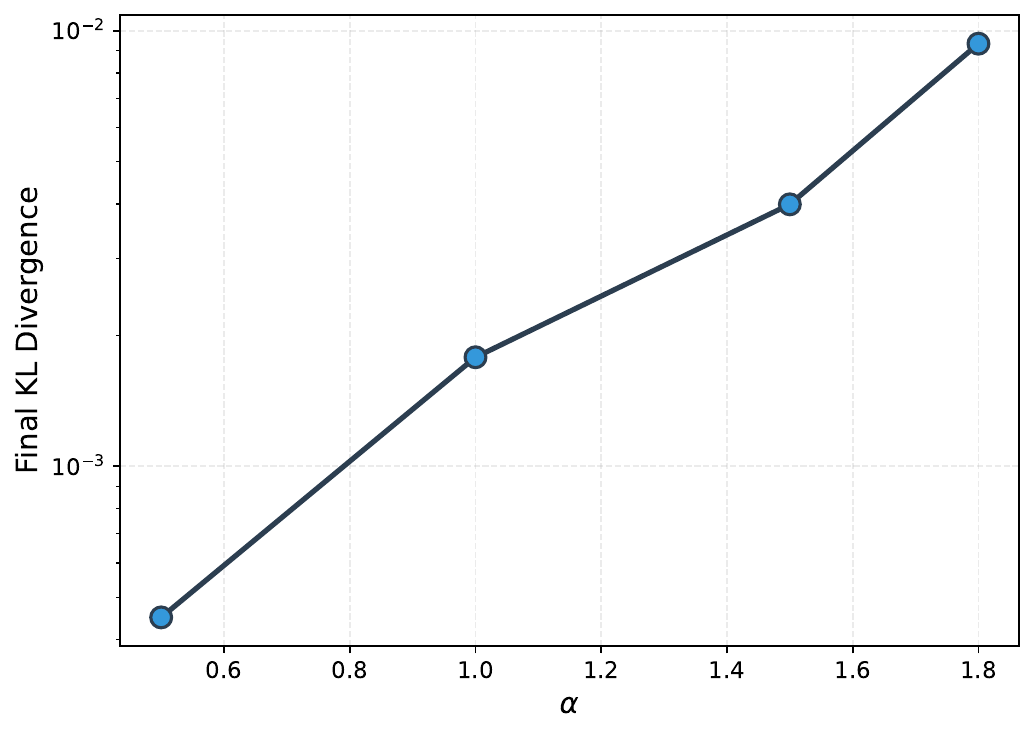}
        \caption{Final KL vs $\alpha$ (alanine dipeptide)}
        \label{fig:alanine_kl}
    \end{subfigure}
    \caption{(a) KL divergence evolution on the two-dimensional MNIST latent space for different $\alpha$; (b) final KL as a function of $\alpha$ for the alanine dipeptide system.}
    \label{fig:kl_results}
\end{figure}

Figure~\ref{fig:kl_results} 
reports the convergence behavior on both datasets. For MNIST, smaller $\alpha$ gives faster initial 
decay, reflecting stronger nonlocal transport, while larger $\alpha$ slows 
convergence toward the local limit. For alanine dipeptide, the final KL as a function of $\alpha$ attains its smallest reported KL near $\alpha \approx 0.5$ among the tested values. These results confirm that the nonlocal mechanism is effective 
in low-dimensional real-world representations, although the optimal $\alpha$ 
depends on the target geometry. A limitation of the current formulation is that all real-data experiments 
are restricted to two-dimensional latent spaces (dihedral angles for alanine 
and PCA for MNIST). This is due to the grid-based discretization, which 
scales as $N = \mathcal{O}(h^{-d})$ and is primarily suited for problems 
with low-dimensional structure or meaningful reaction coordinates.

\section{Conclusion}
In this paper, we introduced the nonlocal Wasserstein Hamiltonian flow, a 
density-based sampling algorithm that embeds L{\'e}vy jump kernels directly into the dynamics of the probability evolution. Unlike particle-based methods, our approach operates at the density level. 
It enables direct long-range interactions between spatial locations, 
avoiding the path dependence that limits conventional gradient-based 
samplers on multimodal distributions. The main contributions of this work are threefold:

(i) We derived a graph-discretized nonlocal Hamiltonian system with L{\'e}vy
jump kernels. It evolves the probability density directly on a graph and
enables long-range mode communication without traversing low-density
barriers.
    
(ii) We established a local exponential convergence theory via linearization 
under a spectral gap condition. The convergence rate depends on the 
spectrum of the nonlocal graph Laplacian.

(iii) We introduced a sparse graph construction with local and long-range edges 
sampled from the L{\'e}vy kernel. It reduces the per-step complexity 
from \(\mathcal{O}(N^2)\) to \(\mathcal{O}(N(k_{local}+m))\) while preserving the 
heavy-tailed interaction structure essential for nonlocal mode exploration. 
    
The numerical experiments show that the proposed method achieves
improved mode coverage and balanced mode occupancy on the tested
multimodal benchmarks. The spectral analysis further reveals a trade-off
between nonlocal connectivity and second-order inertia. Stronger
nonlocality enlarges the spectral gap \(\lambda_1\), which speeds up the
first-order flow but weakens the second-order advantage. In the small-gap regime \(\lambda_1<1\), the second-order flow is
asymptotically faster.

However, the proposed framework presents several limitations. The fixed spatial discretization is subject to the curse of dimensionality, and both the L{\'e}vy kernel parameters and the damping coefficient require problem-specific tuning. On the tested multimodal benchmarks, the nonlocal density-based
formulation achieves better mode coverage and more balanced occupancy
than MCMC, MALA, and underdamped Langevin dynamics under the reported
computational budgets.

This suggests that the nonlocal gradient flow approach holds significant potential for challenging sampling problems.
Future work includes extending the method to high-dimensional
settings, designing adaptive graph constructions that balance local and
long-range connectivity, and further investigating the optimal damping
strategy beyond the linearized regime.

\section*{Acknowledgments}
This work was partly supported by the NSFC International Collaboration
Fund for Creative Research Teams (Grant W2541005), the Guangdong
Provincial Key Laboratory of Mathematical and Neural Dynamical Systems
(Grant 2024B1212010004), the Cross Disciplinary Research Team on Data
Science and Intelligent Medicine (2023KCXTD054), the Guangdong-Dongguan
Joint Research Fund (Grant 2023A1515140016), the Dongguan Key Laboratory
for Data Science and Intelligent Medicine, the Natural Science Foundation
of Guangdong (Grant No. 2025A1515011188), the Hubei Provincial Natural
Science Foundation of China (Grant 2026AFA034), and the Hubei Key
Laboratory of Engineering Modeling and Scientific Computing.

\section*{Statements and Declarations}

\paragraph{Competing interests.}
The authors declare no competing interests.
\paragraph{Data availability.}
The data supporting the numerical results of this study are available from the corresponding author upon reasonable request.

\appendix
\renewcommand{\thesection}{\Alph{section}}
\setcounter{section}{0}
\section{Variational Derivation}\label{app:ham}

We provide the detailed variational calculations of Proposition~\ref{prop:hamiltonian}.

The Hamiltonian (\ref{eq:Hamilton_sym}) depends on $\Phi$ only through the quadratic term. For any test function $\psi$, the first variation is
\begin{equation}
\left\langle \frac{\delta \mathcal{H}}{\delta \Phi}, \psi \right\rangle = \frac14 \cdot 2 \iint (\Phi(x)-\Phi(y)) (\psi(x)-\psi(y)) \hat{\rho}(x,y) J\mathrm{m}(\mathrm{d}x,\mathrm{d}y).
\label{eq:var_phi_step1}
\end{equation}
Using the symmetry \(J(x,\mathrm{d}y)\mathrm{m}(\mathrm{d}x) = J(y,\mathrm{d}x)\mathrm{m}(\mathrm{d}y)\), the right-hand side of (\ref{eq:var_phi_step1}) simplifies to
\begin{equation}
\left\langle \frac{\delta \mathcal{H}}{\delta \Phi}, \psi \right\rangle = \iint \psi(x) (\Phi(x)-\Phi(y)) \hat{\rho}(x,y) \, J\mathrm{m}(\mathrm{d}x,\mathrm{d}y).
\label{eq:var_phi_step2}
\end{equation}
Hence, for fixed \(z\), the first variation with respect to \(\Phi\) is
\begin{equation}
\frac{\delta \mathcal{H}}{\delta \Phi}(z) = \int_{\Omega} (\Phi(z)-\Phi(y)) \hat{\rho}(z,y) J(z,\mathrm{d}y).
\label{eq:var_phi_result}
\end{equation}
Substituting (\ref{eq:var_phi_result}) into $\partial_t \rho = \frac{\delta\mathcal{H}}{\delta\Phi}$ yields (\ref{eq:state}).

We now compute the variation with respect to $\rho$. Since $\hat{\rho}(x,y)=\theta(\rho(x),\rho(y))$ depends on $\rho$ through both arguments, the variational derivative is
\begin{equation}
\frac{\delta \hat{\rho}(x,y)}{\delta \rho(z)} = \partial_1\theta(\rho(x),\rho(y)) \, \delta(z-x) + \partial_2\theta(\rho(x),\rho(y)) \, \delta(z-y).
\label{eq:var_rhoh}
\end{equation}

Applying (\ref{eq:var_rhoh}) to the kinetic energy, we obtain, for fixed \(z\),
\begin{align}
\frac{\delta K}{\delta \rho(z)}
&= \frac14 \int_{\Omega} (\Phi(z)-\Phi(y))^2 \partial_1\theta(\rho(z),\rho(y)) \, J(z,\mathrm{d}y) \nonumber \\
&\quad + \frac14 \int_{\Omega} (\Phi(x)-\Phi(z))^2 \partial_2\theta(\rho(x),\rho(z)) \, J(x,\mathrm{d}z),
\label{eq:var_rho_step1}
\end{align}
where the first integral is over \(y\) and the second over \(x\), with \(z\) fixed.

Using the symmetry \(J\mathrm{m}(\mathrm{d}x,\mathrm{d}z) = J\mathrm{m}(\mathrm{d}z,\mathrm{d}x)\) and the identity \(\partial_2\theta(r,s)=\partial_1\theta(s,r)\), the second term in (\ref{eq:var_rho_step1}) becomes
\[
\frac14 \int_{\Omega} (\Phi(z)-\Phi(x))^2 \partial_1\theta(\rho(z),\rho(x)) \, J(z,\mathrm{d}x),
\]
where the integral is over \(x\) with \(z\) fixed. Thus the two terms are identical, and we have
\begin{equation}
\frac{\delta K}{\delta \rho(z)} = \frac12 \int_{\Omega} (\Phi(z)-\Phi(y))^2 \partial_1\theta(\rho(z),\rho(y)) \, J(z,\mathrm{d}y).
\label{eq:var_rho_result}
\end{equation}

The variation of the potential energy (\ref{eq:KL}) is standard:
\begin{equation}
\frac{\delta E}{\delta \rho(z)} = \log\frac{\rho(z)}{\pi(z)} + 1.
\label{eq:var_E}
\end{equation}

Finally, substituting (\ref{eq:var_rho_result}) and (\ref{eq:var_E}) into $\partial_t \Phi = -\frac{\delta\mathcal{H}}{\delta\rho} = -(\frac{\delta K}{\delta\rho} + \frac{\delta E}{\delta\rho})$ yields (\ref{eq:momentum}).

\section{Proofs of theoretical results }
\label{app:theory_proofs}
This appendix collects the technical proofs supporting the convergence
analysis in Section~\ref{sec:convergence}, including the linearization
and local convergence analysis, the temporal discretization error, and
the sparse graph approximation error. 
 
Let \(\Omega \subset \mathbb{R}^d\) be a bounded domain, and let \(\Omega_h := \{x_i\}_{i=1}^N \subset \Omega\) be a regular grid with spacing \(h\), with volume element \(h^d\).

\subsection{Local convergence and spectral analysis}
\label{app:local_convergence}

This subsection collects the technical details underlying the local
convergence analysis in Section~\ref{subsec:mixing}. We first derive the
linearization around equilibrium and the associated spectral structure,
and then prove the local exponential convergence, optimal damping, and
the local metric equivalence used in the Wasserstein convergence
estimate.

\paragraph{Linearization and spectral structure.}

First, using $\log(1+\varphi_i)=\varphi_i+O(\varphi^2)$, we have
\begin{equation}\label{eq:expansion_G}
G_i-G_j
= \log(1+\varphi_i)-\log(1+\varphi_j)
= \varphi_i-\varphi_j+O(\varphi^2).
\end{equation}

Second, substituting \(\rho_i=\pi_i(1+\varphi_i)\) into the definition of the logarithmic mean function and expanding around \(\pi\) gives
\begin{equation}\label{eq:expansion_theta}
\theta_{ij}(\rho) = \theta(\pi_i,\pi_j) + \partial_1\theta(\pi_i,\pi_j)\pi_i\varphi_i + \partial_2\theta(\pi_i,\pi_j)\pi_j\varphi_j + O(\varphi^2).
\end{equation}
Thus \(\theta_{ij}(\rho) = \theta_{ij}^0 + O(\varphi)\), where \(\theta_{ij}^0 := \theta(\pi_i,\pi_j)\).

Third, we consider the nonlocal convection term defined in (\ref{eq:convection}). Its expansion gives
\begin{equation}\label{eq:expansion_C}
\mathcal C_{ij}(\rho,u)
=\frac12\left(\sum_k\partial_1\theta_{ik}u_{ik}^2\nu_{ik}
- \sum_k\partial_1\theta_{jk}u_{jk}^2\nu_{jk}\right)
= O(\|u\|^2).
\end{equation}
Thus it does not contribute to the first-order linearization.
We can define the operators
\[
(Bu)_i:=\pi_i^{-1}\sum_{j:(i,j)\in\mathcal E}\theta_{ij}^0u_{ij}\nu_{ij},\qquad
(C\varphi)_{ij}:=\varphi_i-\varphi_j.
\]
Then the linearized dynamics can be written compactly as
\begin{align}\label{eq:linearized_dynamics}
\partial_t\varphi=-Bu,\qquad \partial_t u=C\varphi-\gamma u.
\end{align}
Eliminating \(u\) gives
\(\partial_{tt}\varphi+\gamma\partial_t\varphi+A\varphi=0\)
with \(A=BC\), which coincides with the explicit form in the main text.

\begin{proof}[Proof of self-adjointness]
For any \(\phi,\psi\in\ell^2(\pi)\), where \(\ell^2(\pi)\) is equipped with the inner product
\(\langle \phi,\psi\rangle_\pi := \sum_i \pi_i \phi_i \psi_i\), we have
\[
\begin{aligned}
\langle A\phi,\psi\rangle_\pi
= \sum_i\pi_i(A\phi)_i\psi_i = \sum_{i,j}\theta_{ij}^0(\phi_i-\phi_j)\psi_i\nu_{ij} = \frac12\sum_{i,j}\theta_{ij}^0(\phi_i-\phi_j)(\psi_i-\psi_j)\nu_{ij},
\end{aligned}
\]
where the last equality follows from symmetrization using
$\theta_{ij}^0=\theta_{ji}^0$ and $\nu_{ij}=\nu_{ji}$.
The expression is symmetric in $\phi$ and $\psi$,
hence $A$ is self-adjoint on $\ell^2(\pi)$.
Moreover,
\(
\langle A\phi,\phi\rangle_\pi
= \frac12\sum_{i,j}\theta_{ij}^0(\phi_i-\phi_j)^2\nu_{ij}\ge0,
\)
so $A$ is positive semidefinite.
\end{proof}

We now relate the linearized operator $A$ to the nonlocal Fisher information
\[
\mathcal I(\rho\|\pi) := \frac12 \sum_{i,j} \theta_{ij}(\rho)(G_i-G_j)^2 \nu_{ij},
\qquad G_i := \log\frac{\rho_i}{\pi_i}+1.
\]
Substituting (\ref{eq:expansion_G}) and (\ref{eq:expansion_theta}) into the Fisher information yields
\[
\begin{aligned}
\mathcal I(\rho\|\pi)
= \frac12 \sum_{i,j} \theta_{ij}^0 (\varphi_i-\varphi_j)^2 \nu_{ij}
+ O(\|\varphi\|^3) 
= \langle A\varphi, \varphi\rangle_\pi + O(\|\varphi\|_\pi^3).
\end{aligned}
\]
This quadratic expansion provides the key link between the linearized operator $A$ and the local dissipation rate of the relative entropy.

\paragraph{Proof of local exponential convergence.}

\begin{proof}[Proof of Theorem~\ref{thm:local_exp}]
	To analyze local convergence, we linearize the dynamics around the
	equilibrium \((\pi,0)\). In the perturbation variables
	\(z=(\varphi,u)\), where \(\rho_i=\pi_i(1+\varphi_i)\), the nonlinear
	dynamics takes the form
	\[
	\dot z=\mathrm{M}z+\mathcal R(z),
	\qquad\text{equivalently}\qquad
	\ddot{\varphi}+\gamma\dot{\varphi}+A\varphi=\mathcal R_\varphi(\varphi,u),
	\quad A=BC,
	\]
	where \(\mathrm{M}\) is the linearization operator on \((\varphi,u)\), \(A=BC\)
	is the stiffness operator obtained after eliminating \(u\). The remainder \(\mathcal R\)
	collects the quadratic and higher-order terms; by Taylor's theorem on
	\(\mathcal U_\delta\), \(\mathcal R(0)=0\), \(D\mathcal R(0)=0\), and
	\[
	\|\mathcal R(z)\|_{\mathcal X}\le C_R\|z\|_{\mathcal X}^2
	\quad\text{for all }z\in\mathcal U_{\delta},
	\]
	where
	\(\|z\|_{\mathcal X}^2:=\|\varphi\|_{\ell^2(\pi)}^2+\|u\|_{\mathcal U}^2\),
	and \(\|\cdot\|_{\mathcal U}\) is the norm in the kinetic part of
	\(\mathcal L\).
	
	For the linearized system \(\ddot{\varphi}+\gamma\dot{\varphi}+A\varphi=0\),
	expanding \(\varphi\) in the eigenbasis of \(A\), \(\varphi(t)=\sum_kc_k(t)\phi_k\),
	gives \(\ddot c_k+\gamma\dot c_k+\lambda_kc_k=0\). Each mode decays as
	\(O(e^{-r_k(\gamma)t})\), except at critical damping
	\(\gamma=2\sqrt{\lambda_k}\), where \(c_k(t)=(a_k+b_kt)e^{-r_k(\gamma)t}\).
	The global decay is dominated by \(\lambda_1\), with rate
	\[
	r_1(\gamma)=
	\begin{cases}
		\gamma/2, & \gamma\le 2\sqrt{\lambda_1},\\[2pt]
		\bigl(\gamma-\sqrt{\gamma^2-4\lambda_1}\bigr)/2, & \gamma>2\sqrt{\lambda_1}.
	\end{cases}
	\]
	In particular, \(r_1(\gamma)>0\) for all \(\gamma>0\), and the linearized
	decay rate is maximized at
	\(\gamma^*=2\sqrt{\lambda_1}\), where \(r_1(\gamma^*)=\sqrt{\lambda_1}\).

	Let \(S(t)=e^{\mathrm{M}t}\) denote the solution semigroup of the linearized
	system \(\dot z=\mathrm{M}z\), so that \(z(t)=S(t)z(0)\) solves the linearized
	dynamics.
	Setting \(\lambda_c:=2r_1(\gamma)>0\), the modal computation yields
	\begin{equation}\label{eq:linear_semigroup_decay}
		\|S(t)z(0)\|_{\mathcal X}
		\le
		C_S(1+t)e^{-r_1(\gamma)t}\|z(0)\|_{\mathcal X}.
	\end{equation}
	For \(\gamma\ne2\sqrt{\lambda_1}\), the factor \((1+t)\) can be replaced
	by \(1\). The corresponding estimate for the full phase-space variable
	\(z=(\varphi,u)\) follows from the modal decomposition of the
	first-order linearized operator \(\mathrm M\): the \(\varphi\)- and
	\(u\)-components associated with each nonzero eigenmode of \(A\)
	inherit the same exponential rate, and any additional components in the
	kernel of the continuity operator decay according to the damping term.
	Hence \eqref{eq:linear_semigroup_decay} holds for the full phase-space
	variable \(z=(\varphi,u)\).

	For the full nonlinear system,
	\(z(t)=S(t)z(0)+\int_0^tS(t-s)\mathcal R(z(s))\,\mathrm{d}s\), with
	\(\|\mathcal R(z)\|_{\mathcal X}\le C_R\|z\|_{\mathcal X}^2\) on
	\(\mathcal U_\delta\). Let
	\(
	q_\gamma(t):=1
	\)
	for \(\gamma\neq2\sqrt{\lambda_1}\) and
	\(
	q_\gamma(t):=1+t
	\)
	at \(\gamma=2\sqrt{\lambda_1}\). The linear semigroup estimate gives
	\(\|S(t)\|_{\mathcal X\to\mathcal X}\le C_Sq_\gamma(t)e^{-r_1(\gamma)t}\),
	hence
	\[
	\|z(t)\|_{\mathcal X}
	\le
	C_Sq_\gamma(t)e^{-r_1(\gamma)t}\|z(0)\|_{\mathcal X}
	+
	C_RC_S
	\int_0^t
	q_\gamma(t-s)e^{-r_1(\gamma)(t-s)}
	\|z(s)\|_{\mathcal X}^2\,\mathrm{d}s.
	\]
	Let \(T^*>0\) denote the maximal time on which the bootstrap bound
	\(
	\|z(t)\|_{\mathcal X}\le2C_Sq_\gamma(t)e^{-r_1(\gamma)t}\|z(0)\|_{\mathcal X}
	\)
	holds for \(0\le t\le T^*\). Substituting this bound into the integral
	term yields
	\[
	\|z(t)\|_{\mathcal X}
	\le
	C_Sq_\gamma(t)e^{-r_1(\gamma)t}\|z(0)\|_{\mathcal X}
	+
	4C_RC_S^3\|z(0)\|_{\mathcal X}^2
	\int_0^t
	q_\gamma(t-s)q_\gamma(s)^2
	e^{-r_1(\gamma)(t-s)}e^{-2r_1(\gamma)s}
	\,\mathrm{d}s.
	\]
	Since \(q_\gamma(t)\ge1\), \(q_\gamma(t-s)\le q_\gamma(t)\), and
	\(
	e^{-r_1(\gamma)(t-s)}e^{-2r_1(\gamma)s}=e^{-r_1(\gamma)t}e^{-r_1(\gamma)s}
	\),
	we obtain
	\[
	\|z(t)\|_{\mathcal X}
	\le
	C_Sq_\gamma(t)e^{-r_1(\gamma)t}\|z(0)\|_{\mathcal X}
	\bigl(1+C_RC_1(r_1)C_S^2\|z(0)\|_{\mathcal X}\bigr),
	\]
	where
	\(
	C_1(r_1):=4\int_0^\infty q_\gamma(s)^2e^{-r_1(\gamma)s}\,\mathrm{d}s<\infty
	\).
	Choose \(\delta_0>0\) small enough that
	\(\delta_0\le\delta\), \(C_RC_1(r_1)C_S^2\delta_0\le\frac12\), and
	\(2C_SM_{r_1}\delta_0\le\delta\), with
	\(
	M_{r_1}:=\sup_{t\ge0}q_\gamma(t)e^{-r_1(\gamma)t}<\infty
	\).
	Then for \(\|z(0)\|_{\mathcal X}\le\delta_0\) the bootstrap bound
	improves to
\(
\|z(t)\|_{\mathcal X}
\le\frac32 C_Sq_\gamma(t)e^{-r_1(\gamma)t}
\|z(0)\|_{\mathcal X}
\)
	on \(0\le t\le T^*\), and
	\(\sup_{0\le t\le T^*}\|z(t)\|_{\mathcal X}\le\delta\). Hence the
	solution can be continued beyond \(T^*\), so \(T^*=\infty\): the solution
	exists globally, remains in \(\mathcal U_\delta\), and
	\begin{equation}\label{eq:nonlinear_phase_decay}
		\|z(t)\|_{\mathcal X}^2
		\le
		C_{\mathrm{nl}}q_\gamma(t)^2e^{-\lambda_c t}\|z(0)\|_{\mathcal X}^2,
		\qquad
		\lambda_c=2r_1(\gamma).
	\end{equation}

	Finally, on a sufficiently small neighborhood of \(\pi\), the mass
	constraint \(\sum_i\pi_i\varphi_i=0\) and the Taylor expansion of the
	relative entropy give
	\begin{equation}\label{ieq:euqality_phi}
		E(\rho)
		=
		\frac12\|\varphi\|_{\ell^2(\pi)}^2
		+
		O(\|\varphi\|_{\ell^2(\pi)}^3).
	\end{equation}
	Hence, after possibly reducing \(\delta\), there exist \(c_{11},c_{12}>0\)
	with
	\begin{align}  
	\label{ieq:energy_decay}
	c_{11}\|\varphi\|_{\ell^2(\pi)}^2\le E(\rho)\le c_{12}\|\varphi\|_{\ell^2(\pi)}^2
	\end{align}
	Together with the local equivalence of the kinetic energy
	\(K(\rho,u)=\tfrac12\|u\|_{\mathcal U,\rho}^2\) and the corresponding velocity norm,
	this yields \(c_L,C_L>0\) with
	\(
	c_L\|(\varphi,u)\|_{\mathcal X}^2\le\mathcal L(\rho,u)\le C_L\|(\varphi,u)\|_{\mathcal X}^2
	\)
	for all \((\rho,u)\in\mathcal U_\delta\). Combining this with
	\eqref{eq:nonlinear_phase_decay},
	\[
	\mathcal L(\rho(t),u(t))
	\le
	C_{\mathrm{ent}}q_\gamma(t)^2e^{-\lambda_c t}\mathcal L(\rho(0),u(0)),
	\qquad
	C_{\mathrm{ent}}=\frac{C_L}{c_L}C_{\mathrm{nl}}.
	\]
	Since \(q_\gamma(t)=1\) for \(\gamma\neq2\sqrt{\lambda_1}\) and
	\(q_\gamma(t)=1+t\) at critical damping, this proves the claimed decay
	estimate.
	
\end{proof}

\paragraph{Optimal damping and mixing-time consequence.}

\begin{proof}[Proof of Corollary~\ref{cor:acceleration}]
	The modal equation $\ddot c_k+\gamma\dot c_k+\lambda_k c_k=0$ was derived
	in the proof of Theorem~\ref{thm:local_exp}; its characteristic equation
	yields the stated decay rate $r_k(\gamma)$. Since
	\(\lambda_1\) is the smallest positive eigenvalue, the global decay
	rate is \(r(\gamma)=\min_kr_k(\gamma)=r_1(\gamma)\).  Maximizing the one-dimensional function
	$r_1(\gamma)$ gives $\gamma^*=2\sqrt{\lambda_1}$ and $r^*=\sqrt{\lambda_1}$.
	At critical damping, Theorem~\ref{thm:local_exp} bounds the
		normalized density error by $C(1+t)e^{-\sqrt{\lambda_1}t}$, where $C$
		is fixed for the graph and initial data under consideration.
		The envelope is decreasing for sufficiently large $t$; its final
		crossing yields
		$\sqrt{\lambda_1}T_{\rm mix}^{(2)}
		\le \log(1/\varepsilon)+O(\log\log(1/\varepsilon))$.
		The first-order exponential envelope gives the stated first-order bound.
		For a pure slowest mode with zero initial momentum, the exact normalized
		amplitudes are $(1+\sqrt{\lambda_1}t)e^{-\sqrt{\lambda_1}t}$ and
		$e^{-\lambda_1t}$, respectively. Their settling times have the stated
		leading-order ratio.
\end{proof}

\paragraph{Metric equivalence and Wasserstein convergence.}
\begin{proof}[Proof of Lemma~\ref{lem:wnl_l2}]
We establish the local equivalence between \(W_{\mathrm{NL}}\) and the
normalized mass-error norm
\(\|\rho-\pi\|_{\pi^{-1}}^2=\sum_i(\rho_i-\pi_i)^2/\pi_i\).

The nonlocal Wasserstein metric \(W_{\mathrm{NL}}\) of Erbar and
Maas~\cite{MAAS20112250} endows the probability simplex with a
finite-dimensional Riemannian structure; we denote by
\(|\dot\varphi|_{W_{\mathrm{NL}}}\) its tangent norm at \(\pi\), which is
characterized by the quadratic minimization associated with the nonlocal
continuity equation. For the perturbation \(\varphi\), the linearized
continuity equation of Section~\ref{sec:linearization} reads
\(\partial_t\varphi=-Bu\), where
\((Bu)_i=\pi_i^{-1}\sum_j\theta_{ij}^0u_{ij}\nu_{ij}\). Hence, for a
tangent vector \(\dot\varphi\),
\[
|\dot\varphi|_{W_{\mathrm{NL}}}^2
=
\min_{Bu=-\dot\varphi}
\frac12\sum_{i,j}u_{ij}^2\theta_{ij}^0\nu_{ij}.
\]
Introducing a Lagrange multiplier \(\psi\), the Euler--Lagrange equation
shows that the optimal velocity field is of gradient form,
\(u=-C\psi\), with \((C\psi)_{ij}=\psi_i-\psi_j\). Substituting into the
linearized continuity equation gives
\(\dot\varphi=BC\psi=A\psi\). Since
\(\sum_i\pi_i\dot\varphi_i=0\), \(A\) is invertible on the zero-mass
subspace, so \(\psi=A^\dagger\dot\varphi\) with \(A^\dagger\) the
Moore--Penrose pseudo-inverse of \(A\). Consequently,
\(
|\dot\varphi|_{W_{\mathrm{NL}}}^2
=
\langle\dot\varphi,A^\dagger\dot\varphi\rangle_\pi .
\)
Let \(A\phi_k=\lambda_k\phi_k\), with \(\{\phi_k\}_{k\ge1}\) an
orthonormal basis of the zero-mass subspace in \(\ell^2(\pi)\), and
expand \(\dot\varphi=\sum_{k\ge1}a_k\phi_k\). Then
\[
\langle\dot\varphi,A^\dagger\dot\varphi\rangle_\pi
=
\sum_{k\ge1}\frac{a_k^2}{\lambda_k}.
\]
Since \(\lambda_1\le\lambda_k\le\lambda_{\max}\),
\(
\frac1{\lambda_{\max}}\|\dot\varphi\|_{\ell^2(\pi)}^2
\le
|\dot\varphi|_{W_{\mathrm{NL}}}^2
\le
\frac1{\lambda_1}\|\dot\varphi\|_{\ell^2(\pi)}^2.
\)
Finally, \(\rho_i-\pi_i=\pi_i\varphi_i\) and
\(
\|\rho-\pi\|_{\pi^{-1}}^2
=
\sum_i\frac{(\rho_i-\pi_i)^2}{\pi_i}
=
\sum_i\pi_i\varphi_i^2
=
\|\varphi\|_{\ell^2(\pi)}^2.
\)
After shrinking the neighborhood of \(\pi\), the tangent metric bounds
	hold uniformly along paths contained in this neighborhood. Integrating
	these bounds over admissible paths and taking the infimum yields
	\eqref{ieq:W_l2_equiv}.
\end{proof}

\begin{proof}[Proof of Corollary~\ref{cor:WNL}]
	By Theorem~\ref{thm:local_exp}, the phase-space Lyapunov functional decays
	as in \eqref{ieq:energy_decay}. Since \(\mathcal L=E+K\) with \(K\ge0\),
	we have \(E(\rho(t))\le\mathcal L(\rho(t),u(t))\), and by the entropy
	equivalence \eqref{ieq:euqality_phi},
	\[
	\|\rho(t)-\pi\|_{\pi^{-1}}
	\le
	\sqrt{\frac{C_{\mathrm{ent}}c_{12}}{c_{11}}}
	(1+t)e^{-\lambda_c t/2}
	\|\rho(0)-\pi\|_{\pi^{-1}},
	\]
	where we used \(u(0)=0\) so that
	\(\mathcal L(\rho(0),0)=E(\rho(0))\le c_{12}\|\rho(0)-\pi\|_{\pi^{-1}}^2\).
	By Lemma~\ref{lem:wnl_l2},
	\(c_1\|\rho-\pi\|_{\pi^{-1}}\le W_{\mathrm{NL}}(\rho,\pi)\le
	c_2\|\rho-\pi\|_{\pi^{-1}}\) with
	\(c_1=(1-\eta)\lambda_{\max}^{-1/2}\),
		\(c_2=(1+\eta)\lambda_1^{-1/2}\). Hence
	\[
	W_{\mathrm{NL}}(\rho(t),\pi)
	\le
	\frac{c_2}{c_1}
	\sqrt{\frac{C_{\mathrm{ent}}c_{12}}{c_{11}}}
	(1+t)e^{-\lambda_c t/2}
	W_{\mathrm{NL}}(\rho(0),\pi),
	\]
	and since \(c_2/c_1=\frac{1+\eta}{1-\eta}
		\sqrt{\lambda_{\max}/\lambda_1}\),
	\(
	C_W
	=
	\frac{1+\eta}{1-\eta}\sqrt{\frac{\lambda_{\max}}{\lambda_1}}
	\sqrt{\frac{C_{\mathrm{ent}}c_{12}}{c_{11}}}.
	\)
	This completes the proof.
\end{proof}

\subsection{Temporal discretization error (Lemma~\ref{lem:temporal_error})}\label{app:temporal discretization error}

To close the error recursion, we need Lipschitz properties of the nonlinear
terms. The following bounds hold in a neighborhood of the equilibrium
$(\pi,0)$ and are assumed throughout the local error analysis.

\begin{proof}[Proof of Lemma~\ref{lem:lipschitz}]
Let \(e_\rho=\rho^1-\rho^2\), \(e_u=u^1-u^2\). Define the actual weighted degree
\(
D_\nu:=\max_i\sum_j\nu_{ij}.
\) For a fixed grid with mesh size $h>0$, the weights are bounded:
\(
0\le \nu_{ij}\le \nu_{\max}:=h^{-d-\alpha}<\infty.
\)
For each node $i$, 
\begin{equation}
    \sum_j\nu_{ij}\le D_\nu.
    \label{eq:degree-bound}
\end{equation}
Under Assumption~\ref{ass:local_regular}, we have
\(c_\rho\le\rho_i\le C_\rho\) and
\(|\partial_1\theta_{ij}|\le C_{\partial\theta}\). Hence:
\begin{enumerate}
\item[(F1)] By the mean value theorem,
\(|\theta_{ij}(\rho^1)-\theta_{ij}(\rho^2)|
\le C_{\partial\theta}(|\rho_i^1-\rho_i^2|+|\rho_j^1-\rho_j^2|)\).
\item[(F2)] The velocity is uniformly bounded:
\(\|u\|_\infty\le U_{\max}\).
\end{enumerate}

\textbf{Proof of (L1).} Write
\[
\theta_{ij}(\rho^1)u_{ij}^1-\theta_{ij}(\rho^2)u_{ij}^2
=\theta_{ij}(\rho^1)e_{u,ij}
+(\theta_{ij}(\rho^1)-\theta_{ij}(\rho^2))u_{ij}^2.
\]
Let $A_i = \sum_j (\theta_{ij}(\rho^1)u_{ij}^1 - \theta_{ij}(\rho^2)u_{ij}^2)\nu_{ij}$. Using (F1) and (F2),
\[
|A_i|\le \sum_j|e_{u,ij}|\nu_{ij}
+ C_{\partial\theta}U_{\max}\sum_j(|e_{\rho,i}|+|e_{\rho,j}|)\nu_{ij}.
\]
By Cauchy-Schwarz and \eqref{eq:degree-bound},
\(
\sum_j|e_{u,ij}|\nu_{ij}
\le\sqrt{D_\nu}
\Bigl(\sum_j|e_{u,ij}|^2\nu_{ij}\Bigr)^{1/2},
\)
\(
\sum_j|e_{\rho,j}|\nu_{ij}
\le\sqrt{D_\nu}
\Bigl(\sum_j|e_{\rho,j}|^2\nu_{ij}\Bigr)^{1/2},
\)
while \(\sum_j|e_{\rho,i}|\nu_{ij}\le D_\nu|e_{\rho,i}|\).
Squaring, summing over \(i\), and taking square roots yields
\[
\|A\|_{\ell^2}\le L_\theta(\|e_\rho\|_{\ell^2}+\|e_u\|_{\ell^2}),
\qquad
L_\theta=O\bigl(\sqrt{D_\nu}
+ C_{\partial\theta}U_{\max}D_\nu\bigr).
\]

\textbf{Proof of (L2).} By Assumption~\ref{ass:local_regular},
\(|G_i(\rho^1)-G_i(\rho^2)|\le \rho_{min}^{-1}|e_{\rho,i}|\), hence
\[
\|G(\rho^1)-G(\rho^2)\|_{\ell^2}\le L_G\|e_\rho\|_{\ell^2},
\qquad
L_G=\rho_{min}^{-1}.
\]

\textbf{Proof of (L3).} Define
\(S_i(\rho,u)=\sum_k\partial_1\theta(\rho_i,\rho_k)u_{ik}^2\nu_{ik}\), so that
\(\mathcal C_{ij}(\rho,u)=\frac12(S_i(\rho,u)-S_j(\rho,u))\). Add and subtract $\partial_1\theta(\rho_i^1,\rho_k^1)(u_{ik}^2)^2$ and use the bounds on $u$, $\partial_1\theta$, and $\partial_{ab}^2\theta$. This gives
\[
|S_i(\rho^1,u^1)-S_i(\rho^2,u^2)|
\le C_{\theta,2}U_{\max}^2\sum_k(|e_{\rho,i}|+|e_{\rho,k}|)\nu_{ik}
+2C_{\partial\theta}U_{\max}\sum_k|e_{u,ik}|\nu_{ik}.
\]
By Cauchy-Schwarz and \eqref{eq:degree-bound}, squaring and summing
over \(i\) yields
\(
\|\mathcal C(\rho^1,u^1)-\mathcal C(\rho^2,u^2)\|_{\ell^2}
\le L_{\mathcal C,\rho}\|e_\rho\|_{\ell^2}+L_{\mathcal C,u}\|e_u\|_{\ell^2},
\quad
L_{\mathcal C,\rho}=O(C_{\theta,2}U_{\max}^2D_\nu),\quad
L_{\mathcal C,u}=O(C_{\partial\theta}U_{\max}\sqrt{D_\nu}).
\)
\end{proof}

\begin{proof}[Proof of Lemma~\ref{lem:temporal_error}]
Let \((\rho(t),u(t))\) be the continuous solution and
\((\rho^k,u^k)\) its time-discrete approximation at \(t_k=k\Delta t\).
Define the global errors
\(e_\rho^k=\rho^k-\rho(t_k)\) and \(e_u^k=u^k-u(t_k)\).

\textbf{Step 1: Local truncation error.}
Under the stated regularity assumptions, Taylor expansion of the
exact solution gives
\[
\begin{aligned}
	\rho_i(t_{k+1})
	&= \rho_i(t_k)
	-\Delta t\sum_j
	\theta_{ij}(\rho(t_k))u_{ij}(t_{k+1})\nu_{ij}
	+O(\Delta t^2),\\
	u_{ij}(t_{k+1})
	&=u_{ij}(t_k)
	+\Delta t\bigl(
	G_i(t_k)-G_j(t_k)
	+\mathcal C_{ij}(\rho(t_k),u(t_k))
	-\gamma_k u_{ij}(t_k)
	\bigr)
	+O(\Delta t^2),
\end{aligned}
\]
where the remainder bounds are uniform for
\(t_k,t_{k+1}\in[0,T]\).
Hence the local truncation errors satisfy
\(\|\tau_\rho^k\|_{\ell^2}\le C_\tau\Delta t\) and
\(\|\tau_u^k\|_{\ell^2}\le C_\tau\Delta t\).

\textbf{Step 2: Error recursion.}
Subtracting the exact evolution from the numerical scheme gives
\[
\frac{e_\rho^{k+1}-e_\rho^k}{\Delta t}
= -\sum_j\bigl(\theta_{ij}(\rho^k)u_{ij}^{k+1}
-\theta_{ij}(\rho(t_k))u_{ij}(t_{k+1})\bigr)\nu_{ij}+\tau_\rho^k,
\]
\[
\frac{e_u^{k+1}-e_u^k}{\Delta t}
= \bigl(G_i^k-G_j^k-(G_i(t_k)-G_j(t_k))\bigr)
+\bigl(\mathcal{C}_{ij}(\rho^k,u^k)-\mathcal{C}_{ij}(\rho(t_k),u(t_k))\bigr)
-\gamma_k e_u^{k+1}+\tau_u^k.
\]

\textbf{Step 3: Lipschitz bounds.}
Applying Lemma~\ref{lem:lipschitz} and the triangle inequality,
\[
\left\|\frac{e_\rho^{k+1}-e_\rho^k}{\Delta t}\right\|_{\ell^2}
\le L_\theta(\|e_\rho^k\|_{\ell^2}+\|e_u^{k+1}\|_{\ell^2})
+\|\tau_\rho^k\|_{\ell^2},
\]
\[
\left\|\frac{e_u^{k+1}-e_u^k}{\Delta t}\right\|_{\ell^2}
\le L_2(\|e_\rho^k\|_{\ell^2}+\|e_u^k\|_{\ell^2})
+\gamma_k\|e_u^{k+1}\|_{\ell^2}+\|\tau_u^k\|_{\ell^2},
\]
where \(L_2\) depends on \(L_G,L_{\mathcal C,\rho},L_{\mathcal C,u}\). For
\(\Delta t\) sufficiently small such that \(\gamma_k\Delta t\le 1/2\),
the term \(\gamma_k\|e_u^{k+1}\|\) can be absorbed into the left-hand
side.

\textbf{Step 4: Discrete Gronwall.}
The velocity recursion first bounds $\|e_u^{k+1}\|$ in terms of
$\|e_\rho^k\|$, $\|e_u^k\|$, and the local truncation error. Substitution
into the density bound then closes the coupled estimate. Combining the two bounds and setting
\(e^k=\|e_\rho^k\|_{\ell^2}+\|e_u^k\|_{\ell^2}\),
we obtain
\[
e^{k+1}\le(1+C\Delta t)e^k+C\Delta t\|\tau^k\|,
\qquad \|\tau^k\|=O(\Delta t).
\]
Since \(e^0=0\), unrolling the recursion and applying the discrete
Gronwall inequality gives
\(
e^k\le C\Delta t\,e^{Ct_k}.
\)
Therefore, there exists \(C_{\mathrm{temp}}>0\), independent of
\(\Delta t\) and \(k\), such that
\[
\|e_\rho^k\|_{\ell^2}^2+\|e_u^k\|_{\ell^2}^2
\le C_{\mathrm{temp}}(\Delta t)^2 e^{CT}.
\]
This proves the lemma.
\end{proof}

\subsection{Sparse graph error}\label{app:graph error}
We use the sparse graph construction and notation introduced in
Section~\ref{sec:graph_discretization}. The theoretical sparse graph construction performs \(m\)
independent draws with replacement for each node; duplicate selections
are collapsed, and draws that fall in \(\mathrm{KNN}(i)\) are ignored in
the long-range part since the corresponding KNN edge is already included
deterministically. 

The sparse graph has
\[
E=\mathcal O\!\left(N(k_{\mathrm{local}}+m)\right)
\]
edges, giving the same per-step cost for evaluating the nonlocal
interaction. When \(k_{\mathrm{local}}+m\ll N\), this avoids the
\(\mathcal O(N^2)\) cost of a fully connected nonlocal graph.
As $\alpha$ increases, the L{\'e}vy kernel places progressively
less relative weight on long distances. The pure local $\alpha=2$
baseline uses a separate construction with long-range sampling disabled.

The symmetrized kernel \(\tilde\nu_{ij}^{\mathrm{sym}}\) used in the
sparse operator is defined in \eqref{eq:symmetric_kernel}; we now record
its unbiasedness and symmetry, and then define the sparse approximation
error \(\mathcal E_i\) analyzed in this appendix. Throughout this appendix, we work under Assumption~\ref{ass:graph}
(a finite graph sampled once) and Assumption~\ref{ass:small_regime}
(small-sampling regime).

\begin{lemma}[Unbiasedness and symmetry]\label{lem:unbiased}
Let \(\tilde{\nu}_{ij}^{\mathrm{sym}}\) be defined as in
\eqref{eq:symmetric_kernel}. Then

(i) $\mathbb{E}[\tilde{\nu}_{ij}^{\mathrm{sym}}] = \nu_{ij}$ (unbiasedness),

(ii) $\tilde{\nu}_{ij}^{\mathrm{sym}} = \tilde{\nu}_{ji}^{\mathrm{sym}}$ for each realization (pathwise symmetry).

Consequently, $\mathbb{E}[(\mathcal{L}_N^{\mathrm{sparse}}f)_i] = (\mathcal{L}_N^{\mathrm{full}}f)_i$ for each node $i$.
\end{lemma}

\begin{proof}
Since \(\mathbb E[\zeta_{ij}]=Q_{ij}\) and
\(\tilde{\nu}_{ij}^{\mathrm{sym}}=\zeta_{ij}\nu_{ij}/Q_{ij}\),
\(\mathbb E[\tilde{\nu}_{ij}^{\mathrm{sym}}]=\nu_{ij}\). Symmetry follows from
\(\zeta_{ij}=\zeta_{ji}\), \(Q_{ij}=Q_{ji}\), and \(\nu_{ij}=\nu_{ji}\). The
operator statement follows by linearity of expectation.
\end{proof}
For each node $i$, define the sparse approximation error
\(
\mathcal{E}_i := (\mathcal{L}_N^{\mathrm{sparse}}f)_i - (\mathcal{L}_N^{\mathrm{full}}f)_i.
\)
With the unbiasedness from Lemma~\ref{lem:unbiased}, we have $\mathbb{E}[\mathcal{E}_i]=0$.

\begin{lemma}[Variance bound]
\label{lem:variance}
Under Assumption~\ref{ass:graph} and with the notation of Section~\ref{sec:sparse_error} for
\(\mathcal C_i,p_{ij},\eta_{ij},\zeta_{ij},Q_{ij},\tilde\nu_{ij}^{\mathrm{sym}}\),
define
\(
b_{ij}:=\nu_{ij}(f_j-f_i),\quad
Y_{ij}:=\tilde\nu_{ij}^{\mathrm{sym}}(f_j-f_i)=\frac{\zeta_{ij}}{Q_{ij}}b_{ij},
\)
and
\(\mathcal E_i:=(\mathcal L_N^{\mathrm{sparse}}f)_i-(\mathcal L_N^{\mathrm{full}}f)_i\).
Assume
\begin{equation}\label{eq:variance_ass}
\sup_i\sum_{j\in\mathcal C_i}\frac{b_{ij}^2}{p_{ij}+p_{ji}}\le C_{\mathrm{var}}.
\end{equation}
Then
\[
\operatorname{Var}(\mathcal E_i)\le \frac{C_F}{m},
\qquad
C_F:=2C_{\mathrm{var}},
\]
where \(C_F>0\) is independent of \(m\).
\end{lemma}

\begin{proof}
\textbf{(i) Error decomposition.}
KNN edges are deterministic and contribute no sampling variance, so
\(\mathcal E_i=\sum_{j\in\mathcal C_i}(Y_{ij}-b_{ij})\) with
\(Y_{ij}=\zeta_{ij}b_{ij}/Q_{ij}\). Since
\(\zeta_{ij}\sim\mathrm{Bernoulli}(Q_{ij})\),
\(\mathbb E[Y_{ij}]=b_{ij}\) and
\(\operatorname{Var}(Y_{ij})=b_{ij}^2(1-Q_{ij})/Q_{ij}\).

\textbf{(ii) Single-edge variance.}
Since \(1-Q_{ij}=(1-p_{ij})^m(1-p_{ji})^m\le e^{-x_{ij}}\) with
\(x_{ij}=m(p_{ij}+p_{ji})\), we have \(Q_{ij}\ge1-e^{-x_{ij}}\).
Therefore
\[
\frac{1-Q_{ij}}{Q_{ij}}
\le
\frac{e^{-x_{ij}}}{1-e^{-x_{ij}}}
=
\frac{x_{ij}}{e^{x_{ij}}-1}\cdot\frac{1}{x_{ij}}
\le
\frac{1}{x_{ij}}
=
\frac{1}{m(p_{ij}+p_{ji})},
\]
where we used \(x/(e^x-1)\le1\) for \(x>0\). Hence, by \eqref{eq:variance_ass},
\begin{equation}
\sum_{j\in\mathcal C_i}\operatorname{Var}(Y_{ij})
\le
\frac{1}{m}\sum_{j\in\mathcal C_i}\frac{b_{ij}^2}{p_{ij}+p_{ji}}
\le
\frac{C_{\mathrm{var}}}{m}.
\label{eq:variance_bound_Y}
\end{equation}

\textbf{(iii) Negative association.}
For fixed \(i\), the indicators \(\{\eta_{ij}\}_{j\in\mathcal C_i}\)
generated by the i.i.d.\ draws \(J_{i,1},\dots,J_{i,m}\) are negatively
associated \cite{Joag-Dev_NA}; the reverse indicators
\(\{\eta_{ji}\}_{j\in\mathcal C_i}\) come from independent pools of distinct nodes, and independence implies negative association; hence
\(\{\eta_{ij},\eta_{ji}:j\in\mathcal C_i\}\) is negatively associated.
Since \(\zeta_{ij}=\eta_{ij}\vee\eta_{ji}\) is coordinatewise
nondecreasing, closure of negative association under monotone
transformations gives that \(\{\zeta_{ij}\}_{j\in\mathcal C_i}\) are
negatively associated \cite{2015Novel}.

\textbf{(iv) Sign splitting.}
Let \(\mathcal C_i^\pm=\{j:b_{ij}\gtrless0\}\) and
\(\mathcal E_i=\mathcal E_i^++\mathcal E_i^-\) with
\(\mathcal E_i^\pm=\sum_{j\in\mathcal C_i^\pm}(Y_{ij}-b_{ij})\).
Within each sign class the coefficients \(b_{ij}\) share the same sign, so
all covariance terms are nonpositive and
\(\operatorname{Var}(\mathcal E_i^\pm)\le\sum_{j\in\mathcal C_i^\pm}\operatorname{Var}(Y_{ij})\).
Using \(\operatorname{Var}(X+Y)\le2\operatorname{Var}(X)+2\operatorname{Var}(Y)\),
\(
\operatorname{Var}(\mathcal E_i)
\le2\operatorname{Var}(\mathcal E_i^+)+2\operatorname{Var}(\mathcal E_i^-)
\le2\sum_{j\in\mathcal C_i}\operatorname{Var}(Y_{ij}).
\)

\textbf{(v)}
Combining with \eqref{eq:variance_bound_Y},
\(
\operatorname{Var}(\mathcal E_i)\le\frac{2C_{\mathrm{var}}}{m}=\frac{C_F}{m},
\)
which proves the claim.
\end{proof}

\begin{lemma}[Bernstein]
\label{lem:bernstein}
Under the assumptions of Lemma~\ref{lem:variance},
let \(M_{0}:=\max_{j\in\mathcal C_i}|b_{ij}|/Q_{ij}\), a deterministic bound on the support of $|X_{ij}|$. Then for every \(t>0\),
\begin{equation}
\mathbb P\bigl(|\mathcal E_i|\ge t\bigr)
\le
4\exp\Big(-\frac{t^2}{8(C_F/m+M_{0}t/6)}\Big).
\label{eq:bernstein_general}
\end{equation}
If in addition Assumption~\ref{ass:small_regime} holds, then
\begin{equation}
M_{0}\le B_f/m,
\qquad
B_f:=\sup_{i,j\in\mathcal C_i}\frac{|b_{ij}|}{c(p_{ij}+p_{ji})}<\infty,
\label{eq:bernstein_increment}
\end{equation}
and for every fixed \(\hat{c_2}>0\) independent of \(m\), setting
\(C_*:=C_F+B_f\hat{c_2}/6\) and restricting \(0<t\le \hat{c_2}\),
\begin{equation}
\mathbb P\bigl(|\mathcal E_i|\ge t\bigr)\le4\exp(-\hat{c_1}mt^2),
\qquad \hat{c_1}:=\frac{1}{8C_*}.
\label{eq:bernstein_subgaussian}
\end{equation}
\end{lemma}

\begin{proof}
Fix \(i\) and set \(X_{ij}:=Y_{ij}-b_{ij}\), so
\(\mathcal E_i=\sum_{j\in\mathcal C_i}X_{ij}\). By Lemma~\ref{lem:variance},
\(\sum_{j\in\mathcal C_i}\operatorname{Var}(X_{ij})\le C_F/m\).

\textbf{(a) Boundedness.}
For the fixed finite grid,
\(Q_{\min}:=\min_{i,j\in\mathcal C_i}Q_{ij}>0\). Since
\(b_{ij}=\nu_{ij}(f_j-f_i)\) is bounded,
\[
|X_{ij}|
=
|b_{ij}|\left|\frac{\zeta_{ij}}{Q_{ij}}-1\right|
\le
\frac{|b_{ij}|}{Q_{ij}}
\le
M_0
:=
\max_{\ell\in\mathcal C_i}
\frac{|b_{i\ell}|}{Q_{i\ell}}
<\infty.
\]

\textbf{(b) Negative association.} Splitting
\(\mathcal C_i^\pm=\{j:b_{ij}\gtrless0\}\), the coefficients \(b_{ij}\)
within each sign class have the same sign, so \(X_{ij}=Y_{ij}-b_{ij}\)
is a monotone function of \(\zeta_{ij}\) with the same monotonicity
within each class. Since \(\{\zeta_{ij}\}_{j\in\mathcal C_i}\) are
negatively associated (Lemma~\ref{lem:variance}(iii)), the closure of
negative association under monotone transformations gives that each
subfamily \(\{X_{ij}\}_{j\in\mathcal C_i^\pm}\) is negatively associated.

\textbf{(c) Small-sampling regime.}
Under Assumption~\ref{ass:small_regime},
\(Q_{ij}\ge c\,m(p_{ij}+p_{ji})\), where $c=1-1/ e $. Define
\(
p_{\min}:=\min_{i,j\in\mathcal C_i}p_{ij}>0,
\)
which is positive on the fixed finite grid. Since
\(b_{ij}=\nu_{ij}(f_j-f_i)\) is bounded and
\(p_{ij}+p_{ji}\ge2p_{\min}\), the constant
\(
B_f:=\sup_{i,j\in\mathcal C_i}\frac{|b_{ij}|}{c(p_{ij}+p_{ji})}<\infty
\)
is finite, and
\[
|X_{ij}|\le\frac{|b_{ij}|}{Q_{ij}}\le\frac{|b_{ij}|}{c\,m(p_{ij}+p_{ji})}\le\frac{B_f}{m},
\]
so \(M_{0}\le B_f/m\). If in addition \(\|f\|_\infty\le C_f\) and
\(\nu_{ij}\le\nu_{\max}\), then \(|b_{ij}|\le2C_f\nu_{\max}\) and hence
\(
B_f\le\frac{C_f\nu_{\max}}{c\,p_{\min}}.
\)

\textbf{(d) Bernstein.} Applying the Bernstein inequality for negatively
associated bounded random variables \cite{2019Optimal} to each sign class
with threshold \(t/2\), and using
\[
\sum_{j\in\mathcal C_i^\pm}\operatorname{Var}(X_{ij})
\le
\sum_{j\in\mathcal C_i}\operatorname{Var}(X_{ij})
\le
\frac{C_F}{m},
\qquad
M_{0}\le\frac{B_f}{m},
\]
we obtain
\(
\mathbb P\!\left(\Big|\sum_{j\in\mathcal C_i^\pm}X_{ij}\Big|\ge\frac t2\right)
\le
2\exp\!\left(-\frac{t^2/4}{2\left(C_F/m+B_ft/(6m)\right)}\right).
\)
Since \(\mathcal E_i=\sum_{j\in\mathcal C_i^+}X_{ij}+\sum_{j\in\mathcal C_i^-}X_{ij}\),
a union bound gives \eqref{eq:bernstein_general}.

\textbf{(e) Sub-Gaussian form.} Take any \(\hat{c_2}>0\) independent of \(m\)
and set \(C_*:=C_F+B_f\hat{c_2}/6\). Then \(C_F+B_ft/6\le C_*\) for
\(0<t\le\hat{c_2}\), hence
\(\mathbb P(|\mathcal E_i|\ge t)\le4\exp(-mt^2/(8C_*))\), i.e.\
\eqref{eq:bernstein_subgaussian} with \(\hat{c_1}=1/(8C_*)\).
\end{proof}

\begin{lemma}[Uniform high-probability sparse-graph error]
\label{lem:uniform_error}
Under the assumptions of Lemma~\ref{lem:bernstein}, for every
\(\delta\in(0,1)\), with probability at least \(1-\delta\),
\begin{equation}
\max_{1\le i\le N}|\mathcal E_i|
\le C_{\mathrm{graph}}\sqrt{\frac{\log(4N/\delta)}{m}},
\label{eq:uniform_sparse_error}
\end{equation}
provided
\begin{equation}
\sqrt{\frac{\log(4N/\delta)}{\hat c_1m}}\le \hat c_2,
\label{eq:uniform_sparse_condition}
\end{equation}
where \(\hat c_1,\hat c_2\) are as in Lemma~\ref{lem:bernstein} and
\(C_{\mathrm{graph}}>0\) is independent of \(m\), but may depend on
\(p_{\min}\), \(C_f\), \(\nu_{\max}\), and the fixed-grid parameters.
Consequently,
\begin{equation}
\bigl\|\mathcal L_N^{\mathrm{sparse}}f-\mathcal L_N^{\mathrm{full}}f\bigr\|_{\ell^\infty}
\le
C_{\mathrm{graph}}\sqrt{\frac{\log(4N/\delta)}{m}}.
\label{eq:operator_uniform_error}
\end{equation}
\end{lemma}

\begin{proof}
By Lemma~\ref{lem:bernstein}, for \(0<t\le\hat c_2\) and every \(i\),
\(\mathbb P(|\mathcal E_i|\ge t)\le4\exp(-\hat c_1mt^2)\). A union bound
over \(i=1,\dots,N\) gives
\(\mathbb P(\max_i|\mathcal E_i|\ge t)\le4N\exp(-\hat c_1mt^2)\).
Choosing \(t=\sqrt{\log(4N/\delta)/(\hat c_1m)}\), which satisfies
\eqref{eq:uniform_sparse_condition}, makes the right-hand side equal to
\(\delta\), yielding \eqref{eq:uniform_sparse_error} with
\(C_{\mathrm{graph}}=\hat c_1^{-1/2}\). Since
\(\|\mathcal L_N^{\mathrm{sparse}}f-\mathcal L_N^{\mathrm{full}}f\|_{\ell^\infty}
=\max_i|\mathcal E_i|\), this also yields
\eqref{eq:operator_uniform_error}.
\end{proof}

\begin{proof}[Proof of Proposition~\ref{prop:sparse_full}]
It follows directly from Lemma~\ref{lem:uniform_error},
which provides the stated probabilistic bound with the same \(\delta\).
The uniform bound on the pointwise error directly controls the \(\ell^\infty\)-norm of the error:
\(
\|\mathcal L_N^{\mathrm{sparse}}f - \mathcal L_N^{\mathrm{full}}f\|_{\ell^\infty} = \max_i |\mathcal{E}_i| \le C_{\mathrm{graph}}\sqrt{\frac{\log(4N/\delta)}{m}}.
\)
\end{proof}
\begin{proof}[Proof of Lemma~\ref{lem:graph_forcing}]
	Fix $k$. The node sums in the density and convection components
		have coefficients
		\[
	a_{ij}^{\rho,k}
	=\theta_{ij}(\rho_F^k)u_{F,ij}^k,
	\qquad
	a_{ij}^{C,k}
	=\partial_1\theta_{ij}(\rho_F^k)(u_{F,ij}^k)^2,
	\]
	and, for the intermediate density update,
	\[
	a_{ij}^{\rho,+,k}
	=\theta_{ij}(\rho_F^k)u_{F,ij}^{k+1}.
	\]
Their magnitudes are bounded by $C_\theta U_{\max}$ and
		$C_{\partial\theta}U_{\max}^2$, respectively.
		Using the positive $p_{\min}$ of Proposition~\ref{prop:sparse_full}
		and $p_{ij}+p_{ji}\ge p_{\min}$,
		the existing fixed-grid bounds give, for either family,
		\[
		\sup_i\sum_{j\in\mathcal C_i}
		\frac{|a_{ij}^{k}|^2\nu_{ij}^2}{p_{ij}+p_{ji}}
		\le
		\frac{(N-1)\nu_{\max}^2}{p_{\min}}
		\max\{C_\theta^2U_{\max}^2,
		C_{\partial\theta}^2U_{\max}^4\}.
		\]
		The variance constants depend on $N,h,\alpha,p_{\min}$ and
		the state bounds.
	
	For $X_{ij}^k=a_{ij}^k(\tilde\nu_{ij}^{\rm sym}-\nu_{ij})$,
		unbiasedness, the sign splitting used in Lemma~\ref{lem:variance},
		and Lemma~\ref{lem:bernstein} give
		$\mathbb P(|\sum_jX_{ij}^k|>t)\le4e^{-cmt^2}$
		in the same admissible range. The full trajectory is independent of
		the sparse graph draws. Applying the Bernstein bound to the three coefficient families and taking a union bound over \(N\) nodes and \(K+1\) time levels gives \(C_0N(K+1)e^{-cmt^2}\), $C_0:=12$, which yields the stated estimate. Convection is a difference of two
		node sums. Conversion to the fixed full-graph norm $\mathcal X$
		changes only $C_{\rm force}$.
\end{proof}

\begin{proof}[Proof of Theorem~\ref{thm:total_error}]
	Introduce the full-graph discrete solution $\rho_F^k$. Then
	\[
	\|\rho_h^k-\pi\|_{\ell^2(\pi)}
	\le \|\rho_S^k-\rho_F^k\|_{\ell^2(\pi)}
	+\|\rho_F^k-\rho^F(t_k)\|_{\ell^2(\pi)}
	+\|\rho^F(t_k)-\pi\|_{\ell^2(\pi)}.
	\]
	For the first term, subtract the two velocity updates and then the
		density updates in \eqref{eq:num_scheme}. The Lipschitz bounds and
		the forcing estimates of Lemma~\ref{lem:graph_forcing}, including
		the intermediate density stage, yield
		\[
		\|e^{k+1}\|_{\mathcal X}
		\le(1+C\Delta t)\|e^k\|_{\mathcal X}
		+C\Delta t\,C_{\rm force}
		\sqrt{\frac{\log(C_0N(K+1)/\delta_{\rm p})}{m}}.
		\]
		Here $e^k$ is the difference in relative density and full-edge
		velocity. The sparse weights have a deterministic fixed-grid bound:
		for $m\ge1$, $Q_{ij}\ge\max(p_{ij},p_{ji})\ge p_{\min}$,
		so $\tilde\nu_{ij}^{\rm sym}\le\nu_{\max}/p_{\min}$ on sampled
		edges. Hence the Lipschitz constant can be chosen independently
		of $m$ and the realization. Since $e^0=0$, discrete Gronwall gives
		the graph term, with $C_{\rm graph,T}$ bounded by a fixed multiple
		of $T e^{CT}C_{\rm force}$.
	
	Lemma~\ref{lem:temporal_error} on the full graph bounds the second
		term by $C_{\rm temp}e^{CT}\Delta t$.
	Theorem~\ref{thm:local_exp} bounds the third term by
	\begin{equation}\label{app:full_graph_decay}
		\|\rho^F(t_k)-\pi\|_{\ell^2(\pi)}
			\le C_{\rm exp}(1+t_k)e^{-r_1(\gamma)t_k},
	\end{equation}
	where
	\(	C_{\rm exp}=C_{\mathcal L}\pi_{\max}
	\bigl(\pi_{\min}^{-1}\|\rho(0)-\pi\|_{\ell^2(\pi)}
	+\|u(0)\|_{\mathcal U}\bigr).
	\)
	Combining these estimates proves \eqref{eq:total_error}.
\end{proof}
\begin{proof}[Proof of Corollary~\ref{cor:wnl_error}]
	By Lemma~\ref{lem:wnl_l2} and
	\(\|\rho_h^k-\pi\|_{\pi^{-1}}
	=\|\varphi\|_{\ell^2(\pi)}
	\le\pi_{\min}^{-1}\|\rho_h^k-\pi\|_{\ell^2(\pi)}\),
	\[
	W_{\mathrm{NL}}(\rho_h^k,\pi)
	\le
	(1+\eta)\lambda_1^{-1/2}\|\rho_h^k-\pi\|_{\pi^{-1}}
	\le
	\frac{1+\eta}{\pi_{\min}\sqrt{\lambda_1}}
	\|\rho_h^k-\pi\|_{\ell^2(\pi)}.
	\]
	Applying Theorem~\ref{thm:total_error} gives the stated bound.
\end{proof}
\section{Numerical Implementation Details}
\label{app:numerical_details}

\paragraph{Synthetic targets.}
Let \(z=(x,y)\) and \(r=\sqrt{x^2+y^2}\).
The synthetic experiments use the following unnormalized target densities:
\begin{align}
	q_1(x,y)
	&=
	\exp\!\left[-2(r-3)^2\right],\\
	q_2(x,y)
	&=
	\exp\!\left[-2(r-3)^2\right]
	\left[
	\exp\!\left(-2(x-3)^2\right)
	+
	\exp\!\left(-2(x+3)^2\right)
	\right],\\
	q_3(z)
	&=
	\frac14\sum_{\ell=1}^{4}
	\exp\!\left[
	-\frac{\|z-\mu_\ell\|^2}{2\sigma^2}
	\right],
	\qquad \sigma=0.5,
\end{align}
where
\[
\mu_1=(0,2),\qquad
\mu_2=(2,0),\qquad
\mu_3=(0,-2),\qquad
\mu_4=(-2,0).
\]
The second target is the
banana-shaped benchmark. All target densities are evaluated on the
corresponding computational grid and normalized to unit discrete mass.
\paragraph{Numerical settings.}
Unless otherwise stated, the initial density is the standard Gaussian
\(\mathcal N(0,I_2)\).
For the fixed-step experiments, we use adaptive damping unless
a constant value is prescribed, as in the damping-sweep experiments.
Long-range edges are generated node by node using probabilities proportional to \(r^{-(2+\alpha)}\), with existing edges excluded from subsequent candidate sets. For each retained long-range edge \(e=\{i,j\}\), we assign the symmetric weight
$
\nu_e^{\mathrm{num}}
=
\frac{h^2 r_e^{-(2+\alpha)}}{m\,p_e},
$
where \(r_e=\lVert x_i-x_j\rVert\) and \(p_e\) is the normalized proposal probability used by the node that first generates \(e\), computed over its remaining candidate set. Repeated selections are merged. Local KNN edges retain the base weights \(h^2r_e^{-(2+\alpha)}\). The unbiasedness and sampling-error estimates in Section~\ref{sec:sparse_error} concern the sampling scheme with exact inclusion weights defined there.

The equal-wall-clock benchmark is performed on the four-Gaussian-mixture
target using a \(40\)-s computational budget for each method and each
random seed. All runs use the same Intel Core i5-12500 CPU with
NumPy/SciPy linked against Intel MKL, and no GPU acceleration is used.
Five independent seeds are reported.

For the proposed method, the measured algorithmic time includes target
setup, sparse-graph construction, initialization of the graph variables,
and subsequent density evolution. The method uses a \(50\times50\) grid, \(\alpha=1\),
a KNN query size of \(10\), \(m=3\),
\(\Delta t=5\times10^{-4}\), and
\(
\gamma^*=2\sqrt{\lambda_1},
\)
with \(\lambda_1\) computed separately for each sparse-graph realization.

	For the particle baselines, the measured time consists of a \(4\)-s
	warm-up phase followed by a \(36\)-s production phase. Adaptation is
	restricted to the warm-up period. The MCMC baseline is implemented as
	random-walk Metropolis--Hastings (MH), with a frozen proposal scale of
	\(3.815\); MALA uses a frozen step size of \(0.931\). ULD is simulated
	using BAOAB with time step \(0.01\), damping \(\gamma=1\), mass \(1\),
	and inverse temperature \(\beta=1\).

For all methods, KDE construction, KL evaluation, mode analysis,
plotting, and file I/O are excluded from the reported algorithmic time.

\bibliography{refs} 

\end{document}